\documentclass{amsart}

\usepackage[utf8]{inputenc}
\usepackage{verbatim}
\usepackage{graphicx}
\usepackage{graphicx,caption2,psfrag,float,color}
\usepackage{amssymb}
\usepackage{amscd}
\usepackage{amsmath}

\newtheorem{theorem}{Theorem}[section]

\newtheorem{lemma}[theorem]{Lemma}
\newtheorem{proposition}[theorem]{Proposition}
\numberwithin{equation}{subsection}
\newtheorem{definition}[theorem]{Definition}

\newcommand{\A}{\alpha}

\renewcommand{\i}{\mathbf i}

\title{
Generalizations of a cotangent sum associated to the Estermann zeta function}
\author{Helmut Maier and Michael Th. Rassias}
\date{\today}
\address{Department of Mathematics, University of Ulm, Helmholtzstrasse 18, 8901 Ulm, Germany.}
\email{helmut.maier@uni-ulm.de}
\address{Department of Mathematics, ETH-Z\"{u}rich, R\"{a}mistrasse 101, 8092 Z\"{u}rich, Switzerland.}
\email{michail.rassias@math.ethz.ch, michailr@princeton.edu}
\thanks{}

\begin{document}

 \maketitle
 
\begin{abstract} Cotangent sums are associated to the zeros of the Estermann zeta function. They have also proven to be of importance in the Nyman-Beurling criterion for the Riemann Hypothesis.\\
The main result of the paper is the proof of the existence of a unique positive measure $\mu$ on $\mathbb{R}$, with respect to which certain normalized cotangent sums are equidistributed.\\ 
Improvements as well as further generalizations of asymptotic formulas regarding the relevant cotangent sums are obtained. We also prove an asymptotic formula for a more general cotangent sum as well as asymptotic results for the moments of the cotangent sums under consideration. We also give an estimate for the rate of growth of the moments of order $2k$, as a function of $k$.
\newline

\end{abstract}

\section{Introduction}
Cotangent sums are associated to the zeros of the Estermann zeta function. R. Balasubramanian, J. B. Conrey and D. R. Heath-Brown \cite{bala}, used properties of the Estermann zeta function to prove asymptotic formulas for mean-values of the product consisting of the Riemann zeta function and a Dirichlet polynomial. Period functions and families of cotangent sums appear in recent work of S. Bettin and J. B. Conrey (cf. \cite{BEC}). They generalize the Dedekind sum and share with it the property of satisfying a reciprocity formula. They prove a reciprocity formula for the V. I. Vasyunin's sum \cite{VAS}, which appears in the Nyman-Beurling criterion for the Riemann Hypothesis.\\
In the present paper, improvements as well as further generalizations of asymptotic formulas regarding the relevant cotangent sums are obtained. We also prove an asymptotic formula for a more general cotangent sum as well as asymptotic results and upper bounds for the moments of the cotangent sums under consideration. Furthermore, we obtain detailed information about the distribution of the values of these cotangent sums. We also give an estimate for the rate of growth of the moments of order $2k$, as a function of $k$.\\ \\
%
%
\subsection{The cotangent sum and its applications}
%
%
%
%
\noindent The present paper is focused in the study of the following cotangent sum\index{cotangent sum}:
\begin{definition}$\label{x:thedef}$
$$c_0\left(\frac{r}{b}\right):=-\sum_{m=1}^{b-1}\frac{m}{b}\cot\left(\frac{\pi mr}{b}\right)\:,$$
where $r$, $b\in\mathbb{N}$, $b\geq 2$, $1\leq r\leq b$ and $(r,b)=1$.
\end{definition}
\noindent The function $c_0(r/b)$ is odd and periodic of period 1 and its value is an algebraic number. Its properties of being odd and periodic are depicted in the following graphs:
\begin{figure}[h!]
  \centering%
      \includegraphics[scale=0.51]{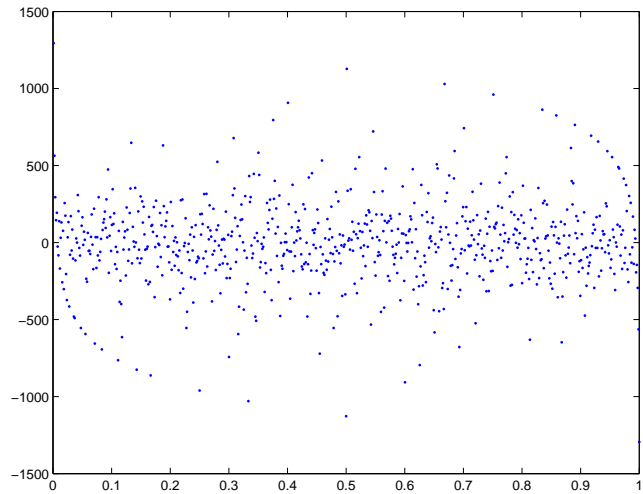}
  \caption{{Graph of $c_0(r/b)$, for $1\leq r \leq b,\ b=757 $, with $(r,b)=1$.}}
\end{figure}
\begin{figure}[h!]
  \centering%
      \includegraphics[scale=0.51]{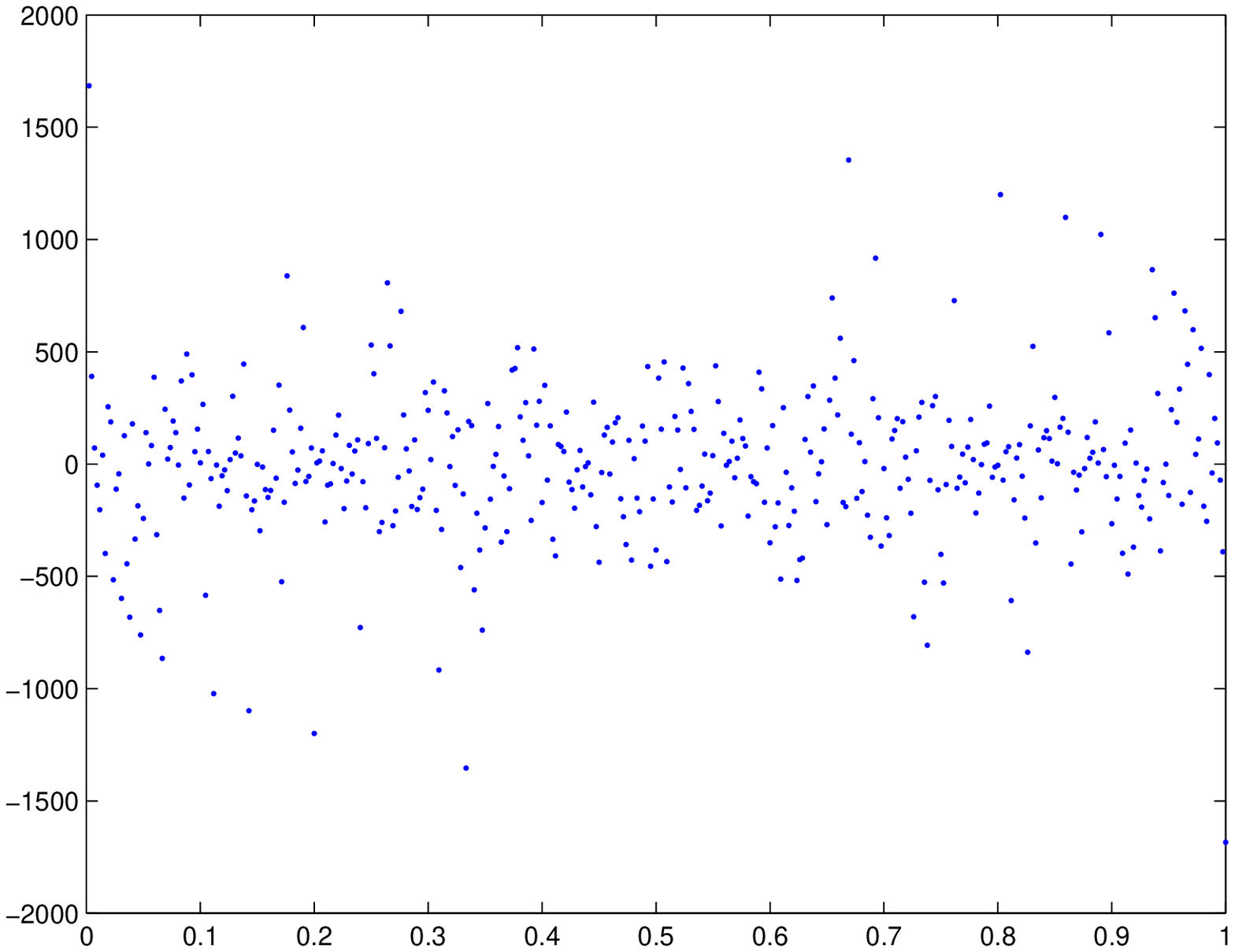}
  \caption{{Graph of $c_0(r/b)$, for $1\leq r \leq b,\ b=946 $, with $(r,b)=1$.}}
\end{figure}
\newpage
\begin{figure}[h!]
  \centering%
      \includegraphics[scale=0.51]{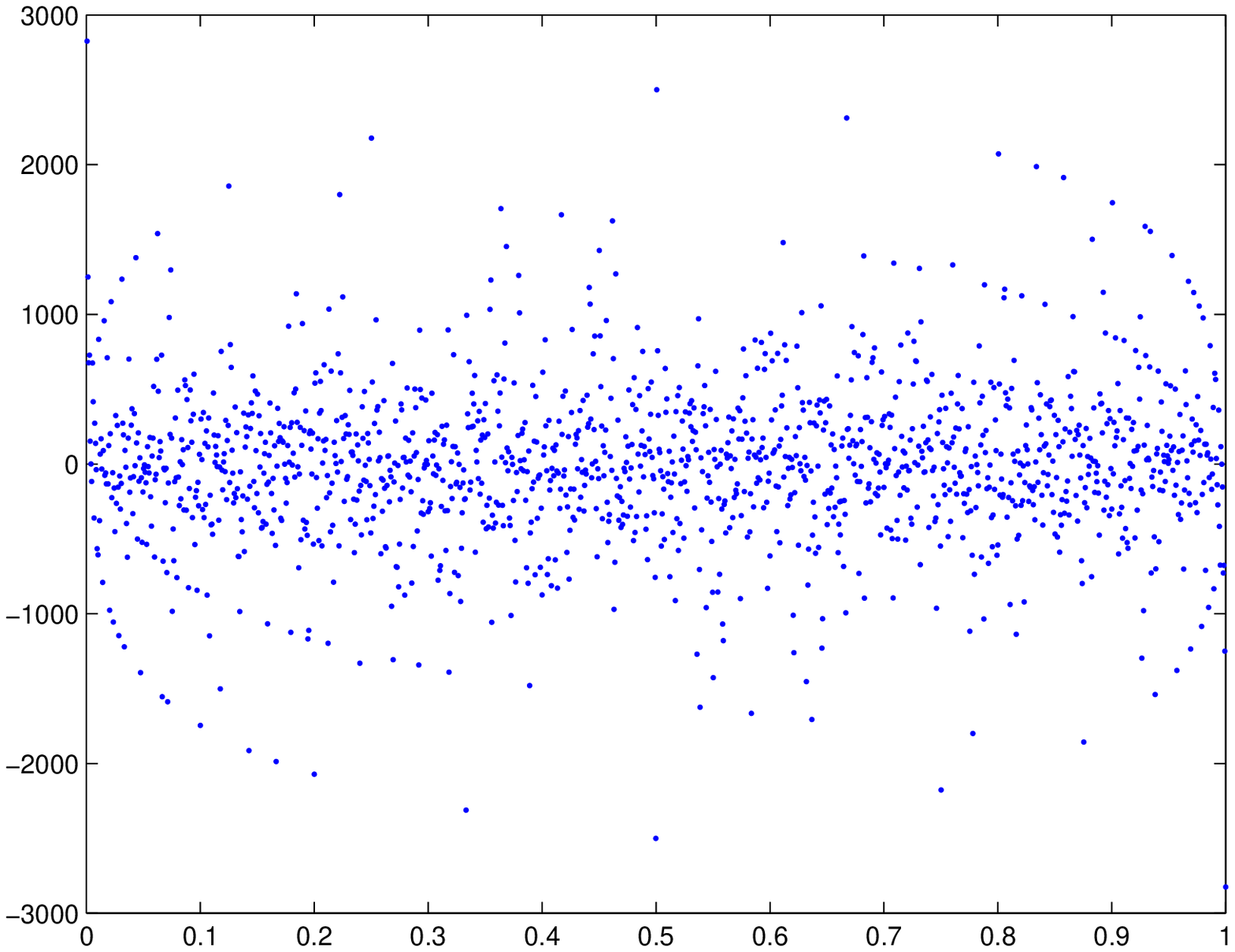}
  \caption{{Graph of $c_0(r/b)$, for $1\leq r \leq b,\ b=1471 $, with $(r,b)=1$.}}
\end{figure}
\begin{figure}[h!]
  \centering%
      \includegraphics[scale=0.51]{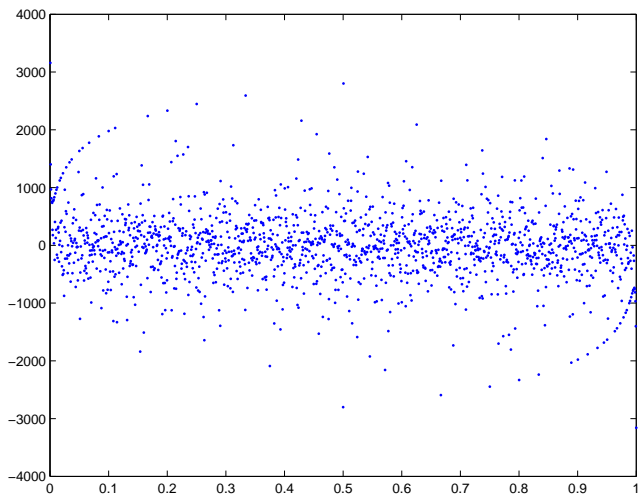}
  \caption{{Graph of $c_0(r/b)$, for $1\leq r \leq b,\ b=1619 $, with $(r,b)=1$.}}
\end{figure}
It is interesting to mention that for hundreds of integer values of $k$ for which we have examined the graph of $c_0(r/b)$ by the use of MATLAB, the resulting figure always has a shape similar to an ellipse.
\newpage
Part of our goal is to understand this phenomenon, and we will do it to some extent. The main result in this respect is contained in Theorem \ref{x:introprob}, which provides information about equidistribution and moments of these sums.\\

Before presenting the main results of the paper regarding this cotangent sum, we shall demonstrate its significance by exhibiting its relation 
to other important functions in number theory, such as the Estermann and the Riemann zeta functions\index{function!Estermann zeta}, and its connections to major open problems in Mathematics, such as the Riemann Hypothesis\index{Riemann Hypothesis}.
\begin{definition}\label{x:defester}
The Estermann zeta function $E\left(s,\frac{r}{b},\alpha\right)$ is defined by the Dirichlet series
$$E\left(s,\frac{r}{b},\alpha\right)=\sum_{n\geq 1}\frac{\sigma_{\alpha}(n)\exp\left(2\pi inr/b\right)}{n^s}\:, $$
where $Re\:s>Re\:\alpha+1$, $b\geq 1$, $(r,b)=1$ and\index{function!generalized divisor}
$$\sigma_{\alpha}(n)=\sum_{d|n}d^{\alpha}\:.$$
\end{definition}
It is worth mentioning that T. Estermann\index{Estermann, T.} (see \cite{EST}) introduced and studied the above function in the special case when $\alpha=0$. Much later, it was studied by I. Kiuchi\index{Kiuchi, I.} (see \cite{kiu}) for $\alpha\in(-1,0]$.\\
The Estermann zeta function can be continued analytically to a meromorphic function, on the whole complex plane up to two simple poles $s=1$ and $s=1+\alpha$ if $\alpha\neq0$ or a double pole at $s=1$ if $\alpha=0$ (see \cite{EST}, \cite{ishi2}, \cite{sle}).\\
Moreover, it satisfies the functional equation\index{function!Estermann zeta!functional equation of}:
\begin{align*}
E\left(s,\frac{r}{b},\alpha\right)&=\frac{1}{\pi}\left(\frac{b}{2\pi} \right)^{1+\alpha-2s}\Gamma(1-s)\Gamma(1+\alpha-s)\\
&\times \left(\cos\left( \frac{\pi \alpha}{2}\right)E\left(1+\alpha-s,\frac{\bar{r}}{b},\alpha\right)- \cos\left(\pi s- \frac{\pi \alpha}{2}\right)E\left(1+\alpha-s,-\frac{\bar{r}}{b},\alpha\right)\right),
\end{align*}
where $\bar{r}$ is such that $\bar{r} r\equiv1\:(\bmod\; b)$ and $\Gamma(s)$ stands for the Gamma function.\\
\noindent R. Balasubramanian\index{Balasubramanian, R.}, J. B. Conrey\index{Conrey, J. B.} and D. R. Heath-Brown\index{Heath-Brown, D. R.} \cite{bala}, used properties of $E\left(0,\frac{r}{b},0\right)$ to prove an asymptotic formula for
$$I=\int_0^T\left|\zeta\left(\frac{1}{2}+it\right)\right|^2\left|A\left(\frac{1}{2}+it\right)\right|^2dt\:,$$
where $A(s)$ is a Dirichlet polynomial.\\
Asymptotics for functions of the form of $I$ are useful for theorems which provide a lower bound for the portion of zeros of the Riemann zeta-function $\zeta(s)$\index{function!Riemann zeta!zeros of on the critical line} on the critical line (see \cite{IWC}, \cite{IKW}). \\

M. Ishibashi\index{Ishibashi, M.} (see \cite{ISH}) presented a nice result concerning the value of $E\left(s,\frac{r}{b},\alpha\right)$ at $s=0$.
\begin{theorem}(Ishibashi)\index{theorem!Ishibashi}
Let $b\geq 2$, $1\leq r\leq b$, $(r,b)=1$, $\alpha\in\mathbb{N}\cup\{0 \}$. Then\\
(1) For even $\alpha$, it holds
$$	E\left(0,\frac{r}{b},\alpha\right)=\left(-\frac{i}{2}\right)^{\alpha+1}\sum_{m=1}^{b-1}\frac{m}{b}\cot^{(\alpha)}\left(\frac{\pi mr}{b}\right)+\frac{1}{4}\delta_{\alpha,0}\:,$$
where $\delta_{\alpha,0}$ is the Kronecker delta function\index{function!Kronecker delta}. \\
(2) For odd $\alpha$, it holds
$$E\left(0,\frac{r}{b},\alpha\right)=\frac{B_{\A+1}}{2(\A+1)}\:.$$
In the special case when $r=b=1$, we have
$$E\left(0,1,\alpha\right)=\frac{(-1)^{\A+1}B_{\A+1}}{2(\A+1)}\:,$$
where by $B_m$ we denote the $m$-th Bernoulli number\index{Bernoulli number},
where $B_{2m+1}=0$,
$$B_{2m}=2\frac{(2m)!}{(2\pi)^{2m}}\sum_{\nu\geq1}\nu^{-2m}.$$
\end{theorem}
\noindent Hence for $b\geq 2$, $1\leq r\leq b$, $(r,b)=1$, it follows that
$$E\left(0,\frac{r}{b},0\right)=\frac{1}{4}+\frac{i}{2}c_0\left(\frac{r}{b}\right)\:,$$
where $c_0(r/b)$ is the cotangent sum (see Definition $\ref{x:thedef}$).\\
This result gives a connection between the cotangent sum $c_0(r/b)$ and the Estermann zeta function.\\
\noindent Period functions and families of cotangent sums appear in recent work of S. Bettin\index{Bettin, S.} and J. B. Conrey \cite{BEC}, generalizing the Dedekind sums\index{Dedekind sum} and sharing with it the property of satisfying a reciprocity formula\index{cotangent sum!reciprocity formula}. Bettin and Conrey proved the following reciprocity formula for $c_0(r/b)$:\index{formula!reciprocity}
$$c_0\left(\frac{r}{b}\right)+\frac{b}{r}c_0\left(\frac{b}{r}\right)-\frac{1}{\pi r}=\frac{i}{2}\psi_0\left(\frac{r}{b}\right)\:,$$
where
$$\psi_0(z)=-2\:\frac{\log 2\pi z -\gamma}{\pi i z}-\frac{2}{\pi }\int_{\left(\frac{1}{2}\right)}\frac{\zeta(s)\zeta(1-s)}{\sin \pi s}\:z^{-s}\:ds\:,$$
and $\gamma$ stands for the Euler-Mascheroni constant.\\
\noindent This reciprocity formula demonstrates that $c_0(r/b)$ can be interpreted as an ``imperfect'' quantum modular form\index{quantum modular form} of weight 1, in the sense of D. Zagier\index{Zagier, D.}
(see \cite{BETT2}, \cite{ZAG}).\\ \\
The cotangent sum $c_0(r/b)$ can be associated to the study of the Riemann Hypothesis\index{Riemann Hypothesis}, also through its relation with the so-called Vasyunin sum\index{Vasyunin sum}.
The Vasyunin sum is defined as follows:
$$V\left(\frac{r}{b}\right):=\sum_{m=1}^{b-1}\left\{\frac{mr}{b}\right\}\cot\left(\frac{\pi mr}{b}\right)\:,$$
where $\{u\}=u-\lfloor u \rfloor$, $u\in\mathbb{R}.$\\
It can be shown (see \cite{BETT2}, \cite{BEC}) that 
$$V\left(\frac{r}{b}\right)=-c_0\left(\frac{\bar{r}}{b}\right),$$
where, as mentioned previously, $\bar{r}$ is such that $\bar{r} r\equiv1\:(\bmod\; b)$.\\
The Vasyunin sum is itself associated to the study of the Riemann hypothesis through the following identity (see \cite{BETT2}, \cite{BEC}):
\begin{align*}
\frac{1}{2\pi(rb)^{1/2}}\int_{-\infty}^{+\infty}\left|\zeta\left(\frac{1}{2}+it\right)\right|^2\left(\frac{r}{b}\right)^{it}\frac{dt}{\frac{1}{4}+t^2}&=\frac{\log 2\pi -\gamma}{2}\left(\frac{1}{r}+\frac{1}{b} \right)\tag{1}\\
&+\frac{b-r}{2rb}\log\frac{r}{b}-\frac{\pi}{2rb}\left(V\left(\frac{r}{b}\right)+V\left(\frac{b}{r}\right)\right).
\end{align*}
Note that the only non-explicit function in the right hand side of (1) is the Vasyunin sum.\\
\noindent The above formula is related to the Nyman-Beurling-Ba\'ez-Duarte-Vasyunin approach to the Riemann Hypothesis\index{Riemann Hypothesis!Nyman-Beurling-Ba\'ez-Duarte-Vasyunin approach} (see \cite{bag}, \cite{BETT2}).
According to this approach, the Riemann Hypothesis is true if and only if \index{Riemann Hypothesis!equivalence}
$$\lim_{N\rightarrow+\infty}d_N=0,$$
where
$$d_N^{2}=\inf_{D_N}\frac{1}{2\pi}\int_{-\infty}^{+\infty}\left|1-\zeta\left(\frac{1}{2}+it\right) D_N\left(\frac{1}{2}+it\right)\right|^2\frac{dt}{\frac{1}{4}+t^2}$$
and the infimum is taken over all Dirichlet polynomials\index{Dirichlet polynomials}
$$D_N(s)=\sum_{n=1}^{N}\frac{a_n}{n^s}.$$
Hence, from the above arguments it follows that from the behavior of $c_0(r/b)$, we understand the behavior of $V(r/b)$ and thus from (1) we 
may hope to obtain crucial information related to the Nyman-Beurling-Ba\'ez-Duarte-Vasyunin approach to the Riemann Hypothesis.\\

Therefore, to sum up, one can see from all the above that the cotangent sum $c_0(r/b)$ is strongly related to important functions of Number Theory and its
properties can be applied in the study of significant open problems, such as Riemann's Hypothesis.
%
%
%
%
%
%
%
%
%
\subsection{Main result}
\vspace{5mm}
We now come to the main result of the paper, which states the equidistribution of certain normalized cotangent sums with respect to a positive
measure, which is also constructed in the following theorem.
\begin{definition}
For $z\in\mathbb{R}$, let 
$$F(z)=\text{meas}\{\alpha\in[0,1]\::\: g(\alpha)\leq z  \},$$
where $``\text{meas}"$ denotes the Lebesgue measure,
$$g(\alpha)=\sum_{l=1}^{+\infty}\frac{1-2\{l\alpha\}}{l}$$  
and 
$$C_0(\mathbb{R})=\{f\in C(\mathbb{R})\::\: \forall\: \epsilon>0,\: \exists\:\text{a compact set}\ \mathcal{K}\subset\mathbb{R},\:\text{such that}\ |f(x)|<\epsilon,\forall\: x\not\in \mathcal{K}  \}.$$
\end{definition}
\textit{Remark.} The convergence of this series has been investigated by R. de la Bret\`eche and G. Tenenbaum (see \cite{bre}). It depends on the partial fraction
expansion of the number $\alpha$.\\
\begin{theorem}\label{x:introprob}
i) $F$ is a continuous function of $z$.\\
ii) Let $A_0,$ $A_1$ be fixed constants, such that $1/2< A_0<A_1<1$. Let also
$$H_k=\int_0^1\left(\frac{g(x)}{\pi}\right)^{2k}dx,$$
$H_k$ is a positive constant depending only on $k$, $k\in\mathbb{N}$.\\ 
There is a unique positive measure $\mu$ on $\mathbb{R}$ with the following properties:\\
\ (a) For $\alpha<\beta\in\mathbb{R}$ we have
$$\mu([\alpha,\beta])=(A_1-A_0)(F(\beta)-F(\alpha)).$$
(b)
\begin{equation}
\int x^kd\mu=\left\{
\begin{array}{l l}
    (A_1-A_0)H_{k/2}\:, & \quad \text{for even}\: k\\
    0\:, & \quad \text{otherwise}\:.\\
  \end{array} \right.
 \nonumber
\end{equation}
(c) For all $f\in C_0(\mathbb{R})$, we have
$$\lim_{b\rightarrow+\infty}\frac{1}{\phi(b)}\sum_{\substack{r\::\: (r,b)=1\\ A_0b\leq r\leq A_1b}}f\left( \frac{1}{b}c_0\left( \frac{r}{b}\right) \right)=\int f\:d\mu,$$
where $\phi(\cdot)$ denotes the Euler phi-function.
\end{theorem}
\textit{Remark.} R. W. Bruggeman (see \cite{bruggeman1}, \cite{bruggeman2}) and I. Vardi (see \cite{vardi}) have investigated the equidistribution of Dedekind sums. In contrast with the work in this paper, they consider an additional averaging over the denominator.
\subsection{Outline of the proof and further results}
\vspace{5mm}
In \cite{ras}, M. Th. Rassias proved the following asymptotic formula:
\begin{theorem}$\label{x:rassiass}$
For $b\geq 2$, $b\in\mathbb{N}$, we have
$$c_0\left(\frac{1}{b}\right)=\frac{1}{\pi}\:b\log b-\frac{b}{\pi}(\log 2\pi-\gamma)+O(1)\:.$$
\end{theorem}
In that paper, a method which applies properties of
fractional parts in order to approach the cotangent sum in question is described. This method is generalized in the present paper, where
some stronger results are being proved.\\ 
We initially provide a proof of an improvement of Theorem $\ref{x:rassiass}$ as an asymptotic expansion. Namely, we prove the following:
\begin{theorem}$\label{x:vasi}$
Let $b,n\in\mathbb{N}$, $ b\geq 6N$, with $N=\left \lfloor n/2 \right \rfloor+1$.There exist absolute real constants $A_1,A_2\geq 1$ and absolute real constants $E_l$, $l\in\mathbb{N}$ with $|E_l|\leq (A_1l)^{2l}$, such that for each $n\in\mathbb{N}$ we have
$$c_0\left(\frac{1}{b} \right)=\frac{1}{\pi}b\log b-\frac{b}{\pi}(\log 2\pi-\gamma)-\frac{1}{\pi} +\sum_{l=1}^nE_lb^{-l}+R^*_n(b)$$
where $$|R^*_n(b)|\leq (A_2n)^{4n}\:b^{-(n+1)}.$$
\end{theorem}
\noindent Additionally, we investigate the cotangent sum $c_0\left(\frac{r}{b} \right)$ for a fixed arbitrary positive integer value of $r$ and for large integer values of $b$ and prove the following results.
\begin{proposition}\label{x:kiriarxo}
For $r$, $b\in\mathbb{N}$ with $(r,b)=1$, it holds
$$c_0\left(\frac{r}{b}\right)=\frac{1}{r}\:c_0\left(\frac{1}{b}\right)-\frac{1}{r}Q\left(\frac{r}{b}\right)\:,$$
where
$$Q\left(\frac{r}{b}\right)=\sum_{m=1}^{b-1}\cot\left(\frac{\pi mr}{b}\right)\left\lfloor \frac{rm}{b}\right\rfloor\:. $$
\end{proposition}
\begin{theorem}$\label{x:teliko}$
Let $r,b_0\in\mathbb{N}$ be fixed, with $(b_0,r)=1$. Let $b$ denote a positive integer with $b\equiv b_0\:(\bmod\:r)$. Then, there exists a constant
$C_1=C_1(r,b_0)$, with $C_1(1,b_0)=0$, such that
$$c_0\left(\frac{r}{b}\right)=\frac{1}{\pi r}b\log b-\frac{b}{\pi r}(\log 2\pi -\gamma)+C_1\:b+O(1),$$
for large integer values of $b$.
\end{theorem}
\begin{theorem}$\label{x:QQ}$
Let $k\in\mathbb{N}$ be fixed. Let also $A_0$, $A_1$ be fixed constants such that $1/2<A_0<A_1<1$. Then there exist explicit constants $E_k>0$ and $H_k>0$, depending only on $k$, such that\\ \\
(a)
$$ \sum_{\substack{r:(r,b)=1\\A_0b\leq r\leq A_1b}}Q\left( \frac{r}{b}\right)^{2k}=E_k\cdot(A_1^{2k+1}-A_0^{2k+1})b^{4k}\phi(b)(1+o(1)),\ \ (b\rightarrow+\infty).$$
(b)
$$\sum_{\substack{r:(r,b)=1\\A_0b\leq r\leq A_1b}}Q\left(\frac{r}{b}\right)^{2k-1}=o\left(b^{4k-2}\phi(b)\right),\ \ (b\rightarrow+\infty).$$
(c)
$$\sum_{\substack{r:(r,b)=1\\A_0b\leq r\leq A_1b}}c_0\left(\frac{r}{b} \right)^{2k}=H_k\cdot(A_1-A_0)b^{2k}\phi(b)(1+o(1)),\ \ (b\rightarrow+\infty).$$
(d)
$$\sum_{\substack{r:(r,b)=1\\A_0b\leq r\leq A_1b}}c_0\left(\frac{r}{b}\right)^{2k-1}=o\left(b^{2k-1}\phi(b)\right),\ \ (b\rightarrow+\infty).$$
\end{theorem}
Using the method of moments, we deduce detailed information about the distribution of the values of $c_0(r/b)$, where $A_0b\leq r\leq A_1b$ and $b\rightarrow+\infty$. Namely, we prove Theorem \ref{x:introprob}.\\

Finally, we study the convergence of the series
$$\sum_{k\geq 0}H_kx^{2k}$$
and prove the following theorem:
\begin{theorem}\label{x:introconv}
The series
$$\sum_{k\geq 0}H_kx^{2k},$$
converges only for $x=0$.
\end{theorem}
Another interesting question is whether
the series 
$$\sum_{k\geq 0}\frac{H_k}{(2k)!}x^{2k},$$
has a positive radius of convergence. This would lead to a simplification in the proof of our equidistribution result, since in this case we could 
apply the theory of distributions which are determined by their moments.
%
%
%
%
%
%
\vspace{10mm}
\section{Approximating $c_0(1/b)$ for every integer value of $b$}
\vspace{5mm}
\noindent It is a known fact (see \cite{ras}), that
\begin{proposition}$\label{x:fractional}$
For every $a$, $b$, $n\in\mathbb{N}$, $b \geq 2$, with $b\nmid na$ we have
$$x_n:=\left\{\frac{na}{b}\right\}=\frac{1}{2}-\frac{1}{2b}\sum_{m=1}^{b-1}\cot\left(\frac{\pi m}{b}\right)\sin\left(2\pi mn\frac{a}{b}\right)\:.$$
\end{proposition}
\noindent Therefore, we obtain the following proposition.
\begin{proposition}$\label{x:ena}$
For every positive integer $b$, $b\geq 2$, we have
\[
c_0\left(\frac{1}{b}\right)=\frac{1}{\pi}\sum_{\substack{a\geq1 \\ b\nmid a}}\frac{b(1-2x_1)}{a}\:.\tag{1}
\]
\end{proposition}
\noindent Set 
$$G_L(b)=\sum_{\substack{1\leq a\leq L \\ b\nmid a}}\left(\frac{b}{a}\left(1+2\left\lfloor \frac{a}{b}\right\rfloor\right)-2\right),$$
then the following lemma follows (see \cite{ras}).
\begin{lemma}$\label{x:G}$
For every $b$, $L\in\mathbb{N}$, with $b$, $L\geq 2$, it holds
\[
G_L(b)=-\log\frac{L}{b}+b(\log L+\gamma)-2L+S(L;b)+
O\left(\frac{b}{L}  \right),
\]
where
$$S(L;b)=2b\sum_{1\leq a\leq L}\frac{1}{a}\left\lfloor \frac{a}{b}\right\rfloor.$$
\end{lemma}
The key tool for obtaining an asymptotic expansion for $S(L;b)$ is the generalized Euler summation formula. The following definition is needed.
\begin{definition}\label{x:bern}
The sequence $B_j$ of Bernoulli numbers is defined by $B_{2n+1}=0$,
$$B_{2n}=2\frac{(2n)!}{(2\pi)^{2n}}\sum_{\nu\geq1}\nu^{-2n}.$$
If $f$ is a function that is differentiable at least $(2N+1)$ times in $\left[ 0,Z\right]$, let
$$r_N(f,Z)=\frac{1}{(2N+1)!}\int_{0}^{Z}(u-\left\lfloor u\right\rfloor+B)^{2N+1}f^{(2N+1)}(u)du,$$
where the following notation is used
$$(u-\left\lfloor u\right\rfloor+B)^{2N+1}=\left((u-\left\lfloor u\right\rfloor)+B\right)^{2N+1}:=\sum_{j=0}^{2N+1}\binom{2N+1}{j} (u-\left\lfloor u\right\rfloor)^{j}B_{2N+1-j}.$$
Additionally, let
$$F_i(k,b)=((k+1)b-1)^{-i}-(kb-1)^{-i}.$$
\end{definition}
\begin{theorem}$\label{x:euler}$(Generalized Euler Summation Formula (cf. \cite{eul}))\\
Let $f$ be $(2N+1)$ times differentiable in the interval $\left[ 0,Z\right]$, then\index{formula!generalized Euler summation}
$$ \sum_{\nu=0}^Zf(\nu)=\frac{f(0)+f(Z)}{2}+\int_0^{Z}f(u)du+\sum_{j=1}^{N}\frac{B_{2j}}{(2j)!}\left( f^{(2j-1)}(Z)-f^{(2j-1)}(0) \right)+r_{N}(f,Z).  $$
\end{theorem}
\begin{lemma}$\label{x:24}$
For $N\in\mathbb{N}$, we have
$$S(L;b)=2b\sum_{k\leq L/b}k\left( \log\frac{(k+1)b-1}{kb-1}+\frac{1}{2}F_1(k,b) \right)+2b\sum_{j=1}^N\frac{B_{2j}}{2j}\sum_{k\leq L/b}kF_{2j}(k,b)+2br_N\left(f,\frac{L}{b}\right),$$
where the function $f$ satisfies:
\begin{eqnarray}
f(u)=\left\{ 
  \begin{array}{l l}
   1/u\:, & \quad \text{if $u\geq 1$}\vspace{2mm}\\ 
  0\:, & \quad u=0\:,\\
  \end{array} \right.
\nonumber
\end{eqnarray}
and $f\in C^{\infty}([ 0,\infty))$ with $f^{(j)}(0)=0$ for $j\leq 2N+1$.
\begin{proof} By splitting the range of summation for $S(L;b)$ into subintervals on which $\left\lfloor a/b\right\rfloor$ is constant, we have
$$S(L;b)=2b\sum_{k\leq L/b}k\left(\sum_{kb\leq a<(k+1)b}\frac{1}{a} \right).$$
For the inner sum we apply Theorem $\ref{x:euler}$ and we obtain
\begin{eqnarray}
\sum_{kb\leq a<(k+1)b}\frac{1}{a} &=&\int_{kb-1}^{(k+1)b-1}\frac{du}{u}+\frac{1}{2}F_1(k,b)+\sum_{j=1}^N\frac{B_{2j}}{2j}F_{2j}(k,b)\nonumber\\
&-&\int_{kb-1}^{(k+1)b-1}\left( \sum_{l=0}^{2N+1}\binom{2N+1}{l} (u-\left\lfloor u\right\rfloor)^{l}B_{2N+1-l}\right)u^{-(2N+2)}du.\nonumber
\end{eqnarray}
Lemma $\ref{x:24}$ now follows from Definition $\ref{x:bern}$.
\end{proof}
\end{lemma}
\begin{lemma}$\label{x:rN}$
Let
$$r_N(b)=\sum_{l=0}^{2N+1}\binom{2N+1}{l}B_{2N+1-l}\sum_{k\leq L/b}k\:I(b,k,l),$$
where 
$$I(b,k,l)=\int_{kb-1}^{(k+1)b-1}(u-\left\lfloor u\right\rfloor)^{l}u^{-(2N+2)}du.$$
Then there exist absolute constants $C_0$, $C_1$, such that
$$r_N(b)=C_0+C(N,b)5^N(2N+1)!b^{-(2N+1)} ,$$
where $$|C(N,b)|\leq C_1.$$
\end{lemma}
\begin{proof}
The functions $I(b,k,l)$ are differentiable with respect to $b$ for $b>0$, except for integer values of $b$. By the chain rule we get
\begin{eqnarray}
\frac{dI(b,k,l)}{db}&=&(k+1)\left((k+1)b-\left\lfloor (k+1)b\right\rfloor\right)^{l}((k+1)b-1)^{-(2N+2)}\nonumber\\
&-&k(kb-\left\lfloor kb\right\rfloor)^{l}(kb-1)^{-(2N+2)}.\nonumber
\end{eqnarray}
Thus
$$\frac{dI(b,k,l)}{db}=O\left( k^{-(2N+1)}b^{-(2N+2)} \right),$$
because
$$(kb- \left\lfloor kb\right\rfloor)^l\leq1\ \ \text{and}\ \ ((k+1)b-\left\lfloor (k+1)b\right\rfloor)^l\leq 1,$$
for $0\leq l\leq 2N+1$.\\
Thus, for $b\not\in\mathbb{Z}$ we obtain
$$\left|\frac{\partial r_N(b)}{\partial b}\right|\leq C_2 5^N(2N+1)!b^{-(2N+2)} ,$$
for an absolute constant $C_2$, since 
$$\binom{2N+1}{l}\leq (4+o(1))^N\ \ \text{and}\ \ B_{2N+1-l}\leq (2N+1)! $$
\end{proof}
\begin{lemma}
We have
$$ F_j(k,b)=(k+1)^{-j}b^{-j}\sum_{\nu\geq0}\binom{-j}{\nu}(k+1)^{-\nu}b^{-\nu}-k^{-j}b^{-j}\sum_{\nu\geq0}\binom{-j}{\nu}k^{-\nu}b^{-\nu}. $$
\end{lemma}
\begin{proof} From Definition $\ref{x:bern}$ we obtain,
\begin{eqnarray}
F_j(k,b)&=&((k+1)b-1)^{-j}-(kb-1)^{-j}\nonumber\\
&=&(k+1)^{-j}b^{-j}\left(1-\frac{1}{(k+1)b}\right)^{-j}-k^{-j}b^{-j}\left( 1-\frac{1}{kb} \right)^{-j}\nonumber\\
&=&(k+1)^{-j}b^{-j}\sum_{\nu\geq 0}\binom{-j}{\nu}(k+1)^{-\nu}b^{-\nu}-k^{-j}b^{-j}\sum_{\nu\geq 0}\binom{-j}{\nu}k^{-\nu}b^{-\nu}\nonumber,
\end{eqnarray}
by the binomial formula.
\end{proof}
\begin{lemma}$\label{x:27}$
Let $L,b,n\in\mathbb{N}$, $L\geq b\geq 6N$, with $N=\left \lfloor n/2 \right \rfloor+1$.There exist absolute constants $A_1,A_2\geq 1$, $F\in\mathbb{R}$ and absolute constants $E_l$, $l\in\mathbb{N}$ with $$|E_l|\leq (A_1l)^{2l},$$ such that for each $n\in\mathbb{N}$ we have
$$S(L;b)=2L-b\log\frac{L}{b}+\log\frac{L}{b}+Fb+\gamma-1+\sum_{l=1}^nE_lb^{-l}+R_n(b,L)+O_n\left(\frac{b^2}{L} \right),$$
where $$|R_n(b,L)|\leq (A_2n)^{4n}b^{-(n+1)}+O_n\left(\frac{1}{L} \right).$$
\end{lemma}
\begin{proof}
By Lemma $\ref{x:24}$, for $N\in\mathbb{N}$ we obtain
\begin{align*}
S(L;b)&=2b\sum_{k\leq L/b}k\left(\log\frac{(k+1)b-1}{kb-1}+\frac{1}{2}F_1(k,b) \right)\tag{2}\\
&+2b\sum_{j=1}^{N}\frac{B_{2j}}{2j}\sum_{k\leq L/b}kF_{2j}(k,b)+s(N,b),\\
\end{align*}
where 
$$s(N,b)=2b\:r_{N}\left(f,\frac{L}{b}\right),$$
as it is defined in Lemma $\ref{x:24}$.\\
We expand the terms in the above expression, using the Taylor expansion, as follows
\[
\log\left(1+\frac{1}{k} \right)=\frac{1}{k}-\frac{1}{2k^2}+\sum_{\nu\geq 3}\frac{(-1)^{\nu+1}}{\nu}k^{-\nu}\tag{3}
\]
\[
\log\left(1-\frac{1}{(k+1)b} \right)=-\frac{1}{(k+1)b}-\sum_{\nu\geq 2}\nu^{-1}(k+1)^{-\nu}b^{-\nu}\tag{4}
\]
\[
\log\left(1-\frac{1}{kb} \right)=-\frac{1}{kb}-\sum_{\nu\geq 2}\nu^{-1}k^{-\nu}b^{-\nu}.\tag{5}
\]
\begin{align*}
F_1(k,b)&=\frac{1}{(k+1)b-1}-\frac{1}{kb-1}\tag{6}\\
&=\frac{1}{(k+1)b}\sum_{\nu\geq 0}(k+1)^{-\nu}b^{-\nu}-\frac{1}{kb}\sum_{\nu\geq 0}k^{-\nu}b^{-\nu}\\
&=b^{-1}\left(\frac{1}{k+1}-\frac{1}{k} \right)+\sum_{\nu\geq 2}b^{-\nu}((k+1)^{-\nu}-k^{-\nu}).
\end{align*}
\begin{align*}
F_{2j}(k,b)&=(k+1)^{-2j}b^{-2j}\sum_{\nu\geq 0}\binom{-2j}{\nu}(k+1)^{-\nu}b^{-\nu}\tag{7}\\
&-k^{-2j}b^{-2j}\sum_{\nu\geq 0}\binom{-2j}{\nu}k^{-\nu}b^{-\nu}.\\
\end{align*}
Insertion of the formulas (3)--(7) into (2) yields
\begin{align*}
S(L;b)&=2b\sum_{k\leq L/b}k\left( \frac{1}{k}-\frac{1}{2k^2}-\frac{1}{(k+1)b}+\frac{1}{kb} \right)+2b\sum_{k\leq L/b}k\sum_{\nu\geq 3}\frac{(-1)^{\nu+1}}{\nu}k^{-\nu}\tag{8}\\
&-2b\sum_{k\leq L/b}k\sum_{\nu\geq 2}\nu^{-1}(k+1)^{-\nu}b^{-\nu}+2b\sum_{k\leq L/b}k\sum_{\nu\geq 2}\nu^{-1}k^{-\nu}b^{-\nu}\\
&+\sum_{k\leq L/b}k\left( \frac{1}{k+1}-\frac{1}{k} \right)+2b\sum_{k\leq L/b}k\sum_{\nu\geq 2}b^{-\nu}((k+1)^{-\nu}-k^{-\nu})\\
&+2b\sum_{j=1}^N\frac{B_{2j}}{2j}\sum_{k\leq L/b} k(k+1)^{-2j}b^{-2j}\sum_{\nu\geq 0}\binom{-2j}{\nu}(k+1)^{-\nu}b^{-\nu}\\
&-2b\sum_{j=1}^N\frac{B_{2j}}{2j}\sum_{k\leq L/b}k^{-2j+1}b^{-2j}\sum_{\nu\geq 0}\binom{-2j}{\nu}k^{-\nu}b^{-\nu}+s(N,b).
\end{align*}
We introduce the following constants
$$D_1=\sum_{\nu\geq 3}\frac{(-1)^{\nu+1}}{\nu}\sum_{k\geq 1}k^{1-\nu}$$
and
$$D_{2,\nu}=\sum_{k\geq 1}k(k^{-\nu}-(k+1)^{-\nu}),\ \nu\geq 2.$$
From the mean value theorem, we have 
\[
k^{-\nu}-(k+1)^{-\nu}=\nu(k+\phi)^{-(\nu+1)},\tag{9}
\]
for some $\phi$, with $0<\phi<1$,
\begin{align*}
\sum_{\nu\geq 3}\frac{(-1)^{\nu+1}}{\nu}\sum_{k\leq L/b}k^{1-\nu}&=\sum_{\nu\geq 3}\frac{(-1)^{\nu+1}}{\nu}\left(\sum_{k\geq 1}k^{1-\nu}+O\left(\frac{1}{\nu}\left(\frac{b}{L} \right)^{\nu-2} \right) \right)\tag{10}\\
&=\sum_{\nu\geq 3}\frac{(-1)^{\nu+1}}{\nu}\sum_{k\geq 1}k^{1-\nu}+O\left(\sum_{\nu\geq 3}\frac{1}{\nu^2}\left(\frac{b}{L} \right)^{\nu-2} \right)\\
&=D_1+O\left(\frac{b}{L}  \right)
\end{align*}
\begin{align*}
&\sum_{\nu\geq 2}\nu^{-1}\sum_{k\leq L/b}k\left(k^{-\nu}-(k+1)^{-\nu}  \right)b^{-\nu}\tag{11}=\\
&=\sum_{\nu\geq 2}\nu^{-1}b^{-\nu}\left(\sum_{k\geq 1}k(k^{-\nu}-(k+1)^{-\nu})+O\left(\sum_{k>L/b}k(k^{-\nu}-(k+1)^{-\nu}) \right)  \right)\\
&=\sum_{\nu\geq 2}\nu^{-1}b^{-\nu}\left(D_{2,\nu}+O\left(\nu\left(\frac{b}{L} \right)^{\nu-1} \right) \right)=\sum_{2\leq\nu\leq M}\nu^{-1}b^{-\nu}D_{2,\nu}+O\left(\frac{b^{-1}}{L} \right)+\theta_{M}(b)b^{-(M+1)},
\end{align*}
where $M\in\mathbb{N}$, $|\theta_M(b)|\leq 1$.
\begin{align*}
&2b\sum_{k\leq L/b}k\sum_{\nu\geq 2}b^{-\nu}((k+1)^{-\nu}-k^{-\nu})=2b\sum_{\nu\geq 2}b^{-\nu}\sum_{k\leq L/b}k((k+1)^{-\nu}-k^{-\nu})\\
&=2b\sum_{\nu\geq 2}b^{-\nu}\left(\sum_{k\geq 1}k((k+1)^{-\nu}-k^{-\nu})+O\left(\sum_{k>L/b}k((k+1)^{-\nu}-k^{-\nu}) \right)  \right)\\
&=-2b\sum_{2\leq \nu\leq M}b^{-\nu}D_{2,\nu}+2\theta'_M(b)\left(\sum_{\nu>M}b^{1-\nu} \right)+O\left(\frac{b^{\nu+1}}{L^{\nu}} \right),
\end{align*}
where $|\theta'_M(b)|\leq 1$. Therefore
\begin{align*}
2b\sum_{k\leq L/b}&k\sum_{\nu\geq 2}b^{-\nu}((k+1)^{-\nu}-k^{-\nu})=\tag{12}\\
&=-2\sum_{2\leq \nu\leq M}b^{1-\nu}D_{2,\nu}+O\left(\frac{b^{\nu+1}}{L^{\nu}} \right)+4\theta_{M}(b)b^{-M}.
\end{align*}
\begin{align*}
\lambda_{L,b,j}&:=\sum_{k\leq L/b}\left(k(k+1)^{-2j}b^{-2j}\sum_{\nu\geq 0}\binom{-2j}{\nu}(k+1)^{-\nu}b^{-\nu}-k^{-2j+1}b^{-2j}\sum_{\nu\geq 0}\binom{-2j}{\nu}k^{-\nu}b^{-\nu}\right)\\
&=\sum_{\nu\geq 0}\binom{-2j}{\nu}b^{-2j-\nu}\sum_{k\leq L/b}k((k+1)^{-2j-\nu}-k^{-2j-\nu})\\
&=\sum_{\nu\geq 0}\binom{-2j}{\nu}b^{-2j-\nu}\left(D_{2,2j+\nu}+O\left(\sum_{k>L/b}k((k+1)^{-2j-\nu}-k^{-2j-\nu}) \right)  \right).
\end{align*}
Since
$$\left| \frac{\binom{-2j}{\nu+1}b^{-2j-\nu+1}}{\binom{-2j}{\nu}b^{-2j-\nu}}\right| \leq \frac{1}{2},\ \text{for}\ b\geq 6N,$$
there exist numbers $\theta_{j,M}(b)$ with $|\theta_{j,M}(b)|\leq 1$, such that
$$\sum_{\nu>M}\binom{-2j}{\nu}b^{-2j-\nu}D_{2,2j+\nu}=2\theta_{j,M}(b)(2j)^{M+1}b^{-2j-M-1}.$$
Therefore
\begin{align*}
\lambda_{L,b,j}&=\sum_{0\leq \nu\leq M}\binom{-2j}{\nu}b^{-2j-\nu}D_{2,2j+\nu}+2\theta_{j,M}(b)(2j)^{M+1}b^{-2j-M-1}\tag{13}\\
&+O\left( \sum_{\nu\geq 0}\left|\binom{-2j}{\nu}\right| b^{-2j-\nu}\sum_{k>L/b}k((k+1)^{-2j-\nu}-k^{-2j-\nu})\right)
\end{align*}
If we substitute the approximations (9)-(13) in (8), we obtain
\begin{align*}
&S(L;b)=2b\left(\frac{L}{b}-\frac{1}{2}\left(\log\frac{L}{b}+\gamma+O\left(\frac{b}{L} \right) \right) \right)-2\sum_{k\leq L/b}\left(1-\frac{1}{k+1} \right)+\frac{2L}{b}\\
&+2bD_1+O\left(\frac{b^2}{L} \right)+2\sum_{2\leq \nu\leq M}\nu^{-1}b^{1-\nu}D_{2,\nu}+O\left(\frac{1}{L} \right)+\theta_M(b)b^{-M}\\
&-\log\frac{L}{b}-\gamma+1+O\left(\frac{b}{L} \right)-2\sum_{2\leq \nu\leq M}b^{1-\nu}D_{2,\nu}+O\left(\frac{b^{\nu+1}}{L^{\nu}} \right)+4\theta_M(b)b^{-M}+s(N,b)\\
&+2b\sum_{j=1}^N\frac{B_{2j}}{2j}\left( \sum_{0\leq\nu\leq M}\binom{-2j}{\nu}b^{-2j-\nu}D_{2,2j+\nu}+2\theta_{j,M}(b)(2j)^{M+1}b^{-2j-M-1} +O(b^{-1}L^{-2j+1})\right)\\
&=2L-b\log\frac{L}{b}+2\gamma b+O\left(\frac{b^2}{L} \right)-2\frac{L}{b}+2\left(\log\frac{L}{b}+\gamma-1+O\left(\frac{b}{L} \right) \right)+2\frac{L}{b}\\
&+2bD_1+O\left(\frac{b^2}{L} \right)+2\sum_{2\leq \nu\leq M}\nu^{-1}b^{1-\nu}D_{2,\nu}+O\left(\frac{1}{L} \right)+\theta_M(b)b^{-M}\\
&-\log\frac{L}{b}-\gamma+1+O\left(\frac{b}{L} \right)-2\sum_{2\leq \nu\leq M}b^{1-\nu}D_{2,\nu}+O\left(\frac{b^{\nu+1}}{L^{\nu}} \right)+4\theta_M(b)b^{-M}\\
&+\sum_{j=1}^N\frac{B_{2j}}{j} \sum_{0\leq\nu\leq M}\binom{-2j}{\nu}b^{1-2j-\nu}D_{2,2j+\nu}+2\sum_{j=1}^N\frac{B_{2j}}{j} \theta_{j,M}(b)(2j)^{M+1}b^{-2j-M}\\
&+O\left(\sum_{j=1}^N\left|\frac{B_{2j}}{j}\right|L^{-2j+1}\right)+s(N,b).
\end{align*}
Hence, we get
\begin{align*}
&S(L;b)=2L-b\log\frac{L}{b}+2\gamma b+\log\frac{L}{b}+\gamma-1+2bD_1+2\sum_{2\leq \nu\leq M}\nu^{-1}b^{1-\nu}D_{2,\nu}\tag{14}\\
&-2\sum_{2\leq \nu\leq M}b^{1-\nu}D_{2,\nu}+\sum_{j=1}^N\frac{B_{2j}}{j} \sum_{0\leq\nu\leq M}\binom{-2j}{\nu}b^{1-2j-\nu}D_{2,2j+\nu}+O\left(\frac{b^2}{L} \right)\\
&+5\theta_M(b)b^{-M}+2\sum_{j=1}^N\frac{B_{2j}}{j} \theta_{j,M}(b)(2j)^{M+1}b^{-2j-M}+O\left(\sum_{j=1}^N\left|\frac{B_{2j}}{j}\right|L^{-2j+1}\right)+s(N,b).
\end{align*}
We now choose $M=n+1$. We define
$$F=2(\gamma+D_1+C_0),$$
where $C_0$ is defined as in Lemma $\ref{x:rN}$, and
$$E_l=(2(l+1)^{-1}-2)D_{2,l+1}+\sum_{\substack{j\leq(l+1)/2\\ j\leq N}}\frac{B_{2j}}{j}\binom{-2j}{l+1-2j}D_{2,l+1}.$$
Hence, by (14) and Lemma $\ref{x:rN}$ we obtain
\begin{align*}
S(L;b)&=2L-b\log\frac{L}{b}+\log\frac{L}{b}+Fb+(\gamma-1)+5\theta_{n+1}(b)b^{-n-1}+\sum_{l=1}^nE_lb^{-l}\\
&+\sum_{n+1\leq l\leq 2n+3}E_lb^{-l}+2\sum_{1\leq j\leq \left \lfloor \frac{n}{2} \right \rfloor+1}\frac{B_{2j}}{j}\theta_{j,n+1}(b)(2j)^{n+2}b^{-n-1-2j}\\
&+O_n\left(\frac{b^2}{L} \right)+O\left(\sum_{j=1}^N\left|\frac{B_{2j}}{j}\right|L^{-2j+1} \right)+C(n,b)5^{\left \lfloor n/2 \right \rfloor+1}\left(2\left \lfloor \frac{n}{2} \right \rfloor+3\right)!b^{-(2\left \lfloor n/2 \right \rfloor+2)}.
\end{align*}
We have $|D_{2,l+1}|\leq 1$. Thus, we get
\begin{align*}
|E_l|&\leq 2+\min\left(\frac{l+1}{2},\left \lfloor \frac{n}{2} \right \rfloor+1\right)\max_{\substack{j\leq (l+1)/2 \\ j\leq \left \lfloor n/2 \right \rfloor+1}}\left| \frac{B_{2j}}{j} \right|\left|\binom{-2j}{l+1-2j} \right|\\
&\leq 2+\min\left(\frac{l+1}{2},\left \lfloor \frac{n}{2} \right \rfloor+1\right)(l+1)!\max_{\substack{j\leq (l+1)/2 \\ j\leq \left \lfloor n/2 \right \rfloor+1}}(2j)^{l+1-2j}\\
&\leq 2+\min\left(\frac{l+1}{2},\left \lfloor \frac{n}{2} \right \rfloor+1\right)(l+1)!\max_{\substack{j\leq (l+1)/2 \\ j\leq \left \lfloor n/2 \right \rfloor+1}}(2j)^{l+1}\leq (A_1l)^{2l},
\end{align*}
for some absolute constant $A_1\geq 1$. We set
\begin{align*}
&R_n(b,L)=C(n,b)5^{\left \lfloor n/2 \right \rfloor+1}\left(2\left \lfloor \frac{n}{2} \right \rfloor+3\right)!b^{-(2\left \lfloor n/2 \right \rfloor+2)}+5\theta_{n+1}(b)b^{-n-1}+\sum_{n+1\leq l\leq 2n+3}E_lb^{-l}\\
&+2\sum_{1\leq j\leq \left \lfloor \frac{n}{2} \right \rfloor+1} \frac{B_{2j}}{j}\theta_{j,n+1}(b)(2j)^{n+2}b^{-n-1-2j}+O\left(\sum_{j=1}^N\left|\frac{B_{2j}}{j}\right|L^{-2j+1} \right).
\end{align*}
Thus
\begin{align*}
\left|R_n(b,L) \right|&\leq 5\theta_{n+1}(b)b^{-n-1}+2b^{-(n+1)}\left(\max_{n+1\leq l\leq 2n+3}|E_l|+2\max_{1\leq j\leq \left \lfloor \frac{n}{2} \right \rfloor+1}|B_{2j}|(2j)^{n+2} \right)\\
&+C(n,b)5^{\left \lfloor n/2 \right \rfloor+1}\left(2\left \lfloor \frac{n}{2} \right \rfloor+3\right)!b^{-(2\left \lfloor n/2 \right \rfloor+2)}+O_n\left(\frac{1}{L} \right)\\
&\leq 5\theta_{n+1}(b)b^{-n-1}+2b^{-(n+1)}\left((A_1(2n+3))^{4n+6}+2(n+1)!(n+2)^{n+2} \right)+\\
&+C(n,b)5^{\left \lfloor n/2 \right \rfloor+1}\left(2\left \lfloor \frac{n}{2} \right \rfloor+3\right)!b^{-(2\left \lfloor n/2 \right \rfloor+2)}+O_n\left(\frac{1}{L} \right)\\
&\leq b^{-(n+1)}(A_2n)^{4n}+O_n\left(\frac{1}{L} \right),
\end{align*}
for some absolute constant $A_2\geq 1$.
\end{proof}
\noindent Therefore, we are now able to prove the following proposition.
\begin{proposition}$\label{x:28}$
Let $L,b,n\in\mathbb{N}$, $L\geq b\geq 6N$, with $N=\left \lfloor n/2 \right \rfloor+1$.There exist absolute constants $A_1,A_2\geq 1$, $F\in\mathbb{R}$ and absolute constants $E_l$, $l\in\mathbb{N}$ with $$|E_l|\leq (A_1l)^{2l},$$ such that for each $n\in\mathbb{N}$ we have
$$G_L(b)=b\log b+(F+\gamma)b-1+\sum_{l=1}^nE_lb^{-l}+R_n(b,L)+O_n\left(\frac{b^2}{L} \right), $$
where $$|R_n(b,L)|\leq (A_2n)^{4n}b^{-(n+1)}+O_n\left(\frac{1}{L} \right).$$
\end{proposition}
\begin{proof}
It follows by putting together Lemma $\ref{x:G}$ and Lemma $\ref{x:27}$.
\end{proof}
\noindent However, by the definition of $G_L(b)$ it follows that 
$$c_0\left(\frac{1}{b} \right)=\frac{1}{\pi}\lim_{L\rightarrow+\infty}G_L(b).$$
Thus by Proposition $\ref{x:28}$ we obtain the following theorem.
\begin{theorem}
Let $b,n\in\mathbb{N}$, $ b\geq 6N$, with $N=\left \lfloor n/2 \right \rfloor+1$.There exist absolute constants $A_1,A_2\geq 1$, $H\in\mathbb{R}$ and absolute constants $E_l$, $l\in\mathbb{N}$ with $$|E_l|\leq (A_1l)^{2l},$$ such that for each $n\in\mathbb{N}$ we have
$$c_0\left(\frac{1}{b} \right)=\frac{1}{\pi}b\log b+Hb-\frac{1}{\pi} +\sum_{l=1}^nE_lb^{-l}+R^*_n(b)$$
where $$|R^*_n(b)|\leq (A_2n)^{4n}b^{-(n+1)}.$$
\end{theorem}
\noindent By Vasyunin's theorem, we know that for sufficiently large $b$ it holds\index{theorem!Vasyunin}
$$c_0\left(\frac{1}{b}\right)=\frac{1}{\pi}\:b\log b-\frac{b}{\pi}(\log 2\pi-\gamma)+O(\log b)\:.$$
Therefore, by comparison of the coefficients of $b$ in the above expressions for $c_0(1/b)$ we get:
$$H=\frac{\gamma-\log 2\pi}{\pi}.$$
Hence we obtain the following theorem, that is Theorem $\ref{x:vasi}$ stated in the Introduction.
\begin{theorem}
Let $b,n\in\mathbb{N}$, $ b\geq 6N$, with $N=\left \lfloor n/2 \right \rfloor+1$.There exist absolute constants $A_1,A_2\geq 1$ and absolute real constants $E_l$, $l\in\mathbb{N}$ with $$|E_l|\leq (A_1l)^{2l},$$ such that for each $n\in\mathbb{N}$ we have
$$c_0\left(\frac{1}{b} \right)=\frac{1}{\pi}b\log b-\frac{b}{\pi}(\log 2\pi-\gamma)-\frac{1}{\pi} +\sum_{l=1}^nE_lb^{-l}+R^*_n(b)$$
where $$|R^*_n(b)|\leq (A_2n)^{4n}\:b^{-(n+1)}.$$
\end{theorem}
%
%
%
%
%
%
%
%
%
%
%
%
\vspace{10mm}
\section{Properties of $c_0\left(r/b\right)$ for fixed $r$ and large $b$}
\vspace{5mm}
\noindent We can generalize Proposition $\ref{x:fractional}$ in order to study the cotangent sum $c_0\left(\frac{r}{b}\right)$ for an arbitrary positive integer value of $r$ as $b\rightarrow+\infty$. \\
Following a method similar to the one used to prove Proposition $\ref{x:fractional}$, one can prove that:
\begin{proposition}\label{x:simadiko}
For every $r$, $a$, $b$, $n\in\mathbb{N}$, $b \geq 2$, with $(r,b)=1$, $b\nmid na$, we have
$$\sum_{m=1}^{b-1}\cot\left(\frac{\pi mr}{b}\right)\cos\left(2\pi m\frac{nra}{b}\right)=0\:.$$
and
$$x_n=\frac{1}{2}-\frac{1}{2b}\sum_{m=1}^{b-1}\cot\left(\frac{\pi mr}{b}\right)\sin\left(2\pi m\frac{nra}{b}\right)\:.$$
\end{proposition}
\noindent Similarly to the case when $r=1$, by the use of the identity
$$\sum_{a\geq 1}\frac{\sin(a\theta)}{a}=\frac{\pi-\theta}{2},\:0<\theta<2\pi\:,$$
when $b$ is such that $(r,b)=1$ and $b\nmid a$, we obtain
$$\sum_{\substack{a\geq1 \\ b\nmid a}}\frac{b(1-2x_1)}{a}=\pi rc_0\left(\frac{r}{b}\right)+\pi\sum_{m=1}^{b-1}\cot\left(\frac{\pi mr}{b}\right)\left\lfloor \frac{rm}{b}\right\rfloor\:.$$
Equivalently, by Proposition $\ref{x:ena}$ we can write
\begin{proposition}$\label{x:vasiko}$
For $r$, $b\in\mathbb{N}$ with $(r,b)=1$, it holds
$$c_0\left(\frac{r}{b}\right)=\frac{1}{r}\:c_0\left(\frac{1}{b}\right)-\frac{1}{r}Q\left(\frac{r}{b}\right)\:,$$
where
$$Q\left(\frac{r}{b}\right)=\sum_{m=1}^{b-1}\cot\left(\frac{\pi mr}{b}\right)\left\lfloor \frac{rm}{b}\right\rfloor\:. $$
\end{proposition}
\noindent By the use of the above proposition, we shall prove the following theorem.
\begin{theorem}$\label{x:133}$
Let $r,b_0\in\mathbb{N}$ be fixed, with $(b_0,r)=1$. Let $b$ denote a positive integer with $b\equiv b_0\:(\bmod\:r)$. Then, there exists a constant
$C_1=C_1(r,b_0)$, with $C_1(1,b_0)=0$, such that
$$c_0\left(\frac{r}{b}\right)=\frac{1}{\pi r}b\log b-\frac{b}{\pi r}(\log 2\pi -\gamma)+C_1\:b+O(1),$$
for large integer values of $b$.
\end{theorem}
\begin{proof}
By Proposition \ref{x:vasiko}, we know that
$$c_0\left(\frac{r}{b}\right)=\frac{1}{r}\:c_0\left(\frac{1}{b}\right)-\frac{1}{r}Q\left(\frac{r}{b}\right)\:.$$
However, by splitting the range of summation of $Q(r/b)$ into subintervals on which $\left\lfloor rm/b\right\rfloor$ assumes constant values, we have
\begin{align*}
Q\left(\frac{r}{b}\right)&=\sum_{m=1}^{b-1}\cot\left(\frac{\pi mr}{b}\right)\left\lfloor \frac{rm}{b}\right\rfloor\\
&=\sum_{j=0}^{r-1}j\sum_{j\leq \left\lfloor \frac{rm}{b}\right\rfloor<j+1}\cot\left(\frac{\pi mr}{b}\right).
\end{align*}
We shall evaluate the inner sum by applying the partial fraction decomposition of the cotangent function. It is a known fact from Complex Analysis that
\begin{align*}
\pi \cot(\pi z)&=\frac{1}{z}+\sum_{\substack{n=-\infty\\ n\neq 0}}^{+\infty}\left( \frac{1}{z-n}+\frac{1}{n} \right)\\
&=\frac{1}{z}+\frac{1}{z-1}+g_*(z),
\end{align*}
where
$$g_*(z)=\frac{1}{z+1}+2z\sum_{n\geq 1}\frac{1}{z^2-n^2}.$$
It follows that $g_*(z)$ is a continuously differentiable function for $0\leq z<1$.\\
\noindent We consider the sets
$$S_j=\left\{ rm\: :\: bj\leq rm<b(j+1),\:m\in\mathbb{Z} \right\}.$$
Then
$$S_j=\left\{ bj+s_j,\: bj+s_j+r,\ldots,\:bj+s_j+d_jr \right\},$$
where $s_j$ is a positive integer different from zero and $d_j$ is an appropriate nonnegative integer, since $(b,r)=1$.\\
Let  
$$b=s_j+d_jr+t_j\ \ \text{with}\ \ 1\leq t_j< r.$$
By the definition of $S_j$ we have
\[
s_j\equiv -bj\:(\bmod\: r)\ \ \text{and}\ \ t_j\equiv b-s_j\:(\bmod\: r)\tag{15}
\]
and thus
\[
t_j\equiv b(j+1)\: (\bmod\: r).\tag{16}
\]
By the definition of $S_j$ and application of partial fraction decomposition, we obtain
\begin{align*}
\sum_{j\leq \left\lfloor \frac{rm}{b}\right\rfloor<j+1}\cot\left(\frac{\pi mr}{b}\right)&=\sum_{l=0}^{d_j}\cot\left(\pi \frac{s_j+lr}{b} \right)\text{(since the cotangent function has period\:}\pi)\\
&=\frac{b}{\pi}\sum_{l=0}^{d_j}\frac{1}{s_j+lr}+\frac{b}{\pi}\sum_{l=0}^{d_j}\frac{1}{s_j+lr-b}+\sum_{l=0}^{d_j}g_*\left(\frac{s_j+lr}{b}  \right).\tag{17}
\end{align*}
We shall apply Euler's summation formula\index{formula!Euler summation} (cf. \cite{eul}, p. 47). Let $f$ be a continuously differentiable function on the interval $[0,n]$, then we have
$$\sum_{\nu=0}^nf(\nu)=\frac{f(0)+f(n)}{2}+\int_0^nf(x)dx+\int_0^nf'(x)P_1(x)dx,$$
where $P_1(x)=x-\left\lfloor x\right\rfloor-1/2$ is the Bernoulli polynomial of first degree. We obtain
\begin{align*}
\sum_{l=0}^{d_j}\frac{1}{s_j+lr}&=\int_0^{d_j}\frac{du}{s_j+ur}-r\int_0^{d_j}\frac{P_1(u)}{\left(s_j+ur\right)^2}du+\frac{1}{2s_j}+\frac{1}{2(s_j+d_jr)}\\
&=\frac{1}{r}\log(s_j+d_jr)-\frac{1}{r}\log s_j-r\int_0^{+\infty}\frac{P_1(u)}{\left(s_j+ur\right)^2}du+\frac{1}{2s_j}+O\left(\frac{1}{b}\right)\tag{18}.
\end{align*}
By the definition of $S_j$ we have
$$b(j+1)\leq bj+s_j+d_jr+r$$
and therefore
$$s_j+d_jr=b+O(1).$$
Analogously
$$t_j+d_jr=b+O(1).$$
By the substitution $l=d_j-\tilde{l}$ and Euler's summation formula, we obtain
\begin{align*}
\sum_{l=0}^{d_j}\frac{1}{s_j+lr-b}&=-\sum_{\tilde{l}=0}^{d_j}\frac{1}{t_j+\tilde{l}r}\ \ (\text{since}\ t_j=b-d_jr-s_j)\\
&=\frac{1}{2(s_j+d_jr-b)}+\frac{1}{2(s_j-b)}\\
&-\int_0^{d_j}\frac{du}{t_j+ur}+r\int_0^{+\infty}\frac{P_1(u)}{\left(t_j+ur\right)^2}du+O\left(\frac{1}{b} \right)\tag{19}
\end{align*}
since
$$\frac{1}{2(t_j+d_jr)}=O\left(\frac{1}{b} \right),$$
because of the definition of $S_j$.\\
By the substitution $\nu=u/b$ and by the property $d_j=b/r+O(1)$ and Euler's summation formula, we obtain
\[
\sum_{l=0}^{d_j}g_*\left(\frac{ s_j+lr}{b}\right)=b\int_0^{1/r}g_*\left( vr \right)dv+O(1),\tag{20}
\]
because $$s_j+d_jr=b+O(1).$$
Therefore by (17), (18), (19) and (20) we obtain
\begin{align*}
c_0\left(\frac{r}{b}\right)&=\frac{1}{r}\:c_0\left(\frac{1}{b}\right)-\frac{1}{r}Q\left(\frac{r}{b}\right)\\
&=\frac{1}{r}\:c_0\left(\frac{1}{b}\right)  -\frac{1}{r}\sum_{j=0}^{r-1}j\sum_{j\leq \left\lfloor \frac{rm}{b}\right\rfloor<j+1}\cot\left(\frac{\pi mr}{b}\right)\\
&=\frac{1}{r}\:c_0\left(\frac{1}{b}\right)-\frac{1}{r}\sum_{j=0}^{r-1}j\left(\frac{b}{\pi} \sum_{l=0}^{d_j}\frac{1}{s_j+lr}+\frac{b}{\pi} \sum_{l=0}^{d_j}\frac{1}{s_j+lr-b} +\sum_{l=0}^{d_j}g_*\left(\frac{ s_j+lr}{b}\right)\right)\\
&=\frac{1}{r}\:c_0\left(\frac{1}{b}\right)-\frac{b}{\pi r}\sum_{j=0}^{r-1}j\left( \frac{1}{r}\log(s_j+d_jr)-\frac{1}{r}\log s_j-r\int_0^{+\infty}\frac{P_1(u)}{\left(s_j+ur\right)^2}du
+\frac{1}{2s_j}+O\left( \frac{1}{b}\right) \right)\\
&-\frac{b}{\pi r}\sum_{j=0}^{r-1}j\left( -\frac{1}{2t_j}+\frac{1}{2(s_j-b)}-\left( \log(t_j+rd_j)-\log t_j \right)\frac{1}{r}+r\int_0^{+\infty}\frac{P_1(u)}{\left(t_j+ur\right)^2}du
+O\left( \frac{1}{b}\right) \right)\\
&-\frac{b}{r}\sum_{j=0}^{r-1}j\int_0^{1/r}g_*\left( vr \right)dv+O(1).
\end{align*}
Thus, by Theorem $\ref{x:rassiass}$, we obtain
\begin{align*}
c_0\left(\frac{r}{b}\right)&=\frac{1}{\pi r}b\log b-\frac{b}{\pi r}(\log 2\pi -\gamma)+O(1)\\
&-\frac{b}{\pi r}\sum_{j=0}^{r-1}j\left( \frac{1}{r}\log(s_j+d_jr)-\frac{1}{r}\log s_j-r\int_0^{+\infty}\frac{P_1(u)}{\left(s_j+ur\right)^2}du+\frac{1}{2s_j}+ O\left( \frac{1}{b}\right) \right)\\
&-\frac{b}{\pi r}\sum_{j=0}^{r-1}j\left(- \frac{1}{2t_j}+\frac{1}{2(s_j-b)} -\frac{1}{r}\log(t_j+rd_j)+\frac{1}{r}\log t_j+r\int_0^{+\infty}\frac{P_1(u)}{\left(t_j+ur\right)^2}du+
O\left( \frac{1}{b}\right) \right)\\
&-\frac{b}{ r}\sum_{j=0}^{r-1}j\int_0^{1/r}g_*\left( vr \right)dv+O(1).
\end{align*}
Thus
\[
\sum_{j=0}^{r-1}\frac{j}{r}\log(s_j+d_jr)=\left( \log b+O\left( \frac{1}{b}\right) \right)\sum_{j=0}^{r-1}\frac{j}{r},\tag{21}
\]
\[  
\sum_{j=0}^{r-1}\frac{j}{r}\log(t_j+d_jr)=\left( \log b+O\left( \frac{1}{b}\right) \right)\sum_{j=0}^{r-1}\frac{j}{ r},\tag{22}
\]
\[
\sum_{j=0}^{r-1}\frac{j}{s_j-b}=O\left( \frac{1}{b}\right) ,\tag{23}
\]
\[
\sum_{j=0}^{r-1}jO\left( \frac{1}{b}\right)=O\left( \frac{1}{b}\right) ,\tag{24}
\]
and
\[
\frac{b}{r}\sum_{j=0}^{r-1}j\int_0^{1/r}g_*\left( vr \right)dv=kb,\tag{25}
\]
where $k$ is a real constant depending only upon $r$. \\
By (21), (22), (23), (24) and (25), we obtain
\begin{align*}
c_0\left(\frac{r}{b}\right)&=\frac{1}{\pi r}b\log b-\frac{b}{\pi r}\log{2\pi}+\frac{b}{\pi r}\gamma+O(1)\\
&-\frac{b}{\pi r}\left( \log b+O\left(\frac{1}{b}\right) \right)\sum_{j=0}^{r-1}\frac{j}{r}+\frac{b}{\pi r^2}\sum_{j=0}^{r-1}j\log s_j\\
&+\frac{b}{\pi}\sum_{j=0}^{r-1}j\int_0^{+\infty}\frac{P_1(u)}{(s_j+ur)^2}du-\frac{b}{2\pi r}\sum_{j=0}^{r-1}\frac{j}{s_j}-\frac{b}{\pi r}\sum_{j=0}^{r-1}jO\left(\frac{1}{b}\right)\\
&+\frac{b}{2\pi r}\sum_{j=0}^{r-1}\frac{j}{t_j}-\frac{b}{2\pi r}O\left(\frac{1}{b}\right)+\frac{b}{\pi r}\left(\log b+O\left(\frac{1}{b}\right)  \right)\sum_{j=0}^{r-1}\frac{j}{r}\\
&-\frac{b}{\pi r^2}\sum_{j=0}^{r-1}j\log{t_j}-\frac{b}{\pi}\sum_{j=0}^{r-1}j\int_0^{+\infty}\frac{P_1(u)}{(t_j+ur)^2}du\\
&-\frac{b}{\pi r}O\left(\frac{1}{b}\right)-\frac{b}{r}\sum_{j=0}^{r-1}j\int_0^{1/r}g_*(vr)dv+O(1).
\end{align*}
Therefore,
$$c_0\left(\frac{r}{b}\right)=\frac{1}{\pi r}b\log b-\frac{b}{\pi r}(\log 2\pi -\gamma)+C_1\:b+O(1),$$
where
\begin{align*}
C_1&=\frac{1}{\pi r^2}\sum_{j=0}^{r-1}j\log\frac{s_j}{t_j}-\frac{1}{2\pi r}\sum_{j=0}^{r-1}j\left(\frac{1}{s_j}-\frac{1}{t_j}  \right)\\
&+\frac{1}{\pi}\sum_{j=0}^{r-1}j\int_0^{+\infty}P_1(u)\left(\frac{1}{(s_j+ur)^2}-\frac{1}{(t_j+ur)^2}  \right)du\\
&-\frac{1}{r}\sum_{j=0}^{r-1}j\int_0^{1/r}g_*(vr)dv,
\end{align*}
which by (15), (16) depends only on $r$ and $b_0$. This completes the proof of the theorem.
\end{proof}
%
%
%
%
%
%
%
%
\vspace{10mm}
\section{Moments of the cotangent sum $c_0(r/b)$ for fixed large $b$}
\vspace{5mm}
\noindent 
A crucial feature of the sum 
$$\sum_{l=0}^{d_j}\cot\left(\pi\frac{s_j+lr}{b}\right)$$
is the dominating influence of the terms
$$\cot\left(\pi\frac{s_j}{b}\right),$$
which are obtained for $l=0$, for small values of $s_j$. The cause of this fact is the singularity of the function $\cot x$ at $x=0$.\\
A similar influence is exercised by the terms with small values of $t_j$, caused by the singularity of $\cot x$ at $x=\pi$. Thus, these terms should be 
treated separately. The other terms may be expected to cancel, since 
$$\int_{\epsilon}^{\pi-\epsilon}\cot x\: dx=0,$$
coming from the functional equation 
$$\cot(\pi-x)=-\cot x.$$
Because of formula (15), that is 
$$s_j\equiv -bj\:(\bmod\: r)$$
and because of formula (16), that is 
$$t_j\equiv b(j+1)\:(\bmod\: r)$$
the quality of this cancelation will depend on good equidistribution properties of the fractions
$$\frac{jb}{r}\;(\bmod 1)$$
for $j$ ranging over short intervals. It is a well-known fact from Diophantine approximation that these equidistributions are only good if the fraction $b/r$ cannot be well approximated by fractions with small denominators. Lemma  \ref{x:41} provides a preparation for estimating the number of such values for $r$.\\ \\
Let $A_0$, $A_1$ be constants satisfying $1/2<A_0<A_1<1$. These constants will remain fixed throughout the section.\\
\noindent For $m\in\mathbb{N}$, let $\tilde{d}(m):=\tilde{d}(m,b)$ denote the number of divisors $r$ of $m$ that satisfy $$A_0b\leq r \leq A_1b,\ (r,b)=1.$$
\begin{lemma}\label{x:41}
Let $0< \delta \leq 1$, $\mathcal{L}_0=b\delta$, $(s,b)=1$ and $|s|\leq \mathcal{L}_0/2$. Then there exists a
fixed constant $M>0$ such that
$$ \sum_{l\leq \mathcal{L}_0}\tilde{d}(lb+s)\leq M\delta \phi(b),$$
where $\phi$ stands for the Euler totient function\index{function!Euler totient}.
\end{lemma}
\begin{proof}
Let $0<\Delta<1/2$. For $-1/2\leq u < 1/2$, let
\begin{eqnarray}
\chi(u;\Delta)=\left\{ 
  \begin{array}{l l}
   1\:, & \quad \text{if $u\in[-\Delta,\Delta]$}\vspace{2mm}\\ 
   0, & \quad \text{otherwise}\:.\\
  \end{array} \right.
\nonumber
\end{eqnarray}
We extend the definition of $\chi(u;\Delta)$ to all real numbers by requiring periodicity:
$$\chi(u+1;\Delta)=\chi(u;\Delta),\ \ \text{for all}\ u\in\mathbb{R}.$$
We set $\delta_*=4\delta$ and
$$\tilde{\chi}(u)=\delta_*^{-1}\int_0^{\delta_*}\chi(u;\delta_*+v)dv.$$
We obtain the Fourier expansion 
$$\tilde{\chi}(u)=\sum_{n=-\infty}^{+\infty} a(n)e(nu),$$
where 
\begin{eqnarray}
a(n)=\left\{ 
  \begin{array}{l l}
   \frac{\delta_*^{-1}}{4\pi^2n^2 }\left(e(2n\delta_*)-e(n\delta_*)-e(-n\delta_*)+e(-2n\delta_*) \right)\:, & \quad \text{if $n\neq 0$}\vspace{2mm}\\ 
   \frac{3}{2}\delta_*, & \quad \text{if $n=0$}\:\\
  \end{array} \right.
\nonumber
\end{eqnarray}
and $$e(u)=e^{2\pi i u},\ u\in\mathbb{R}.$$
We have
\[a(n)=\left\{ 
  \begin{array}{l l}
   O(\delta_*)\:, & \quad \text{if $|n|\leq \delta_*^{-1}$}\vspace{2mm}\\ 
   O(\delta_*^{-1}n^{-2}), & \quad \text{if $|n|>\delta_*^{-1}$}\:.\\
  \end{array} \right.
\tag{26}\]
Let $r$, $q$ be such that $lb+s=rq$. Then we obtain
$$rq\equiv s\:(\bmod b)$$
or equivalently
\[
q\equiv r^*s\:(\bmod b)\tag{27},
\]
where $r^*$ is defined by $rr^*\equiv 1\:(\bmod b).$\\
Now assume $A_0b\leq r \leq A_1b,\ (r,b)=1.$ It follows that $b/r<2$. Therefore, for 
$$rq=lb+s\leq 2\mathcal{L}_0b$$
it follows that
$$q\leq \frac{2\mathcal{L}_0b}{r}< 4\mathcal{L}_0=4b\delta$$
and thus
$$\frac{q}{b}< 4\delta=\delta_*.$$
Since $\chi(u;\delta_*+v)=1$ for $u\leq \delta_*$ and $v\geq 0$, we have
$$\chi(u;\delta_*)\leq \tilde{\chi}(u).$$
From (26) and (27) we have
$$\frac{q}{b}\equiv r^*\frac{s}{b}\:(\bmod 1).$$
From the periodicity of $\tilde{\chi}$ and its Fourier expansion, we obtain
\[
\sum_{l\leq\mathcal{L}_0}\tilde{d}(lb+s)\leq \sum_{\substack{r\: (\!\bmod b)\\ (r,b)=1}}\tilde{\chi}\left(\frac{sr^*}{b}\right)\leq \sum_{n=-\infty}^{+\infty}|a(n)|\:\vline\: \sum_{\substack{r\:(\!\bmod b)\\ (r,b)=1}}e\left(\frac{nsr^*}{b} \right)\:\vline\:. \tag{28}
\]
Making now use of the Ramanujan sum\index{Ramanujan sums}
$$c_q(n) =\sum_{\substack{r\:(\!\bmod q)\\ (r,q)=1}}e\left(\frac{nr^*}{q} \right)$$
we obtain from (28) the following
\[
\sum_{l\leq\mathcal{L}_0}\tilde{d}(lb+s)\leq \sum_{\substack{n=-\infty\\ n\neq 0}}^{+\infty}|a(n)|\:|c_b(ns)|+|a(0)|\phi(b). 
\tag{29}
\]
From the well-known formula
\[
c_q(n)=\sum_{d|(q,n)}\mu\left(\frac{q}{d} \right)d,
\]
(see \cite{IKW}, formula (3.2), p.44) and the fact that $(b,s)=1$, we obtain $(b,ns)=(b,n)$ and therefore
\[
c_b(ns)=\sum_{d|(b,ns)}\mu\left(\frac{b}{d} \right)d=\sum_{d|(b,n)}\mu\left(\frac{b}{d} \right)d=c_b(n).\tag{30}
\]
From \cite{IKW} (formula (3.5)), we have
$$|c_b(n)|\leq (b,n).$$
From this inequality and (30) we obtain
\begin{align*}
\sum_{l\leq\mathcal{L}_0}\tilde{d}(lb+s)&\leq \sum_{\substack{n=-\infty\\ n\neq 0}}^{+\infty}|a(n)|(b,n)+|a(0)|\phi(b). \tag{31}\\
&\leq M\delta\phi(b).
\end{align*}
\end{proof}
%
%

\noindent We now establish the equidistribution properties of the fractions 
$$\frac{jb}{r}\:(\bmod\:1).$$
We introduce a sequence of exceptional sets\index{sequence!of exceptional sets} $\mathcal{E}(m)$. The quality of the equidistribution of $jb/r\:(\bmod\:1),$ will be good for values of $r$
that do not belong to an exceptional set $\mathcal{E}(m)$ with a small number $m$.
\begin{lemma}$\label{x:35}$
Let $1/2<A_0<A_1<1$. Let $\theta\in\{1,-1\}$. Let $m_0$ be a sufficiently large positive real constant. Let 
$$m_0\leq m\leq 10\log\log b.$$
Then, for all values of $r$ such that $A_0b\leq r\leq A_1b$, $(b,r)=1$ which do not belong to an exceptional set $\mathcal{E}(m)$
with $$|\mathcal{E}(m)|=O\left(\phi(b)2^{-m}\right),$$ the following holds:\\ 
Let $U_1, U_2, j_1, j_2$ be real numbers such that $U_1\geq b^{-1}2^{5m}$, $U_2=U_1(1+\delta_1)$, $U_2\leq 1$, where $j_2-j_1\geq b2^{-(2m+1)}$,
$$2^{-m}\leq \delta_1\leq 2^{-m+1}.$$
Then we have,
$$\vline\left\{ j\: : \: j_1\leq j\leq j_2,\: \left\{\frac{\theta jb}{r}\right\}\in\left[U_1,U_2 \right] \right\} \vline= (j_2-j_1)\delta_1U_1\left(1+O\left(2^{-m} \right)\right).$$
\end{lemma}
\begin{proof}
Let $\mathcal{L}_0=b2^{-3m}.$ By the Dirichlet approximation theorem\index{theorem!Dirichlet approximation} (cf. \cite{Pra}, Satz 10.1) there exists $l\leq \mathcal{L}_0$, $l\in\mathbb{N}$ and $a\in\mathbb{N}$ with
$(a,l)=1$, such that
\[
\vline \frac{b}{r}-\frac{a}{l} \vline \leq \frac{1}{l\mathcal{L}_0}.\tag{32}
\]
Let $l_0$ be the smallest integer value of $l$ with the property (32). In the sequel, we denote by $\|y \|$ the distance of the real number $y$ to the nearest integer. From (32) it follows that 
\[
\left\| \frac{l_0b}{r}\right\|\leq\frac{1}{\mathcal{L}_0}.
\]
We first deal with the case that 
\[
\left\|\frac{l_0b}{r}\right\|\leq \mathcal{L}_0^{-1}2^{-m}.\tag{33}
\]
We set $$\eta=\eta(r)=\left\|\frac{l_0b}{r} \right\|.$$
From (32) it follows that
$$\vline\:\frac{l_0b}{r}-a\: \vline =\eta(r)$$
and thus, setting $s=r\eta(r)$, we obtain
$$l_0b-ar=\pm s\ \ \text{and thus}\ \ r\:|\: l_0b\mp s.$$
By (32), (33), $\mathcal{L}_0=b2^{-3m}$ and the inequality $rb^{-1}<1$ it follows that 
$$s\leq 2^{2m}.$$
Thus, the number of all possible values of $r$ satisfying (33) is at most
\[
\sum_{\substack{l\leq\mathcal{L}_0\\|s|\leq 2^{2m}}}\tilde{d}(lb+s).
\]
By Lemma $\ref{x:41}$ it follows that
$$\sum_{\substack{l\leq\mathcal{L}_0\\|s|\leq 2^{2m}}}\tilde{d}(lb+s)=O\left( 2^{-3m}\phi(b)2^{2m}\right)=O\left(\phi(b)2^{-m} \right).$$
Let now $r$ be such that 
\[
\mathcal{L}_0^{-1}2^{-m}< \left\|\frac{l_0b}{r} \right\| \leq \mathcal{L}_0^{-1}. \tag{34}
\]
Let $1\leq l_1\leq l_0$ and $\theta\in\{-1,1 \}$. We partition the set
$$\{j\: :\: j\in\mathbb{N},\: j\equiv \theta l_1(\bmod l_0)\}, $$
as follows:\\
Let 
$$m_1^{(1)}\leq m_1^{(2)}<m_2^{(1)}\leq m_2^{(2)}<\cdots<m_R^{(1)}\leq m_R^{(2)} ,$$
where $R$ depends upon $r,b,l_1$, namely $R=R(r,b,l_1)$, such that the fractional part
$$\left\{\frac{\theta(l_1+ml_0)b}{r}\right\}\in\left[ U_1,U_2\right]$$
for $$m_v^{(1)}\leq m\leq m_v^{(2)},\ \text{where}\ 1\leq v\leq R$$
and 
$$\left\{\frac{\theta(l_1+ml_0)b}{r}\right\}\not\in \left[ U_1,U_2\right], $$
otherwise.\\
We set
$$I(m)=\left[\left\{\frac{ml_0b}{r}\right\}, \left\{\frac{(m+1)l_0b}{r}\right\} \right].$$
The length of every interval $I(m)\:(\bmod 1)$ is $\eta(r)$ since 
$$\eta(r)=\left\|\frac{l_0b}{r} \right\|.$$
Additionally, by the definition of the sequence $\left(m_v^{(1)}\right)$, for $2\leq v\leq R-1$, the interval $I(m_v-1)$
must contain one of the two endpoints of the interval $ \left[ U_1,U_2\right]$. Thus
$$\vline\: \left\{\frac{\theta(l_1+m_vl_0)b}{r} \right\}-U_i \:\vline \leq \eta(r),$$
where $i\in\{1,2 \}$. We then also have
$$\vline\: \left\{\frac{\theta (l_1+m_{v+1}l_0)b}{r} \right\}-U_i \:\vline \leq \eta(r).$$
Thus, by the triangle inequality we obtain
$$\vline\: \left\{\frac{\theta(l_1+m_{v+1}l_0)b}{r} \right\}\: - \: \left\{\frac{\theta(l_1+m_vl_0)b}{r} \right\} \:\vline \leq 2\eta(r).$$
Since the intervals $I(m)$ are adjacent $\bmod 1$ and the union of intervals $I(m)$ for $m_v^{(1)}\leq m \leq m_v^{(2)}$ is the interval
$$\left[ \left\{\frac{\theta(l_1+m_{v+1}l_0)b}{r} \right\}, \left\{\frac{\theta(l_1+m_vl_0)b}{r} \right\} \right], $$
it follows that the union of the intervals $I(m)$ has total length equal to $1+O(\eta)$. 
Therefore, the number of these intervals $I(m)$ is 
\[
m_{v+1}^{(1)}-m_v^{(2)}=\eta(r)^{-1}+O(1),\tag{35}
\]
for all values of $v$ with $1\leq v\leq R-1.$\\
For $2\leq v\leq R-1$, the interval $ \left[ U_1,U_2\right]$ is covered by $(m_{v}^{(2)}-m_v^{(1)})+O(1)$ adjacent intervals $I(m)$. Hence, we have
\[
m_{v}^{(2)}-m_v^{(1)}=\eta(r)^{-1}(U_2-U_1)+O(1),\tag{36}
\]
for all values of $v$ with $2\leq v\leq R-1.$\\
For $v=1$, $v=R$, we obtain
\[
m_{v}^{(2)}-m_v^{(1)}\leq\eta(r)^{-1}(U_2-U_1)+O(1).\tag{37}
\]
By (34), we have
$$\mathcal{L}_0^{-1}2^{-m}< \eta(r)\leq \mathcal{L}_0^{-1},$$
where $\mathcal{L}_0=b2^{-3m}$ and thus
\[
\eta(r)^{-1}\geq b2^{-3m}.\tag{38}
\]
From the hypotheses for $U_1$, $\delta_1$ and $r$:
 $$U_1\geq b^{-1}2^{5m},\ \delta_1\geq 2^{-m},\ r> \frac{b}{2},$$ we obtain
\[
U_2-U_1\geq b^{-1}2^{4m}.\tag{39}
\]
From (38) and (39) we obtain
\[
\eta(r)^{-1}(U_2-U_1)\geq 2^{m}.\tag{40}
\]
From (36) combined with (40), we get
\[
m_{v}^{(2)}-m_v^{(1)}=\eta(r)^{-1}(U_2-U_1)\left(1+O\left( 2^{-m}\right)\right),\tag{41}
\]
for all values of $v$ with $2\leq v\leq R-1$.\\
For $v=1$, $v=R$, we obtain from(37) combined with (41) the following
\[
m_{v}^{(2)}-m_v^{(1)}\leq\eta(r)^{-1}(U_2-U_1)\left(1+O\left( 2^{-m}\right)\right).\tag{42}
\]
The interval $[j_1,j_2]$ is covered by $N$ complete residue systems $\bmod l_0$, where $$N=(j_2-j_1)l_0^{-1}+O(1).$$
Since $l_0\leq b2^{-3m}$, it follows that
\[
N=(j_2-j_1)l_0^{-1}\left(1+O\left( 2^{-m}\right)\right).\tag{43}
\]
Therefore by (41) we have
\[
R=\delta_1(j_2-j_1)l_0^{-1}\eta(r)\left(1+O\left( 2^{-m}\right)\right).\tag{44}
\]
By (41), (42) and (44) we obtain that there are
$$R\:\eta(r)^{-1}(U_2-U_1)\left(1+O\left( 2^{-m}\right)\right)=\delta_1(j_2-j_1)U_1l_0^{-1}\left(1+O\left( 2^{-m}\right)\right)$$
values of $j$ satisfying the relations
$$j_1\leq j\leq j_2,\ \left\{\frac{\theta jb}{r}\right\}\in\left[ U_1,U_2\right],\ j\equiv \theta l_1(\bmod\: l_0).  $$
We obtain the desired result of Lemma $\ref{x:35}$ by summing over all residue-classes $l_1(\bmod\: l_0)$.
\end{proof}
\noindent As a preparation for the study of the dominating terms 
$$\cot\left(\pi\frac{s_j}{b} \right),$$
we now investigate an inverse problem:\\
\textit{How are the values of $j$ distributed, if the value of $s_j$ is fixed?}\\
\noindent This requires the simultaneous localization of the values for $r$
and its multiplicative inverses $r^*\:(\bmod\: b)$. This localization will be accomplished via Fourier Analysis and upper bounds for Kloosterman sums.\index{Kloosterman sums}
\begin{lemma}\label{x:43}
Let $1/2<A_0<A_1<1$ and $r\in\mathbb{N}$. Let $\alpha\in (0,1)$, $\delta>0$ such that $\alpha+\delta<1$. We define 
$b^*=b^*(r,b)\in\mathbb{N}$ by $bb^*\equiv1\:(\bmod\: r)$ and $r^*=r^*(r,b)\in\mathbb{N}$ by $rr^*\equiv1\:(\bmod\: b)$. Then, we have
\begin{align*}
N(\alpha,\delta)&:=\:\vline\: \left\{r\::\: r\in\mathbb{N},\: (r,b)=1,\ A_0b\leq r\leq A_1b,\ \alpha\leq \frac{b^*}{r}\leq \alpha+\delta\right\}\:\vline\:   \\
&=\delta (A_1-A_0)\phi(b)(1+o(1)),\ \ (b\rightarrow +\infty).
\end{align*}
\end{lemma}
\begin{proof}
The Diophantine equation
$$bx+ry=1$$
has exactly one solution $(x_0,y_0)$ with 
$$-\left \lfloor \frac{r}{2} \right \rfloor < x_0\leq \left \lfloor \frac{r}{2} \right \rfloor,\ \ -\left \lfloor \frac{b}{2} \right \rfloor< y_0\leq \left \lfloor \frac{b}{2} \right \rfloor.$$
We have
\[
b^*\equiv x_0\:(\bmod\: r),\ \ r^*\equiv y_0\:(\bmod\: b)\tag{45}
\]
Therefore, for $\beta\in \left(-1/2,1/2\right]$ and $\delta>0$ with $\beta+\delta<1/2$ and $\beta-\delta>-1/2$ we have
\begin{align*}
&\:\vline\: \left\{r\::\: r\in\mathbb{N},\ (r,b)=1,\ A_0b\leq r\leq A_1b,\ \frac{y_0}{b}\in [\beta, \beta+\delta]\right\}\:\vline\:\tag{46} \\
&=\:\vline\: \left\{r\::\: r\in\mathbb{N},\ (r,b)=1,\ A_0b\leq r\leq A_1b,\ \frac{x_0}{r}\in [-(\beta+\delta), -\beta]\right\}\:\vline\:+O(1)\\
&=\:\vline\: \left\{r\::\: r\in\mathbb{N},\ (r,b)=1,\ A_0b\leq r\leq A_1b,\ \frac{b^*}{r}(\bmod 1)\in [-(\beta+\delta), -\beta]\right\}\:\vline\:+O(1),
\end{align*}
where $\frac{b^*}{r}(\bmod 1)\in [-(\beta+\delta), -\beta]$ stands for
\[\frac{b^*}{r}\in\left\{ 
  \begin{array}{l l}
   \left[-(\beta+\delta), -\beta\right]+1, & \quad \text{if $\beta\geq 0$}\vspace{2mm}\\ 
   \left[-(\beta+\delta), -\beta\right], & \quad \text{if $\beta <0$}\:.\\
  \end{array} \right.
\]
\noindent Let $\Delta>0$, such that $\beta+\delta+\Delta\leq 1/2$, $0\leq v\leq \Delta$.\\
\noindent We define the functions
\[\chi_1(u,v)=\left\{ 
  \begin{array}{l l}
   1, & \quad \text{if $u\in[\beta+\Delta-v,\beta+\delta-\Delta+v]$}\vspace{2mm}\\ 
   0, & \quad \text{otherwise}\:\\
  \end{array} \right.
\tag{47}\]
and
\[\chi_2(u,v)=\left\{ 
  \begin{array}{l l}
   1, & \quad \text{if $u\in[\beta-\Delta+v,\beta+\delta+\Delta-v]$}\vspace{2mm}\\ 
   0, & \quad \text{otherwise}\:\\
  \end{array} \right.
\tag{48}\]
as well as the function $l_1,l_2$ by
$$l_i(u)=\Delta^{-1}\int_0^{\Delta}\chi_i(u,v)dv,\ \ \text{for}\ i=1,2.$$
Let the function
\[\tilde{\chi}(r,\beta)=\left\{ 
  \begin{array}{l l}
   1, & \quad \text{if $\frac{r^*}{b}\in[\beta,\beta+\delta]$}\vspace{2mm}\\ 
   0, & \quad \text{otherwise}\:.\\
  \end{array} \right.
\tag{49}\]
Since $l_i$ for $i=1,2$ is obtained from $\chi_i$ by averaging over $v$ and since $0\leq \chi_i(u,v)\leq 1$ for $i=1,2$, it follows that $0\leq l_i(u)\leq 1$ for $i=1,2$.\\ 
From (47) we have 
$$l_1\left(\frac{r^*}{b}\right)=0,\ \text{if}\ \frac{r^*}{b}\not\in[\beta,\beta+\delta].$$
Similarly, from (48) we have
$$l_2\left(\frac{r^*}{b}\right)=1,\ \text{if}\ \frac{r^*}{b}\in[\beta,\beta+\delta].$$
Thus, we obtain
\[l_1\left(\frac{r^*}{b}\right)\leq \tilde{\chi}(r,\beta)\leq l_2\left( \frac{r^*}{b}\right)\tag{50}.\]
We have the Fourier expansions
$$l_i(u)=\sum_{n=-\infty}^{+\infty}a(n)e(nu),\ \text{for}\ i=1,2.$$
The Fourier coefficients $a(n)$ are computed as follows:\\
For $i=1$:
$$a(0)=\Delta^{-1}\int_0^{\Delta}\left(\int_{\beta-\Delta+v}^{\beta+\delta+\Delta-v}1\:du \right)dv=\delta+\Delta,$$
as well as
\begin{align*}
a(n)&=\Delta^{-1}\int_0^{\Delta}\left(\int_{\beta-\Delta+v}^{\beta+\delta+\Delta-v}e(-nu)\:du \right)dv\\
&=\Delta^{-1}\int_0^{\Delta}-\frac{1}{2\pi i n}\left[e(-n(\beta+\delta+\Delta-v)-e(-n(\beta-\Delta+v)) \right]dv\\
&=-\frac{1}{4\pi^2n^2}\Delta^{-1}\left(e(-n(\beta+\delta)-e(-n(\beta+\delta+\Delta))-e(-n\beta)+e(-n(\beta-\Delta))  \right).
\end{align*}
From the above and an analogous computation for $i=2$, we obtain
$$a(0)=\delta+R_1,\ \text{where}\ |R_1|\leq \Delta$$
and 
 \[a(n)=\left\{ 
  \begin{array}{l l}
   O(\Delta), & \quad \text{if $|n|\leq \Delta^{-1}$}\vspace{2mm}\\ 
   O(\Delta^{-1}n^{-2}), & \quad \text{if $|n|>\Delta^{-1}$}\:.\\
  \end{array} \right.
\tag{51}\]
Let $\Delta_1>0$, such that $A_0-\Delta_1>1/2$, $A_1+\Delta_1<1$ and $0\leq v\leq \Delta_1$. \\
We define the functions 
\[\chi_3(u,v)=\left\{ 
  \begin{array}{l l}
   1, & \quad \text{if $u\in[A_0+v-\Delta_1,A_1-v+\Delta_1]$}\vspace{2mm}\\ 
   0, & \quad \text{otherwise}\:\\
  \end{array} \right.
\tag{52}\]
and
\[\chi_4(u,v)=\left\{ 
  \begin{array}{l l}
   1, & \quad \text{if $u\in[A_0+\Delta_1-v,A_1+\Delta_1+v]$}\vspace{2mm}\\ 
   0, & \quad \text{otherwise}\:\\
  \end{array} \right.
\tag{53}\]
as well as the functions $l_3, l_4$ by 
$$l_i(u)=\Delta_1^{-1}\int_0^{\Delta_1}\chi_i(u,v)dv,\ \ \text{for}\ i=3,4.$$
Let the function
\[\chi^*(r,\beta)=\left\{ 
  \begin{array}{l l}
   1, & \quad \text{if $A_0\leq\frac{r}{b}\leq A_1$}\vspace{2mm}\\ 
   0, & \quad \text{otherwise}\:.\\
  \end{array} \right.
\tag{54}\]
Since $l_i$ for $i=3,4$ is obtained from $\chi_i$ by averaging over $v$ and since 
$$0\leq \chi_i(u,v)\leq 1,\ \  \text{for}\ \  i=3,4,$$ 
we obtain $0\leq l_i(u)\leq 1$ for $i=3,4$.\\
From(52) we have
$$l_3\left(\frac{r}{b} \right)=0,\ \text{if}\ \frac{r}{b}\not\in(A_0,A_1).$$
From (53), we have
$$l_3\left(\frac{r}{b} \right)=1,\ \text{if}\ \frac{r}{b}\in(A_0,A_1).$$
Therefore, we obtain
\[
l_3\left(\frac{r}{b}\right)\leq \chi^*(r,\beta)\leq l_4\left( \frac{r}{b}\right).\tag{55}
\]
An analogous computation as for $l_1,l_2$ gives the Fourier expansions
$$l_i(u)=\sum_{n=-\infty}^{+\infty}c(n)e(nu),\ \text{for}\ i=3,4,$$
with $$c(0)=A_1-A_0+R_2,\ \text{where}\ |R_2|\leq \Delta_1$$ and 
 \[c(n)=\left\{ 
  \begin{array}{l l}
   O(1), & \quad \text{if $|n|\leq \Delta_1^{-1}$}\vspace{2mm}\\ 
   O(\Delta_1^{-1}n^{-2}), & \quad \text{if $|n|>\Delta_1^{-1}$}\:.\\
  \end{array} \right.
\tag{56}\]
From (46), (49), (50), (54) and (55), setting $\beta=-\alpha$, we get the following
\[
\sum_{\substack{r=1\\ (r,b)=1}}^{b-1}l_1\left(\frac{r^*}{b} \right)l_3\left(\frac{r}{b}\right)\leq N(\alpha,\delta)\leq \sum_{\substack{r=1\\ (r,b)=1}}^{b-1}l_2\left(\frac{r^*}{b} \right)l_4\left(\frac{r}{b}\right)\tag{57}.
\]
Therefore
\begin{align*}
\sum_{\substack{r=1\\ (r,b)=1}}^{b-1}l_1\left(\frac{r^*}{b} \right)l_3\left(\frac{r}{b}\right)&=\sum_{m,n=-\infty}^{+\infty}a(m)c(n)\sum_{\substack{r=1\\ (r,b)=1}}^{b-1}e\left( \frac{nr+mr^*}{b}\right)\\
&=\sum_{\substack{m,n=-\infty \\ (m,n)\neq (0,0)}}^{+\infty}a(m)c(n)K(n,m,b)+a(0)c(0)\phi(b),\tag{58}
\end{align*}
with the Kloosterman sums \index{Kloosterman sums}
$$K(n,m,b)=\sum_{\substack{r=1\\ (r,b)=1}}^{b-1}e\left( \frac{nr+mr^*}{b}\right)$$
for both $n\neq 0$, $m\neq 0 $, and the Ramanujan sums\index{Ramanujan sums}
$$K(n,0,b)=\sum_{\substack{r=1\\ (r,b)=1}}^{b-1}e\left( \frac{nr}{b}\right) $$
and
$$K(0,m,b)=\sum_{\substack{r=1\\ (r,b)=1}}^{b-1}e\left( \frac{mr^*}{b}\right).$$
We have the Weil bound
$$|K(n,m,b)|\leq\tau(b)(n,m,b)^{1/2}\:\sqrt{b},$$
(\cite{IKW}, p. 19, Formula 1.60), and the elementary bound
$$|K(n,0,b)|\leq (n,b) \ \ \text{and}\ \ |K(0,m,b)|\leq (m,b) $$
(\cite{IKW}, p. 45, Formula 3.5).\\
We obtain from (58) the following 
\[
\sum_{\substack{r=1\\ (r,b)=1}}^{b-1}l_1\left(\frac{r^*}{b} \right)l_3\left(\frac{r}{b}\right)=(\delta+R_1)(A_1-A_0+R_2)\phi(b)+o(\phi(b)),\ \ (b\rightarrow+\infty)\tag{59},
\]
where
\[
\sum_{\substack{m,n=-\infty \\ (m,n)\neq (0,0)}}^{+\infty}a(m)c(n)K(n,m,b)=o(\phi(b))
\]
for $|R_1|\leq \Delta$ and $|R_2|\leq \Delta_1$.\\ By the same computation we also get
\[
\sum_{\substack{r=1\\ (r,b)=1}}^{b-1}l_2\left(\frac{r^*}{b} \right)l_4\left(\frac{r}{b}\right)=(\delta+R_1)(A_1-A_0+R_2)\phi(b)+o(\phi(b)),\ \ (b\rightarrow+\infty),\tag{60}
\]
for $|R_1|\leq \Delta$ and $|R_2|\leq \Delta_1$. Therefore
\begin{align*}
\sum_{\substack{r=1\\ (r,b)=1}}^{b-1}l_2\left(\frac{r^*}{b} \right)l_4\left(\frac{r}{b}\right)&=\delta(A_1-A_0)\phi(b)+\delta R_2\phi(b)\tag{61}\\
&+R_1(A_1-A_0)\phi(b)+R_1R_2\phi(b)+o(\phi(b)).
\end{align*}
Since $\Delta$ and $\Delta_1$ can be chosen to be arbitrarily small, it follows that (61) implies Lemma \ref{x:43}.
\end{proof}
\noindent By the use of Lemma \ref{x:43} we shall prove that the sum
$$\sum_{|s_j|\leq 2^{m_1}}\cot\left(\pi\frac{s_j}{b}\right)$$
is related to the sum $f(x;m_1)$, which we define and investigate in the next two lemmas. 
\begin{lemma}\label{x:l111}
Let
$$f(x;m_1)=\sum_{l=1}^{2^{m_1}}\frac{B(lx)}{l},$$
where $B(x)=1-2\{x\}$. Then, for $L\in\mathbb{N}$ there are numbers $a(k,L)\in\mathbb{R}$ with
$$a(k,L)=a(k,L,m_1)=O_{\epsilon}(|k|^{-1+\epsilon}),$$ 
where the implied constant is independent from $m_1$,
such that
$$\lim_{N\rightarrow +\infty} \left\| f(x;m_1)^L-\sum_{k=-N}^Na(k,L)e(kx)\right\|_2=0.$$
If $m_2>m_1$, then we have
$$a(k,m_1)=a(k,m_2),\ \ \text{for}\ |k|\leq 2^{m_1}.$$
\end{lemma}
\begin{proof}
We shall prove the statement by induction on $L$.\\
For $L=1$, we have
$$f(x;m_1)=\sum_{l=1}^{2^{m_1}}\frac{B(lx)}{l}.$$
By the Fourier expansion of $B(x)$, we have
$$B(x)=-\frac{i}{\pi}\sum_{\substack{n=-\infty\\ n\neq 0}}^{+\infty}\frac{e(nx)}{n},\ \textit{in}\ L^2.$$
Therefore
$$\lim_{N\rightarrow +\infty} \left\| B(x)+\frac{i}{\pi}\sum_{\substack{|n|\leq N\\ n\neq 0}}\frac{e(nx)}{n}\right\|_2=0.$$
Thus, we also have:
$$\lim_{N\rightarrow +\infty} \left\| f(x;m_1)+\sum_{l=1}^{2^{m_1}}\frac{1}{l}\sum_{\substack{|n|\leq N\\ n\neq 0}}\frac{e(lnx)}{n}\right\|_2=0.$$
We write
$$\frac{i}{\pi}\sum_{l=1}^{2^{m_1}}\frac{1}{l}\sum_{|n|\leq N}\frac{e(lnx)}{n}=\sum_{k=-\infty}^{+\infty}c(k,N)e(kx)$$
and observe that $c(k,N')=c(k,N)$ for all $N'\geq N$.\\
Let $N_0(k)$ be the smallest value of $N$, such that 
$$c(k,N')=c(k,N_0(k)),\ \ \text{for all}\ N'\geq N_0(k).$$
Then we define
$$a(k,1)=c(k,N_0(k))$$
and obtain
$$\lim_{N\rightarrow +\infty} \left\| f(x;m_1)-\sum_{|n|\leq N}a(k,1)e(kx)\right\|_2=0,$$
where
$$a(k,1)=O\left(k^{-1+\epsilon} \right).$$
Since in the definition of $c(k,N)$ there appear only pairs $(l,n)$ with $|ln|=|k|$, for $|k|\leq 2^{m_1}$ the value of $c(k,N)$ will be the same for
$f(x;m_1)$ and $f(x;m_2)$ for $m_2>m_1$.\\
For the induction step from $L$ to $L+1$, we have
\[
f(x;m_1)^L=\sum_{k=-N}^Na(k,L,m_1)e(kx)+R_1(x,N,L)\tag{*}
\]
\[
f(x;m_1)=\sum_{k=-N}^Na(k,1,m_1)e(kx)+R_1(x,N,1),\tag{**}
\]
where 
$$\lim_{N\rightarrow +\infty}\left\| R_1(x,N,L)\right\|_2=0\ \ \text{and}\ \ \lim_{N\rightarrow +\infty}\left\| R_1(x,N,1)\right\|_2=0.$$
Also
$$a(k,L,m_1)=O_{\epsilon}\left(k^{-1+\epsilon}\right).$$
If $m_2>m_1$, then we have:
\[a(k,L,m_1)=a(k,L,m_2),\ \ \text{for}\ |k|\leq 2^{m_1}.\tag{***}\]
We define $b(k,L,N)$ by 
$$b(k,L,N)=\sum_{\substack{(k_1,k_2)\::\:k_1+k_2=n\\ |k_i|\leq N}}a(k_1,L)a(k_2,1).$$
We split the above sum into partial sums as follows:
$$b(k,L,N)=\sum{}_{{}_I}+\sum{}_{{}_{II}}+\sum{}_{{}_{III}}+\sum{}_{{}_{IV}},$$
where 
\begin{align*}
\sum{}_{{}_{I}}&=\sum_{k_1=1}^k{}^{'}a(k_1,L)a(k-k_1,1),\\
\sum{}_{{}_{II}}&=\sum_{k_1=-k}^{-1}{}^{'}a(k_1,L)a(k-k_1,1),\\
\sum{}_{{}_{III}}&=\sum_{j=0}^{+\infty}{}^{'}\:\sum_{2^{^j}k< k_1\leq 2^{^{j+1}}k}a(k_1,L)a(k-k_1,1),\\
\sum{}_{{}_{IV}}&=\sum_{j=0}^{+\infty}{}^{'}\:\sum_{-2^{^{j+1}}k\leq k_1\leq -2^{^j}k}a(k_1,L)a(k-k_1,1).
\end{align*}
where $\sum{}^{'}$ stands for the condition $|k_1|\leq N,$ and $|k-k_1|\leq N$.\\
\textit{Estimation of the sums $\sum{}_{{}_{I}}$, $\sum{}_{{}_{II}}$:}\\
We have
\begin{align*}
\sum{}_{{}_{I}}&=\sum_{k_1=1}^k{}^{'}a(k_1,L)a(k-k_1,1)\\
&=a(1,L)a(k-1,1)+a(2,L)a(k-2,1)+\cdots+a(k-1,L)a(1,1)\\
&=O_{\epsilon}\left( k^{2\epsilon}\left( 1\cdot\frac{1}{k-1}+\frac{1}{2}\frac{1}{k-2}+\cdots+\frac{1}{k-1}\cdot1 \right)\right).
\end{align*}
However,
\begin{align*}
1\cdot\frac{1}{k-1}+\frac{1}{2}\frac{1}{k-2}+\cdots+\frac{1}{k-1}\cdot1&\leq 2\left( 1\cdot \frac{1}{k-1}+\cdots+\frac{1}{\left(\left \lfloor \frac{k}{2} \right \rfloor+1 \right)\left(k-\left(\left \lfloor \frac{k}{2} \right \rfloor+1 \right)  \right)} \right)\\
&=O\left(\frac{1}{k}\sum_{k_1=1}^{\left \lfloor \frac{k}{2} \right \rfloor+1 }\frac{1}{k_1} \right)=O\left(\frac{\log k}{k} \right).
\end{align*}
Therefore, we obtain
$$\sum{}_{{}_{I}}=O_{\epsilon}(k^{-1+\epsilon}),\ \text{for every}\ \epsilon>0.$$
Similarly, we get
$$\sum{}_{{}_{II}}=O_{\epsilon}(k^{-1+\epsilon}),\ \text{for every}\ \epsilon>0.$$
\textit{Estimation of the sums $\sum{}_{{}_{III}}$, $\sum{}_{{}_{IV}}$:}\\
For $j=1$, we obtain the same estimates as for $\sum{}_{{}_{I}}$, $\sum{}_{{}_{II}}$, by similar arguments.\\
For fixed $j\geq 2$, we have
\begin{align*}
\sum{}_{{}_{III,j}}&=\sum_{2^{^j}k< k_1\leq 2^{^{j+1}}k}a(k_1,L)a(k-k_1,1)\\
&=a(2^{^j}k+1,L)a(k-(2^{^j}k+1),1)+a(2^{^j}k+2,L)a(k-(2^{^j}k+2),1)\\
&+\cdots+a(2^{^{j+1}}k-1,L)a(k-(2^{^{j+1}}k+2),1)
\end{align*}
and by the induction hypothesis we get
\begin{align*}
\sum{}_{{}_{III,j}}&=O_{\epsilon}\left( (2^{^j}k)^{\epsilon}\left(\frac{1}{(2^{^j}k+1)(k-(2^jk+1))}+\cdots+\frac{1}{(2^{^{j+1}}k-1)(k-(2^{^{j+1}}k+2))}  \right)  \right)\\
&=O_{\epsilon}\left( (2^{^j}k)^{\epsilon}2^{^j}k\frac{1}{2^{^{2j}}k^2}\right)=O_{\epsilon}\left( (2^{^j}k)^{-1+\epsilon}\right).
\end{align*}
Hence, we obtain
$$\sum{}_{{}_{III}}=\sum_{j=0}^{+\infty}\sum{}_{{}_{III,j}}=\sum_{j=0}^{+\infty}O_{\epsilon}\left( (2^{^j}k)^{-1+\epsilon}\right)=O_{\epsilon}(k^{-1+\epsilon}).$$
Similarly, we get
$$\sum{}_{{}_{IV}}=O_{\epsilon}(k^{-1+\epsilon}).$$
From (*)and (**), we get:
\[
f(x;m_1)^{L+1}=\sum_{k=-N}^Nb(k,L,N)e(kx)+R_3(x,N,L+1)
\]
where 
$$\lim_{n\rightarrow+\infty}\left\| R_3(x,N,L+1)\right\|_2=0.$$
We now set 
$$a(k,L+1):=\sum_{(k_1,k_2)\::\: k_1+k_2=k}a(k_1,L)a(k_2,1).$$
We may estimate the difference
$$|b(k,L,N)-a(k,L+1)|$$
by considering the sums $\tilde{\sum}{}_{{}_{I}}$, $\tilde{\sum}{}_{{}_{II}}$, $\tilde{\sum}{}_{{}_{III}}$, $\tilde{\sum}{}_{{}_{IV}}$ defined as 
$\sum{}_{{}_{I}}$, $\sum{}_{{}_{II}}$, $\sum{}_{{}_{III}}$, $\sum{}_{{}_{IV}}$, but the condition $|k_1|\leq N$, $|k_2|\leq N$ replaced by $|k_1|>N$ or 
$|k_2|>N$.\\
The induction statement for $L+1$ now follows from the definitions of $b(k,L,N)$ and $a(k,L+1)$ by letting $N\rightarrow+\infty$.
\end{proof}
\begin{lemma}\label{x:l222}
For $f(x;m_1)$ defined as in the previous lemma, we have that the limit 
$$\lim_{m_1\rightarrow+\infty}\int_0^1f(x;m_1)^Ldx$$
exists.
\end{lemma}
\begin{proof}
Let $m_2\in\mathbb{N}$ with $m_2>m_1$. We have
$$\left|f(x;m_2)^L-f(x;m_1)^L\right|=\left| f(x;m_2)-f(x;m_1) \right|\cdot \left| \sum_{h=0}^{L-1}f(x;m_2)^hf(x;m_1)^{L-h-1} \right|.$$
Therefore, by the Cauchy-Schwarz inequality\index{inequality!Cauchy-Schwarz} we obtain
\begin{align*}
&\int_0^1\left|f(x;m_2)^L-f(x;m_1)^L\right| dx\\
&=O_L\left(\left( \int_0^1\left| f(x;m_2)-f(x;m_1) \right|^2 dx\right)^{1/2} \left( \int_0^1\left| \sum_{h=0}^{L-1}f(x;m_2)^hf(x;m_1)^{L-h-1} \right|^2 \right)^{1/2}  \right).
\end{align*}
However, by Parseval's identity\index{Parseval's identity} we have
\begin{align*}
\int_0^1\left| f(x;m_2)-f(x;m_1) \right|^2 dx&=\sum_{k=-\infty}^{+\infty}(a(k,1,m_2)-a(k,1,m_1))^2\\
&\leq \sum_{|k|>2^{m_1}}(a(k,1,m_1)^2+a(k,1,m_2)^2),
\end{align*}
since by (***) in the proof of the previous Lemma we have $$a(k,1,m_1)=a(k,1,m_2),\ \text{for}\ \  |k|\leq 2^{m_1}.$$
Now, due to Lemma \ref{x:l111}, we obtain
\begin{align*}
\int_0^1\left| f(x;m_2)-f(x;m_1) \right|^2 dx&=O_{\epsilon}\left( \sum_{|k|>2^{m_1}}k^{-(2+2\epsilon)} \right)\\
&=O_{\epsilon}\left( 2^{m_1(-1+2\epsilon)}\right).
\end{align*}
Additionally, we get
$$\int_0^1\left( \sum_{h=0}^{L-1}f(x;m_2)^hf(x;m_1)^{L-h-1}\right)^2dx=O_L\left(\sum_{h=0}^{L-1}\left(\int_0^1f(x;m_2)^{4h}dx+\int_0^1f(x;m_1)^{4h}dx \right)  \right),$$
since
$$\left|a^hb^{L-h-1} \right|\leq a^{2h}+b^{2(L-h-1)},\ \text{for every} \ a,b\in\mathbb{R}.$$
Therefore
\begin{align*}
&\int_0^1\left( \sum_{h=0}^{L-1}f(x;m_2)^hf(x;m_1)^{L-h-1}\right)^2dx\\
&=O_L\left( \max_{0\leq h\leq L-1} \left\{ \left(\int_0^1f(x;m_2)^{4h}dx \right)^{1/2}+ 
 \left(\int_0^1f(x;m_1)^{4h}dx \right)^{1/2}\right\}  \right) .
 \end{align*}
 Then by Parseval's identity\index{Parseval's identity} it follows that
 \begin{align*}
&\int_0^1\left( \sum_{h=0}^{L-1}f(x;m_2)^hf(x;m_1)^{L-h-1}\right)^2dx\\
&=O_L\left( \max_{0\leq h\leq L-1} \left\{ \left(\sum_{k=-\infty}^{+\infty}a(k,4h,m_2)^2 \right)^{1/2}+ 
 \left(\sum_{k=-\infty}^{+\infty}a(k,4h,m_1)^2 \right)^{1/2}\right\}  \right) .
 \end{align*}
By Lemma \ref{x:l111} we know that
$$a(k,4h,m_i)=O_{\epsilon}(k^{-1+\epsilon}),\ i=1,2,$$
where the implied constant is independent from $m_i$.\\
Hence, by the above estimate we derive the following inequality
$$\int_0^1\left( \sum_{h=0}^{L-1}f(x;m_2)^hf(x;m_1)^{L-h-1}\right)^2 dx \leq C(\epsilon,L),$$
which implies
$$\left\| f(.;m_1)^L-f(.;m_2)^L\right\|_1\leq C'(\epsilon,L),$$
where $C(\epsilon,L)$, $C'(\epsilon,L)$ are positive constants that depend at most on $\epsilon$ and $L$, but not on $m_1$ or $m_2$.\\
From the above estimates it follows that the sequence of functions $(f(x;m_1)^L)_{m_1\geq1}$ forms a Cauchy-sequence\index{sequence!Cauchy} in the space $L^1([0,1])$ of the integrable functions defined over $[0,1]$.\\
Since $L^1([0,1])$ is a complete metric space it follows that there exists a limit function $w(x)\in L^1([0,1])$, such that
$$f(.;m_1)^L\rightarrow w,\ \text{in}\ L^1,\ \text{as}\ m_1\rightarrow +\infty.$$
Then
$$\int_0^1f(x;m_1)^Ldx \rightarrow \int_0^1w(x)dx,\ \text{as}\ m_1\rightarrow +\infty,$$
which completes the proof of the lemma.
\end{proof}
\begin{lemma}\label{x:tene}
For $x\in\mathbb{R}$, let 
$$g(x):=\sum_{l=1}^{+\infty}\frac{1-2\{lx\}}{l}.$$
Then for each $x\in\mathbb{Q}$ the series $g(x)$ converges. \\
For $x\in\mathbb{R}\setminus\mathbb{Q}$, the series $g(x)$ converges if and only if the series 
$$\sum_{m\geq 1}(-1)^m\frac{\log q_{m+1}}{q_m}$$
converges, where $(q_m)_{m\geq 1}$ denotes the sequence of partial denominators\index{sequence!of partial denominators of the continued fraction expansion} of the continued fraction expansion of $x$.
\end{lemma}
\begin{proof} The statement of the lemma is part of Th\'eor\`eme 4.4 of the paper by R. de la Br\'et\`eche\index{de la Br\'et\`eche, R.} and G. Tenenbaum\index{Tenenbaum, G.} in \cite{bre}.
\end{proof}
\textit{Remark:} One can show that the series $g(x)$ can also be written in the form (see \cite{bre}) 
$$-\sum_{l=1}^{+\infty}\frac{\tau(l)}{\pi l}\sin(2\pi lx),$$
where it converges, and $\tau(l)$ stands for the divisor function.\\
In the following, we will prove that the series $g(x)$ converges almost everywhere.
\begin{definition}\label{x:cfra}
Let $\alpha\in [0,1)$ be an irrational number and $\alpha=[a_0;a_1,a_2,\ldots ]$ be the continued fraction expansion of $\alpha$.
We denote the $n$-th convergent of $\alpha$ by $p_n/q_n$.
\end{definition}
\begin{lemma}\label{x:tenel1}
Let $1<K<\sqrt{2}.$ Then there is a positive constant $c_0=c_0(K)$, such that 
$$q_n\geq c_0K^n,$$
for every $n\in\mathbb{N}.$
\end{lemma}
\begin{proof}
We have 
\[p_n=a_np_{n-1}+p_{n-2},\ p_{-1}=1,\ p_{-2}=0\tag{I}\]
and
\[q_n=a_nq_{n-1}+q_{n-2},\ q_{-1}=0,\ q_{-2}=1.\tag{II}\]
From (II) it follows that 
$$q_n\geq 2q_{n-2},$$
for every $n\in\mathbb{N}.$ By induction on $k\in\mathbb{N}$ we conclude that 
\[q_{2k}\geq q_02^k,\tag{III}\]
for every $k\in\mathbb{N}$. From (III) the proof of the lemma follows.
\end{proof}
\begin{lemma}\label{x:tenel2}
Let $\mathcal{F}_n\subseteq [0,1)$, $n\in\mathbb{N}$, be Lebesgue measurable sets such that 
$$\mathcal{F}_1\supseteq \mathcal{F}_2 \supseteq \mathcal{F}_3 \supseteq \cdots \supseteq \mathcal{F}_n \supseteq \mathcal{F}_{n+1} \supseteq \cdots$$
Assume that 
$$\sum_{i=1}^{+\infty}\text{meas}\left(\mathcal{F}_i \right)<+\infty.$$
Then we have 
$$\text{meas}\left\{\alpha\in[0,1)\::\: \alpha\in\mathcal{F}_i\ \text{for infinitely many values of}\ i\in\mathbb{N} \right\}=0.$$
\end{lemma}
\begin{proof}
This is the Borel-Cantelli lemma (cf. \cite{klen}, \cite{prok}).
\end{proof}
\begin{definition}\label{x:del}
Let $q\in\mathbb{N}$, $\delta>0$ and 
$$\Delta(q):=\exp(-q^{\delta}).$$
Then we define the set
$$\mathcal{E}(q,\delta):=\bigcup_{\substack{0\leq a\leq q\\ a\in\mathbb{Z}}}\left[\frac{a}{q}-\Delta(q),\frac{a}{q}+\Delta(q)\right].$$ 
\end{definition}
\begin{definition}
Let $L>1$. Then we define the set 
$$\mathcal{E}(L):=\left\{ \alpha\in[0,1)\::\: \frac{\log q_{m+1}}{q_m}\geq L^{-m}\ \text{for infinitely many values of}\ m\in\mathbb{N}\right\}.$$
\end{definition}
\begin{lemma}\label{x:tenel4}
There is a constant $L_0>1$, such that $$\text{meas}\:\mathcal{E}(L)=0$$ whenever $1<L\leq L_0$. 
\end{lemma}
\begin{proof} By Lemma \ref{x:tenel1} we have for $1< K< \sqrt{2}$:\\
\[q_m\geq c_0K^m.\tag{IV}\]
Let $0<\delta<1$. From (IV) we obtain
\[ q_m^{1-\delta}\geq c_0^{1-\delta}\left(K^{1-\delta}\right)^m\tag{V}\]
If we choose $L_0$ with 
$$1<L_0<K^{1-\delta},$$
we get for all real values of $L$ with $1<L\leq L_0$ the following
\[ q_m^{1-\delta}\geq L^m,\ \text{for}\ m\geq m_0, \tag{VI}\]
where $m_0$ is a sufficiently large positive integer.\\
From (VI) we obtain 
\[ L^{-m}q_m\geq q_m^{\delta}. \tag{VII}\]
Let now $\alpha\in\mathcal{E}(L)$ and $m\geq m_0$ such that 
\[ \frac{\log q_{m+1}}{q_m}\geq L^{-m}. \tag{VIII}\]
We have
$$\frac{p_{m+1}}{q_{m+1}}-\frac{p_m}{q_m}=\frac{(-1)^{m+1}}{q_mq_{m+1}}$$
(cf. \cite{rock}).\\
Since $\alpha$ lies between $p_m/q_m$ and $p_{m+1}/q_{m+1}$ (cf. \cite{rock}) we have by (VIII) the following
$$\left| \alpha-\frac{p_m}{q_m}\right|\leq \frac{1}{q_mq_{m+1}}\leq \frac{1}{q_m\exp(L^{-m}q_m)}$$
and by (VII) we obtain
$$\left| \alpha-\frac{p_m}{q_m}\right| \leq \exp(-q_m^{\delta}).$$
Thus by Definition \ref{x:del}, it follows that 
$$\alpha\in\mathcal{E}(q_m,\delta).$$
By Lemma \ref{x:tenel2}, we therefore have 
\begin{align*}
&\text{meas}\left\{\alpha\in [0,1)\::\: \frac{\log q_{m+1}}{q_m} \geq L^{-m}\ \text{for infinitely many values of}\ m\in\mathbb{N}\right\}\\
&\leq \text{meas}\left\{\alpha\in [0,1)\::\: \alpha\in\mathcal{E}(q,\delta)\ \text{for infinitely many values of}\ q\in\mathbb{N}\right\}\\
&=0.
\end{align*}
\end{proof}
\begin{lemma}\label{x:almostev} 
The series
$$g(\alpha)=\sum_{l=1}^{+\infty}\frac{1-2\{l\alpha\}}{l}$$
converges almost everywhere in $[0,1)$.
\end{lemma}
\begin{proof}
By Lemma \ref{x:tene} the series $g(\alpha)$ converges for each $\alpha\in [0,1)$ such that $\alpha\in\mathbb{Q}$ or $\alpha\in\mathbb{R}\setminus\mathbb{Q}$ and the series
\[\sum_{m\geq 1}(-1)^m\frac{\log q_{m+1}}{q_m}\tag{VI}\]
converges. The series (VI) converges if there exist $m_0\in\mathbb{N}$ and $L>1$, such that 
\[\log q_{m+1} < L^{-m}q_m\ \text{for}\ m\geq m_0\:.\tag{VII}\]
By Lemma \ref{x:tenel4}, (VII) holds almost everywhere in $[0,1)$. This completes the proof of the lemma.
\end{proof}
\textit{Remark:} The convergence of the series (VI) follows from the convergence of the series
$$\sum_{m\geq 1}\frac{\log q_{m+1}}{q_m},$$
which is the defining property of the Brjuno numbers\index{Brjuno numbers}. The set of these numbers is known to have measure 1.
\begin{theorem}\label{x:positive}
Let 
$$D_L:=\lim_{m_1\rightarrow+\infty}\int_0^1f(x;m_1)^Ldx.$$
For $k\in\mathbb{N}$, we have
$$D_{2k}=\int_0^1g(x)^{2k}dx,\ \ \text{as well as}\ \ D_{2k}>0.$$
\end{theorem}
\begin{proof}
Since the sequence $(f(x;m_1))_{m_1\geq 1}$ forms a Cauchy-sequence in the space $L^1([0,1])$, as it was shown in the proof of Lemma \ref{x:l222}, there exists a limit function $w(x)\in L^1([0,1])$ such that 
$$\lim_{m_1\rightarrow +\infty} \left\| f(.\:;m_1)-w(.)  \right\|_1=0.$$
On the other hand we have
$$f(x;m_1)\rightarrow g(x),\ \text{almost everywhere, as}\ m_1\rightarrow +\infty.$$
A subsequence $f(x;\nu_k)$ of the sequence $(f(x;m_1))_{m_1\geq 1}$, $\nu_k\rightarrow +\infty$, as $k\rightarrow +\infty$, converges almost everywhere 
to $w$. Therefore, $g(x)=w(x)$, almost everywhere.\\
Thus, there exists a function $w_L\in L^1([0,1])$ such that
$$f(.;m_1)^L\rightarrow w_L,\ \text{in}\ L^1$$
and so 
$$w_L(x)=g(x)^L,\ \text{almost everywhere}.$$
Hence
$$\int_0^1f(x;m_1)^Ldx \rightarrow \int_0^1w_L(x)dx=\int_0^1 g(x)^L dx.$$
Since not all Fourier coefficients of $w(x)$ are equal to zero, we obtain 
$$\int_0^1g(x)^2 dx>0,$$
and therefore we get
$$D_{2k}=\int_0^1g(x)^{2k}dx>0.$$

\end{proof}

In the following we will study the moments of the sums $Q(r/b)$, which are related to the sums $c_0(r/b)$ by Proposition \ref{x:vasiko}:
$$c_0\left(\frac{r}{b}\right)=\frac{1}{r}\:c_0\left(\frac{1}{b}\right)-\frac{1}{r}Q\left(\frac{r}{b}\right)\:.$$
Here the term 
$$\frac{1}{r}\:c_0\left(\frac{1}{b}\right)$$
provides only a small contribution, since by Theorem \ref{x:rassiass} we have:
$$c_0\left(\frac{1}{b}\right)=O(b\log b).$$
Thus, properties of the moments 
$$\sum_{\substack{r: (r,b)=1\\ A_0b\leq r\leq A_1b }}  c_0\left(\frac{r}{b}\right)^L $$
can easily be extracted from properties of the moments 
$$\sum_{\substack{r: (r,b)=1\\ A_0b\leq r\leq A_1b }}  Q\left(\frac{r}{b}\right)^L $$
by partial summation.\\
\noindent For the treatment of the sum 
$$\sum_{\substack{r: (r,b)=1\\ A_0b\leq r\leq A_1b }}  Q\left(\frac{r}{b}\right)^L $$
we make use of the preparations made in Lemmas \ref{x:35} and \ref{x:43}.\\
\noindent From the sum 
$$Q\left( \frac{r}{b}\right)=\sum_{j=0}^{r-1} j\sum_{l=0}^{d_j}\cot\left(\pi\frac{s_j+lr}{b}\right)  $$ 
we split off the terms with $l=0$ and small values of $s_j$ as well as the terms with small values of $t_j$. 
The resulting sum which provides the main contribution is approximated by the sum $Q(r,b,m_1)$ defined by formula (66), which
depends on 
$$\alpha=\frac{b^*}{r}.$$
We shall use the localization of $\alpha=b^*/r$ established in Lemma \ref{x:43}. For the remaining terms of the sum $Q(r/b)$ we make use of
their cancelation, using the results of Lemma \ref{x:35}.
\begin{theorem}$\label{x:44}$
Let $k\in\mathbb{N}$ be fixed. Let also $A_0$, $A_1$ be fixed constants such that $1/2<A_0<A_1<1$. Then there exists a constant $E_k>0$, depending only on $k$, such that
$$ \sum_{\substack{r:(r,b)=1\\A_0b\leq r\leq A_1b}}Q\left( \frac{r}{b}\right)^{2k}=E_k\cdot(A_1^{2k+1}-A_0^{2k+1})b^{4k}\phi(b)(1+o(1)),\ \ (b\rightarrow+\infty),$$
with 
$$E_k=\frac{D_{2k}}{(2k+1)\pi^{2k}}.$$
\end{theorem}
\begin{proof}
We recall the definition from the proof of Theorem \ref{x:133} of the sets
$$S_j=\left\{ rm\: :\: bj\leq rm<b(j+1),\:m\in\mathbb{Z} \right\}$$
and the fact that
$$S_j=\left\{ bj+s_j,\: bj+s_j+r,\ldots,\:bj+s_j+d_jr \right\}.$$
Because of the assumptions:
$$\frac{1}{2}<A_0<A_1<1,\ \ A_0b\leq r\leq A_1b,$$
we have the following two cases.\\
\textit{Case 1:} $d_j=0$, $S_j=\{ s_j \}$.\\
\textit{Case 2:} $d_j=1$, $S_j=\{ s_j, s_j+r\}$.\\
From formulae (15) and (16), that is 
$$s_j\equiv -bj\:(\bmod\: r)\ \ \text{and}\ \ t_j\equiv b(j+1)\:(\bmod\: r),$$
we obtain
\[
\frac{s_j}{r}=\left\{ -\frac{jb}{r}\right\},\ \ \frac{t_j}{r}=\left\{ \frac{(j+1)b}{r}\right\},\tag{62}
\]
respectively. From (15), (16) and (62) it follows that $s_j$, as well as $t_j$ run through a complete residue system $(\bmod\: r)$, when $j$ runs from $1$ to $r$.
Denote by $$\mathcal{R}_1=\{1,2,\ldots, r-1\}.$$ For each $s\in\mathcal{R}_1$ there is thus a unique value of $j=j(s)\in\mathcal{R}_1$
with $s_j=s$. The value $j=j(s)$ by formula (15) is given by 
\[
\frac{j}{r}=\left\{ -\frac{sb^*}{r}\right\},\tag{63}
\]
where $b^*$ is determined by the relation
$$bb^*\equiv 1\:(\bmod\: r)\ \text{and by} \ 1\leq b^*\leq r-1  .$$
Formula (63) is seen by multiplying the equation 
$$\frac{jb}{r}=-\frac{s}{r}$$
by $b^*$. We obtain
$$\frac{jbb^*}{r}=-\frac{sb^*}{r},$$
from which (63) follows.\\
Similarly, for each $t\in\mathcal{R}_1$ there exists a unique value $h=h(t)\in\mathcal{R}_1$ with $t_h=t$. The value $h=h(t)$ by formula (15) is given by
\[
\frac{h}{r}=\left\{ \frac{(t-1)b^*}{r}\right\}.\tag{64}
\]
We now set
\[
\alpha=\alpha(r,b)=\frac{b^*}{r}.\tag{65}
\]
For a fixed value of $m_1\in\mathbb{N}$ we define 
\[
Q\left(r,b,m_1 \right)=\frac{br}{\pi}\sum_{s=1}^{2^{m_1}}\frac{1-2\{s\alpha\}}{s}=\frac{br}{\pi}f(\alpha;m_1).\tag{66}
\]
Let $m_0$ be the constant defined in Lemma \ref{x:35} and consider  
\[
m_1\geq m_0.\tag{67}
\]
We assume that $r$ does not belong to the exceptional set $\mathcal{E}(m_1)$ as specified in Lemma \ref{x:35} with
\[
|\mathcal{E}(m_1)|=O(\phi(b)2^{-m_1}).\tag{68}
\] 
We partition the sum $Q(r/b)$ into partial sums:
\[
Q\left(\frac{r}{b} \right)=Q_0\left(\frac{r}{b} \right)+Q_1\left(\frac{r}{b} \right),\tag{69}
\]
where
\[
Q_0\left(\frac{r}{b} \right)=\sum_{j=1}^{r-1}{}^{^*}j\sum_{l=0}^{d_j}\cot\left(\pi\frac{s_j+lr}{b} \right),\tag{70}
\]
where $\sum_{}^{^*}$ means that the sum is extended over all values of $j$ for which $$\{\theta jb/r\}\leq b^{-1}2^{m_1}\ \  \text{for either}\ \  \theta =1\ \  \text{or}\ \  \theta=-1$$
and 
\[
Q_1\left(\frac{r}{b} \right)=Q\left(\frac{r}{b} \right)-Q_0\left(\frac{r}{b} \right).\tag{71}
\]
\textit{We first deal with $Q_0(r/b)$.}\\
\textit{\textbf{Step 1.} We shall prove that}
$$Q_0\left(\frac{r}{b} \right)=Q\left(r,b,m_1 \right)+O\left( b2^{m_1}\right).$$
\textit{Proof of Step 1.}
The values of $s_j$, $t_j$ satisfy the inequalities
$$1\leq s_j , t_j\leq 2^{m_1},$$
because of formulae (62) and (70).\\
The values of $j$ corresponding to $s$ by the formula $j=j(s)$ are given by formula (63), whereas the values of $h$ corresponding to $t$ by the formula
$h=h(t)$ are given by the formula (64).\\
In formula (70), we have if $l\neq 0$ for $\theta =-1$,
$$\cot\left(\pi\frac{s_j+lr}{b} \right)=O(1).$$ 
We recall the formula $s_j+d_jr+t_j=b$ and we obtain
if $l=0$ for $\theta=1,$ $$\cot\left(\pi\frac{s_j+lr}{b} \right)=O(1).$$ 
(The variable $\theta$ is implied in formula (70) in the definition of $\sum_{}^{^*}$).\\
We now rewrite the sum $Q_0(r/b)$. \\
In formula (70), we retain from the inner sums 
$$\sum_{l=0}^{d_j}\cot\left(\pi\frac{s_j+lr}{b}\right)$$
only the following terms:\\
If $\left\{jb/r \right\}\leq b^{-1}2^{m_1}$ (the case $\theta=1$), we retain the term for $l=0$ and write $s_j=s;$ if $\{-jb/r\}\leq b^{-1}2^{m_1}$ (the case 
$\theta=-1$), we retain the term for $l=d_j$ and write $t_j=s$ (that is $s_j+d_jr=b-t_j=b-s$). For all other terms in (70) we use the estimate
$$\cot\left(\pi\frac{s_j+lr}{b}\right)=O(1).$$
By recalling (63) and (65), we obtain
$$Q_0\left(\frac{r}{b} \right)=r\left( \sum_{s=1}^{2^{m_1}}\left(\left\{-s\alpha \right\} \cot\left(\pi\frac{s}{b} \right)+\left\{s\alpha \right\}\cot\left(\pi\frac{b-s}{b} \right)\right)\right)+O\left( b2^{m_1}\right),$$
where the error term $O\left( b2^{m_1}\right)$ comes from the estimates 
$$\cot\left(\pi\frac{s_j+lr}{b}\right)=O(1).$$
\textit{Note.} The restriction of $\sum_{}^{^*}$ in (70) is contained in the above formula for $Q_0(r/b)$ in the restriction in the range of summation $s\in\{1,2,\ldots, 2^{m_1}  \}$,
since the possible values of $s$ are $1,2,\ldots,b-r$.

Since $\cot x$ has a pole at $x=0$ we have
$$\cot (\pi x)=\frac{1}{\pi x}+O(1),\ \ \cot(\pi(1-x))=-\frac{1}{\pi x}+O(1),\  (x\rightarrow 0) ,$$
and therefore 
\begin{align*}
Q_0\left(\frac{r}{b} \right)&=r \sum_{s=1}^{2^{m_1}}\left(\left( \frac{b}{\pi s}+O(1)\right)\left\{-s\alpha \right\}+\left(-\frac{b}{\pi s}+O(1) \right)\left\{s\alpha \right\}\right)+O\left( b2^{m_1}\right) \\
&=r \sum_{s=1}^{2^{m_1}}\left(\left( \frac{b}{\pi s}+O(1)\right)(1-\left\{s\alpha \right\})+\left(-\frac{b}{\pi s}+O(1) \right)\left\{s\alpha \right\}\right)+O\left( b2^{m_1}\right)\\
&=r \sum_{s=1}^{2^{m_1}}\left(\frac{b}{\pi s}-\frac{b}{\pi s}\left\{s\alpha \right\}+O(1)-\left\{s\alpha \right\}O(1)-\frac{b}{\pi s}\left\{s\alpha \right\}+\left\{s\alpha \right\}O(1)\right)+O\left( b2^{m_1}\right)\\
&=\frac{br}{\pi}\left(\sum_{s=1}^{2^{m_1}} \frac{1}{s}-\sum_{s=1}^{2^{m_1}} \frac{2}{s}\left\{s\alpha \right\}\right)+r\sum_{s=1}^{2^{m_1}}O(1)-r\sum_{s=1}^{2^{m_1}}\left\{s\alpha \right\}O(1)+r\sum_{s=1}^{2^{m_1}}\left\{s\alpha \right\}O(1)+O\left( b2^{m_1}\right)\\
&=\frac{br}{\pi}\sum_{s=1}^{2^{m_1}}  \frac{1-2\left\{s\alpha \right\}}{s}+O\left( b2^{m_1}\right).
\end{align*}
Thus, we have
\[
Q_0\left(\frac{r}{b} \right)=Q\left(r,b,m_1 \right)+O\left( b2^{m_1}\right)\tag{72}.
\]
This completes the proof of Step 1.
\\ \\
\textit{In the following we deal with $Q_1(r/b)$ as defined in $(71)$. }\\

\textit{\textbf{Step 2.} We shall obtain a decomposition for $Q_1(r/b).$}\\
For $A_0b\leq r\leq A_1b$, let $g_0(r)$ be defined by 
$$g_0(r)=\min\left\{g_1\::\: g_1\in\mathbb{N}\cup\{ 0\} , r\not\in\mathcal{E}((g+1)m_1),\ \text{for all}\ g\geq g_1\right\},$$
where the set $\mathcal{E}(m)$ has been defined in Lemma \ref{x:35}.\\
Fix $r$. Let $g_0=g_0(r)\in\mathbb{N}\cup\{0\}$, and choose sequences $(j_k)$ of real numbers, where $$j_k=j_k(g_0,m_1),\ j_0<j_1<\ldots<j_{l+1}=r$$ and real numbers $\zeta=\zeta(g_0,m_1)$ with $\zeta\in [0,1]$, as follows
$$j_0=0,\ j_1=b2^{-(g_0+1)m_1}+\zeta,\ j_{k+1}=j_k(1+\lambda_k)\ \ \text{for}\ \ 1\leq k\leq l,$$
where
\[
2^{-(g_0+1)m_1}\leq \lambda_k\leq 2^{-(g_0+1)m_1+1}\ \ \text{and}\ \ j_k\not\in\mathbb{Z}\ \ \text{for}\ \ 1\leq k\leq l.\tag{73}
\]
We define 
\[
J_k=\left[j_k,j_{k+1}\right]\ \text{for}\ 0\leq k\leq l. \tag{74}
\]
Thus
\[
[0,r]=\bigcup_{k=0}^l J_k.\tag{75}
\]
Since $J_k$ intersect at most at their endpoints $j_k$, $1\leq k\leq l$, and since $j_k\not\in\mathbb{Z}$ by (73), it follows that each $j\in\mathbb{N}$, with $0\leq j\leq r$
belongs to exactly one interval $J_k$.\\
We choose numbers $u_0$, $u_1,\ldots,$ $u_n$ such that
$$\frac{2^{m_1}}{b}=u_0<u_1<\cdots<u_{n}=\frac{b}{2r},$$
where
$$u_0=b^{-1}2^{m_1}, u_1=b^{-1}2^{5(g_0+1)m_1}, u_{n}=\frac{b}{2r}, u_{h+1}=u_h(1+\xi_h),$$
where 
$$2^{-(g_0+1)m_1}\leq \xi_h\leq 2^{-(g_0+1)m_1+1}$$
and
\[
u_h\not\in\left\{\frac{c}{2r}\::\: c\in\mathbb{Z} \right\},
\]
for $1\leq h\leq n-1.$\\
We then set 
\begin{equation}
\mathcal{H}_h=\left\{
\begin{array}{l l}
    [u_h,u_{h+1}]\:, & \quad \text{for $0\leq h\leq n-1$}\vspace{2mm}\\ 
    b/r-\mathcal{H}_{h-n}\:, & \quad \text{for $n\leq h\leq 2n-1$}\:.\\
  \end{array} \right.\tag{76}
  \nonumber
\end{equation}
Therefore, we obtain
\[
\left[b^{-1}2^{m_1},\frac{b}{r}- b^{-1}2^{m_1}\right]=\bigcup_{h=0}^{2n-1}\mathcal{H}_h,\tag{77}
\]
where the intervals $\mathcal{H}_h$ intersect at most at their endpoints.\\
For $0\leq h\leq 2n-1$, $0\leq k\leq l$, 
we set
\[
\mathcal{F}(h)=\left\{j\::\:1\leq j\leq r-1,\ \left\{-\frac{jb}{r} \right\}\in\mathcal{H}_h  \right\}\tag{78}
\]
and
\[
\mathcal{G}(h,k)=\left\{j\::\: j\in J_k,\ \left\{-\frac{jb}{r} \right\}\in\mathcal{H}_h \right\}\tag{79}
\]
and by (71), (78) and (79), we obtain
\[
Q_1\left(\frac{r}{b} \right)=\sum_{h=0}^{n-1}\sum_{j\in\mathcal{F}(h)}j\sum_{l=0}^{d_j}\cot\left(\pi\:\frac{s_j+lr}{b} \right).\tag{80}
\]
Equivalently 
\[
Q_1\left(\frac{r}{b} \right)=\sum_{h=0}^{n-1}\sum_{k=0}^{l}\sum_{j\in\mathcal{G}(h,k)}j\sum_{l=0}^{d_j}\cot\left(\pi\:\frac{s_j+lr}{b} \right).\tag{81}
\]
We define the sets $V_1$, $V_2$ as follows,
\[
V_1=\left\{j\::\: 1\leq j\leq r,\ \left\{-\frac{jb}{r} \right\}>\frac{b}{r}-1 \right\}\tag{82}
\]
and
\[
V_2=\left\{j\::\: 1\leq j\leq r,\ \left\{-\frac{jb}{r} \right\}<\frac{b}{r}-1 \right\}.\tag{83}
\]
Because of the fact that $j_k\not\in\mathbb{Z}$ for $1\leq k\leq l$, it follows that:\\
 \[\text{Each}\ j,\ 1\leq j\leq r-1,\ \text{belongs to exactly one of the sets}\ V_i,\ i\in\{1,2\}.\tag{84}\]
 Moreover, there is at most one value of $h$, say,
 \[
 h=h_0\tag{85}
 \]
 such that $\mathcal{F}(h)$ has a non-empty intersection with both of the sets $V_1$ and $V_2$. From (80), (81) and (85), we obtain
 \[
 Q_1\left(\frac{r}{b} \right)=Q^{(1)}\left(\frac{r}{b} \right)+Q^{(2)}\left(\frac{r}{b} \right)+Q^{(3)}\left(\frac{r}{b} \right),\tag{86}
 \]
 where we define
 \[
 Q^{(i)}\left(\frac{r}{b} \right)=\sum_{\substack{ h=0 \\ \mathcal{F}(h)\subset V_i}}^{2n-1}\sum_{\substack{j\in\mathcal{F}(h)}}j\sum_{l=0}^{d_j}\cot\left(\pi\:\frac{s_j+lr}{b} \right)\tag{87}
 \]
 or equivalently 
  \[
 Q^{(i)}\left(\frac{r}{b} \right)=\sum_{\substack{ h=0 \\ \mathcal{F}(h)\subset V_i}}^{2n-1}\sum_{k=0}^l\sum_{j\in\mathcal{G}(h,k)}j\sum_{l=0}^{d_j}\cot\left(\pi\:\frac{s_j+lr}{b} \right),
 \]
 for $i=1,2$. We have
 \[
 Q^{(3)}\left(\frac{r}{b} \right)=\sum_{j\in\mathcal{F}(h_0)}j\sum_{l=0}^{d_j}\cot\left(\pi\:\frac{s_j+lr}{b} \right)\tag{88},
 \]
 or equivalently
   \[
 Q^{(3)}\left(\frac{r}{b} \right)=\sum_{k=0}^l\sum_{j\in\mathcal{G}(h_0,k)}j\sum_{l=0}^{d_j}\cot\left(\pi\:\frac{s_j+lr}{b} \right)\tag{89}.
 \]
 If $h_0$ in (85) does not exist, the term $Q^{(3)}(r/b)$ in (86) is missing.\\
This completes Step 2.\\ \\
We now deal with the cases $i=1,2,3$ in (87) and (88), separately. \\

  \textit{\textbf{Step 3.} We shall now prove that}
 $$ Q^{(1)}\left(\frac{r}{b} \right)=O(b^2\:2^{-m_1}).$$
 \textit{Proof of Step 3.} Let $j\in V_1$. From (15), that is $$s_j\equiv -bj\:(\bmod\: r),$$ it follows that
 \[
 \frac{s_j}{r}=\left\{-\frac{jb}{r} \right\}.\tag{90}
 \]
 Thus, for $j\in V_1$ we have $s_j/r>b/r-1$ and hence $s_j+r>b$. Therefore, for $j\in V_1$, we obtain
 $$d_j=0,\ S_j=\{s_j \}.$$
 Hence
  \[
 Q^{(1)}\left(\frac{r}{b} \right)=\sum_{\substack{ h=0 \\ \mathcal{F}(h)\subset V_1}}^{2n-1}\sum_{\substack{j\in\mathcal{F}(h)}}j\cot\left(\frac{\pi s_j}{b} \right)\tag{91}
 \]
 or equivalently
  \[
 Q^{(1)}\left(\frac{r}{b} \right)=\sum_{\substack{ h=0 \\ \mathcal{F}(h)\subset V_1}}^{2n-1}\sum_{k=0}^l\sum_{j\in\mathcal{G}(h,k)}j\cot\left(\frac{\pi s_j}{b} \right)\tag{92}.
 \]
 We have
 \begin{align*}
 \sum_{\substack{ h=0 \\ \mathcal{F}(h)\subset V_1}}^{2n-1}&\sum_{j\in\mathcal{F}(h)}j\cot\left(\frac{\pi s_j}{b} \right)\tag{93}\\
 &=\sum_{\substack{ h=0 \\ \mathcal{F}(h)\subset V_1}}^{n-1}\left(\sum_{j\in\mathcal{G}(h,k)}j\cot\left(\frac{\pi s_j}{b} \right)+\sum_{j\in\mathcal{G}(h+n,k)}j\cot\left(\frac{\pi s_j}{b} \right) \right).
 \end{align*}
Suppose $j\in\mathcal{F}(h)$, i.e.
$$\left\{-\frac{jb}{r} \right\}\in\mathcal{H}_h=[u_h,u_{h+1}],\ 0\leq h\leq 2n-1. $$
Set 
$$\min=\text{minimum}\left(u_h,\frac{b}{r}-u_h \right).$$
By the definition of the intervals $\mathcal{H}_h$, as given by (76) we have:
$$\left|\left\{ -\frac{jb}{r}\right\} -u_h \right|\leq \text{min}\cdot \xi_h\leq \text{min}\cdot 2^{-(g_0+1)m_1}.$$
For $\omega\in [u_h,u_{h+1}]$ we have:
$$\text{cosec}^2\left(\frac{\pi r\omega}{b}\right)=O\left(\text{min}^{-2}\right).$$
By the mean-value theorem, there exists $\omega^{*}$ between $u_h$ and $\left\{-jb/r \right\}$, such that by (90) we get
\begin{align*}
\vline\:\cot\left(\frac{\pi s_j}{b} \right)-\cot\left(\frac{\pi ru_h}{b} \right)  \:\vline&=\:\vline\:\cot\left(\frac{\pi r}{b}\left\{-\frac{jb}{r} \right\} \right)-\cot\left(\frac{\pi ru_h}{b} \right)  \:\vline\tag{94}\\
&=\frac{\pi r}{b} \:\vline\:\left\{-\frac{jb}{r} \right\}-u_h \:\vline\: \text{cosec}^2\left( \frac{\pi r \omega^*}{b}\right)\\
&=O(\text{min}^{-1}\cdot2^{-(g_0+1)m_1}).
\end{align*}
By the formulae 
$$\cot (\pi x)=\frac{1}{\pi x}+O(1),\ \ \cot(\pi(1-x))=-\frac{1}{\pi x}+O(1),\  (x\rightarrow 0) ,$$
and by the definition of $\text{min}$, there are absolute constants $\kappa_1,\kappa_2>0$, such that
\[
\kappa_1\: \text{min}^{-1}\leq\: \vline\:\cot\left(\frac{\pi ru_h}{b} \right) \:\vline\:\leq \kappa_2 \:\text{min}^{-1}.\tag{95}
\]
Putting together the formulae (78), (90), (94) and (95), it follows 
\[
\cot\left(\frac{\pi s_j}{b} \right)=\cot\left(\frac{\pi ru_h}{b} \right)\left(1+O(2^{-(g_0+1)m_1})  \right).\tag{96}
\]
Since $r\not\in\mathcal{E}((g_0+1)m_1)$ the conclusion of Lemma \ref{x:35} holds with $m$ replaced by $(g_0+1)m_1$:\\
Let $U_1,U_2,j_1,j_2$ be real numbers such that 
$$U_1\geq b^{-1}2^{5(g_0+1)m_1},\ U_2=U_1(1+\delta_1),\ U_2\leq 1,\ j_2-j_1\geq b2^{-(2(g_0+1)m_1+1)},$$ where 
$$2^{-(g_0+1)m_1}\leq \delta_1\leq 2^{-(g_0+1)m_1+1}.$$
Then, we have
\[
\:\vline\: \left\{j\::\: j_1\leq j\leq j_2,\ \left\{\frac{\theta jb}{r} \right\}\in[U_1,U_2] \right\}\:\vline=(j_2-j_1)\delta_1U_1\left(1+O(2^{-(g_0+1)m_1})\right)\tag{97}
\]
For $0\leq h\leq n-1$ and $0\leq k\leq l$ we have by (97) for the cardinalities of the sets $\mathcal{G}(h,k)$ and $\mathcal{G}(n+h,k)$:
\begin{align*}
|\mathcal{G}(h,k)|&=\:\vline\: \left\{j\::\: j_k\leq j\leq j_{k+1},\ \left\{-\frac{jb}{r} \right\}\in[u_h,u_{h+1}] \right\}\:\vline\tag{98}\\
&=(u_{h+1}-u_h)(j_{k+1}-j_k)\left(1+ O(2^{-(g_0+1)m_1})\right)
\end{align*}
and
\begin{align*}
|\mathcal{G}(n+h,k)|&=\:\vline\: \left\{j\::\: j_k\leq j\leq j_{k+1},\ \left\{-\frac{jb}{r} \right\}\in[u_{n+h+1},u_{n+h}] \right\}\:\vline\tag{99}\\
&=(u_{n+h}-u_{n+h+1})(j_{k+1}-j_k)\left(1+ O(2^{-(g_0+1)m_1})\right)\\
&=|\mathcal{G}(h,k)|\left(1+ O(2^{-(g_0+1)m_1})\right),
\end{align*}
since by (82), (83) it holds
$$u_{n+h}-u_{n+h+1}=u_{h+1}-u_h.$$
By (96), (98) and (99), for $0\leq h\leq n-1$ and $0\leq k\leq l$, we now obtain
\[
\sum_{j\in\mathcal{G}(h,k)}j\cot\left(\frac{\pi s_j}{b} \right)=(u_{h+1}-u_h)(j_{k+1}-j_k)\cot\left(\frac{\pi r u_h}{b} \right)(1+O(2^{-(g_0+1)m_1}))\tag{100}
\]
and also
\begin{align*}
\sum_{j\in\mathcal{G}(h+n,k)}j\cot\left(\frac{\pi s_j}{b} \right)&=(u_{h+1}-u_h)(j_{k+1}-j_k)\\
&\cdot\cot\left(\frac{\pi r}{b}\left(\frac{b}{r}-u_h \right) \right)(1+O(2^{-(g_0+1)m_1}))\tag{101}
\end{align*}
Since $$\cot(\pi x)=-\cot(\pi(1-x))$$ for all real values of $x$, we get
\[
\cot\left(\frac{\pi r u_h}{b} \right)=-\cot\left(\frac{\pi r}{b}\left(\frac{b}{r}-u_h \right) \right).\tag{102}
\]
Since 
$$\cot (\pi x)=\frac{1}{\pi x}+O(1),\ (x\rightarrow 0) ,$$
we have 
\[\cot\left(\frac{\pi s_j}{b} \right)=O\left(bs_j^{-1}\right).\]
Additionally, it also holds 
$$u_0=b^{-1}2^{m_1},$$
which is the first term of the finite sequence $u_0<u_1<\cdots<u_n$, which was defined earlier.\\
Therefore, from (100), (101) and (102), we get
\[
\sum_{\substack{h=0\\\mathcal{F}(h)\subset V_1}}^{2n-1}\sum_{j\in\mathcal{G}(h,k)}j\cot\left(\frac{\pi s_j}{b} \right)=O\left(b(j_{k+1}-j_k)\max_{0\leq h\leq 2n-1}(u_{h+1}-u_h)2^{-m_1}\right).
\]
Since by formula (75), we have
$$[0,r]=\bigcup_{k=0}^l J_k=\bigcup_{k=0}^l[j_k,j_{k+1}]$$
and by formula (77) we have
$$\left[b^{-1}2^{m_1},\frac{b}{r}-b^{-1}2^{m_1}\right]=\bigcup_{h=0}^{2n-1}\mathcal{H}_h,$$
combining (92) and the above formula, it follows that
  \[
 Q^{(1)}\left(\frac{r}{b} \right)=O(b^2\:2^{-m_1}).
 \]
 This completes the proof of Step 3.\\
 
 \textit{\textbf{Step 4.} We will show that }
$$ Q^{(2)}\left(\frac{r}{b} \right)=O\left(b^2\: 2^{-m_1}\right).$$
\textit{Proof of Step 4.} Let $j\in V_2$. From (15), that is $$s_j\equiv -bj\:(\bmod\: r),$$ it follows that
 \[\frac{s_j}{r}=\left\{-\frac{jb}{r} \right\}. \tag{103}\]
 We define the interval
 \[
 \mathcal{K}_0=\left[ b^{-1}\:2^{5(g_0+1)m_1},\:\frac{b}{r}-1-b^{-1}\:2^{5(g_0+1)m_1}\right).\tag{104}
 \]
Therefore, for $j\in V_2$, we have
\[
\frac{s_j}{r}\in \mathcal{K}_0\tag{105}
\]
and thus
\[
s_j+r<b.\tag{106}
\]
Hence, for $j\in V_2$ we obtain
$$d_j=1,\ S_j=\{s_j, s_j+r \}$$
and therefore we get
\[
 Q^{(2)}\left(\frac{r}{b} \right)=\sum_{\substack{0\leq j\leq r \\ \left\{-jb/r \right\}\in\mathcal{K}_0}}
j\left(\cot\left(\pi\:\frac{ s_j}{b} \right)+\cot\left(\pi\:\frac{ s_j+r}{b}\right) \right)\tag{107}\]
Moreover, we define the intervals
\[
\mathcal{K}_1=\left[ b^{-1}\:2^{5(g_0+1)m_1},\:\frac{b}{2r}-\frac{1}{2} \right)\ \ \text{and}\ \ \mathcal{K}_2=\left[ \frac{b}{2r}-\frac{1}{2},\:\:\frac{b}{r}-1-b^{-1}\:2^{5(g_0+1)m_1}\right).\tag{108}
\]
We set
\[
Q^{(2,i)}\left(\frac{r}{b} \right)=\sum_{\substack{0\leq j\leq r \\ \left\{-jb/r \right\}\in\mathcal{K}_i}}
j\left(\cot\left(\pi\:\frac{ s_j}{b} \right)+\cot\left(\pi\:\frac{ s_j+r}{b}\right) \right),
\tag{109}
\]
for $i=1,2$.\\
We have
\[
Q^{(2)}\left(\frac{r}{b} \right)=Q^{(2,1)}\left(\frac{r}{b} \right)+Q^{(2,2)}\left(\frac{r}{b} \right)\tag{110}.
\]
We now derive a different representation for $Q^{(2,2)}(r/b)$.\\ 
We recall from Section 3 the following facts and notations:
\[
t_j\equiv\: b(j+1)\:(\bmod\:r)\ \ \text{and}\ \ b=s_j+d_jr+t_j.\tag{111}
\]
Thus, we can write
\[
\frac{t_j}{r}=\left\{\frac{(j+1)b}{r} \right\}.\tag{112}
\]
From (111) and the fact that $d_j=1$, it follows that
\[
b=s_j+r+t_j\tag{113}.
\]
From (108), we have
\begin{align*}
\frac{t_j}{r}\in\mathcal{K}_1&\Leftrightarrow \frac{b}{r}-b^{-1}2^{5(g_0+1)m_1}\geq\frac{b-t_j}{r}>\frac{b}{r}-\left( \frac{b}{2r}-\frac{1}{2}\right)\tag{114}\\
&\Leftrightarrow \frac{s_j}{r}=\frac{b-t_j-r}{r}=\left\{-\frac{jb}{r} \right\}\in\mathcal{K}_2.
\end{align*}
We thus obtain
\[
Q^{(2,2)}\left(\frac{r}{b} \right)=\sum_{\substack{0\leq j\leq r \\ \left\{(j+1)b/r \right\}\in\mathcal{K}_1}}
j\left(\cot\left(\pi\:\left(1-\frac{ t_j}{b} \right)\right)+\cot\left(\pi\:\left(1-\frac{ t_j+r}{b}\right)\right) \right).\tag{115}
\]
There is at most one value of $h$, say
\[
h=h_1,\tag{116}
\]
such that $\mathcal{H}_{h_1}$ has a non-empty intersection with both of the intervals $\mathcal{K}_1$ and $\mathcal{K}_2$. In the case that $h_1$ exists, we define
\[
Q^{(2,3)}\left(\frac{r}{b} \right)=\sum_{\substack{0\leq j\leq r \\ \left\{-jb/r \right\}\in\mathcal{H}_{h_1}}}
j\left(\:\vline\:\cot\left(\pi\:\frac{ s_j}{b} \right)\:\vline\:+\:\vline\:\cot\left(\pi\:\frac{ s_j+r}{b}\right) \:\vline\:\right).\tag{117}
\]
Hence, by (110), (115), (116) and (117), we obtain
\[
\:\vline\:Q^{(2)}\left(\frac{r}{b} \right)\:\vline\:\leq \:\vline\:Q^{(2,0)}\left(\frac{r}{b} \right)\:\vline\:+Q^{(2,3)}\left(\frac{r}{b} \right),\tag{118}
\]
where
\begin{align*}
Q^{(2,0)}\left(\frac{r}{b} \right)&=\sum_{\substack{0\leq h\leq n-1\\ \mathcal{H}_h\subset \mathcal{K}_1}}\sum_{\substack{0\leq j\leq r \\ \left\{-jb/r \right\}\in\mathcal{H}_{h}}}j\left(\cot\left(\pi\:\frac{ s_j}{b} \right)+\cot\left(\pi\:\frac{ s_j+r}{b}\right) \right)\tag{119}\\
&+\sum_{\substack{0\leq h\leq n-1\\ \mathcal{H}_h\subset \mathcal{K}_1}}\sum_{\substack{0\leq j\leq r \\ \left\{(j+1)b/r \right\}\in\mathcal{H}_h}}
j\left(\cot\left(\pi\:\left(1-\frac{ t_j}{b} \right)\right)+\cot\left(\pi\:\left(1-\frac{ t_j+r}{b}\right)\right) \right) .
\end{align*}
If $h_1$ does not exist, the term $ Q^{(2,3)}(r/b)$ is missing from formula (118).\\
For $0\leq h\leq n-1$, $0\leq k\leq l$, we define
\begin{align*}
Q^{(2)}_{h,k}\left(\frac{r}{b} \right)&=\sum_{\substack{j\in J_k\\ \left\{-jb/r \right\}\in\mathcal{H}_h}}j\left(\cot\left(\pi\:\frac{ s_j}{b} \right)+\cot\left(\pi\:\frac{ s_j+r}{b}\right) \right)\tag{120}\\
&+\sum_{\substack{j\in J_k\\ \left\{(j+1)b/r \right\}\in\mathcal{H}_h}}j\left(\cot\left(\pi\:\left(1-\frac{ t_j}{b} \right)\right)+\cot\left(\pi\:\left(1-\frac{ t_j+r}{b}\right)\right) \right) .
\end{align*}
From (119) and (120), it follows
\[
Q^{(2,0)}\left(\frac{r}{b} \right)=\sum_{\substack{0\leq h\leq n-1 \\ \mathcal{H}_h\subset\mathcal{K}_1}}\sum_{k=0}^lQ^{(2)}_{h,k}\left(\frac{r}{b} \right).\tag{121}
\]
Due to the fact that 
$$\cot(\pi x)=-\cot(\pi(1-x))\ \text{for all}\ x\in\mathbb{R},$$ 
we obtain
\begin{align*}
Q^{(2)}_{h,k}\left(\frac{r}{b} \right)&=\sum_{\substack{j\in J_k\\ \left\{-jb/r \right\}\in\mathcal{H}_h}}j\left(\cot\left(\pi\:\frac{ s_j}{b} \right)+\cot\left(\pi\:\frac{ s_j+r}{b}\right) \right) \tag{122}\\
&-\sum_{\substack{j\in J_k\\ \left\{(j+1)b/r \right\}\in\mathcal{H}_h}}j\left(\cot\left(\pi\:\frac{ t_j}{b} \right)+\cot\left(\pi\:\frac{ t_j+r}{b}\right) \right).
\end{align*}
By the same reasoning as in (96) we have for $\{-jb/r \}\in\mathcal{H}_h$:
\[
\cot\left(\pi\:\frac{ s_j}{b} \right)=\cot\left(\pi\:\frac{ r}{b}u_h \right)\left(1+O\left(2^{-(g_0+1)m_1}\right)\right)\tag{123}
\]
and
\[
\cot\left(\pi\:\frac{ s_j+r}{b}\right)=\cot\left(\pi\left(\frac{r}{b}u_h+1 \right) \right)\left(1+O\left(2^{-(g_0+1)m_1}\right)\right).\tag{124}
\]
For $\{(j+1)b/r \}\in\mathcal{H}_h$, it follows that
\[
\cot\left(\pi\:\frac{ t_j}{b} \right)=\cot\left(\pi\:\frac{ r}{b}u_h \right)\left(1+O\left(2^{-(g_0+1)m_1}\right)\right)\tag{125}
\]
and
\[
\cot\left(\pi\:\frac{ t_j+r}{b}\right)=\cot\left(\pi\left(\frac{r}{b}u_h+1 \right) \right)\left(1+O\left(2^{-(g_0+1)m_1}\right)\right).\tag{126}
\]
For $j\in J_k=[j_k,j_{k+1}]$, we obtain by (73) that
\[
j=j_k\left(1+O\left(2^{-(g_0+1)m_1}\right)\right)\tag{127}.
\]
By (98) and (99), we have
\begin{align*}
&\vline\:\left\{ j\::\: j\in J_k=[j_k,j_{k+1}],\left\{-\frac{jb}{r} \right\}\in\mathcal{H}_h\right\}\:\vline \tag{128}\\
&=(j_{k+1}-j_k)(u_{h+1}-u_h)\left(1+O\left(2^{-(g_0+1)m_1}\right)\right)
\end{align*}
and also
\begin{align*}
&\vline\:\left\{ j\::\: j\in J_k=[j_k,j_{k+1}],\left\{\frac{(j+1)b}{r} \right\}\in\mathcal{H}_h\right\}\:\vline \tag{129}\\
&=(j_{k+1}-j_k)(u_{h+1}-u_h)\left(1+O\left(2^{-(g_0+1)m_1}\right)\right).
\end{align*}
By (123), (127) and (128), we obtain
\begin{align*}
\sum_{\substack{j\in J_k\\ \left\{-jb/r \right\}\in\mathcal{H}_h}}&j\cot\left(\pi\:\frac{ s_j}{b} \right)\tag{130}\\
&=(j_{k+1}-j_k)(u_{h+1}-u_h)\cot\left(\pi\:\frac{r}{b}u_h \right)\left(1+O\left(2^{-(g_0+1)m_1}\right)\right).
\end{align*}
Additionally, by (124), (127) and (128), we have
\begin{align*}
&\sum_{\substack{j\in J_k\\ \left\{-jb/r \right\}\in\mathcal{H}_h}}j\cot\left(\pi\:\frac{ s_j+r}{b} \right)\tag{131}\\
&=(j_{k+1}-j_k)(u_{h+1}-u_h)\cot\left(\pi\:\left(\frac{r}{b}u_h+1\right) \right)\left(1+O\left(2^{-(g_0+1)m_1}\right)\right)
\end{align*}
By (125), (127) and (129) we get
\begin{align*}
\sum_{\substack{j\in J_k\\ \left\{(j+1)b/r \right\}\in\mathcal{H}_h}}&j\cot\left(\pi\:\frac{ t_j}{b} \right)\tag{132}\\
&=(j_{k+1}-j_k)(u_{h+1}-u_h)\cot\left(\pi\:\frac{r}{b}u_h \right)\left(1+O\left(2^{-(g_0+1)m_1}\right)\right).
\end{align*}
From (126), (127) and (129), it follows that
\begin{align*}
&\sum_{\substack{j\in J_k\\ \left\{(j+1)b/r \right\}\in\mathcal{H}_h}}j\cot\left(\pi\:\frac{ t_j+r}{b} \right)\tag{133}\\
&=(j_{k+1}-j_k)(u_{h+1}-u_h)\cot\left(\pi\:\left(\frac{r}{b}u_h+1\right) \right)\left(1+O\left(2^{-(g_0+1)m_1}\right)\right).
\end{align*}
By (122), (129), (131), (132) and (133), we obtain
\begin{align*}
Q^{(2)}_{h,k}\left(\frac{r}{b} \right)&=O\left( (j_{k+1}-j_k)(u_{h+1}-u_h)u_h^{-1}b\right)\tag{134}\\
&=O\left(b(j_{k+1}-j_k)(u_{h+1}-u_h)2^{-m_1} \right).
\end{align*}
Similarly as for $Q^{(1)}(r/b)$, combining (119) and the above formula we obtain
\[
Q^{(2,0)}\left(\frac{r}{b} \right)=O\left(b^2\: 2^{-m_1}\right).\tag{135}
\]
The same reasoning that leads to the estimate (133) implies
\[
Q^{(2,3)}\left(\frac{r}{b} \right)=O\left(b^2\: 2^{-m_1}\right).\tag{136}
\]
By (118), (135) and (136), it follows that
\[
Q^{(2)}\left(\frac{r}{b} \right)=O\left(b^2\: 2^{-m_1}\right).\tag{137}
\]
This completes the proof of Step 4.\\

\textit{\textbf{Step 5.} We will prove that }
$$Q^{(3)}\left(\frac{r}{b}\right)=O\left(b^2\: 2^{-m_1}\right).$$
\textit{Proof of Step 5.} The same reasoning that leads to the estimate (130) yields 
$$\sum_{\substack{j\in J_k\\ \left\{-jb/r \right\}\in\mathcal{H}_{h_0}}}j\cot\left(\pi\:\frac{ s_j}{b} \right)=O\left( (j_{k+1}-j_k)(u_{h_0+1}-u_{h_0})\cot\left(\pi\:\frac{ r}{b} u_{h_0}\right)\right).$$
Combining this with (89), we obtain
\[
Q^{(3)}\left(\frac{r}{b}\right)=O\left(b^2\: 2^{-m_1}\right).
\]
This completes the proof of Step 5.\\ \\
From the above estimates of $Q^{(1)}(r/b)$, $Q^{(2)}(r/b)$, $Q^{(3)}(r/b)$, we obtain
 \begin{align*}
 Q_1\left(\frac{r}{b} \right)=Q^{(1)}\left(\frac{r}{b} \right)+Q^{(2)}\left(\frac{r}{b} \right)+Q^{(3)}\left(\frac{r}{b} \right)=O\left(b^2\: 2^{-m_1}\right).\tag{138}
 \end{align*}

 \textit{\textbf{Step 6.} We shall prove that}
 $$\sum_{\substack{r:(r,b)=1\\A_0b\leq r\leq A_1b}}Q_1\left(\frac{r}{b} \right)^L=O(b^{2L}\phi(b)2^{-Lm_1}).$$
 \textit{Proof of Step 6.} We now partition the set
 $$ \mathcal{R}_2=\left\{  r\::\: r\in\mathbb{N},\: (r,b)=1,\: A_0b\leq r\leq A_1b\right\},$$
 as follows:\\ 
 Let $g_0\in\mathbb{N}\cup\{0\}$ and define
 \[
 \mathcal{R}_2(g_0)=\{r\::\: r\in\mathcal{R}_2,\: g_0(r)=g_0 \}.\tag{139}
 \]
 We have
 \[
 \mathcal{R}_2=\bigcup_{g_0\in\mathbb{N}\cup\{0\}} \mathcal{R}_2(g_0).\tag{140}
 \]
For a fixed positive integer $L$, we define the sum
 $$\sum{}_{{}_1}=\sum_{\substack{r:(r,b)=1\\A_0b\leq r\leq A_1b}}Q_1\left(\frac{r}{b} \right)^L.$$
 By making use of the partition (140) we obtain
 \[
\sum{}_{{}_1}=\sum_{0\leq g_0\leq  (L+1)\log\log b }\sum{}^{{}^{(g_0)}} +\sum{}_{{}_2}\tag{141}
 \]
 where the partial sums $$\sum{}^{{}^{(g_0)}},\ \ \sum{}_{{}_2}$$ are defined by 
 \[
\sum{}^{{}^{(g_0)}}=\sum_{r\in\mathcal{R}_2(g_0)}Q_1\left(\frac{r}{b} \right) ^L \tag{142}
 \]
 and
 \[
 \sum{}_{{}_2}=\sum_{\substack{r\in\mathcal{R}_2\\g_0(r)>  (L+1)\log\log b} }Q_1\left(\frac{r}{b} \right) ^L.\tag{143}
 \]
 Let $$0\leq g_0\leq  (L+1)\log\log b. $$ By the definition of $g_0(r)$ that is 
 $$g_0(r)=\min\left\{g_1\::\: g_1\in\mathbb{N}\cup\{ 0\} , r\not\in\mathcal{E}((g+1)m_1),\ \text{for}\ g\geq g_1\right\},$$
it follows:
$$g_0(r)=g_0\Rightarrow r\in\mathcal{E}(g_0m_1).$$
By Lemma \ref{x:35} we have
\[
|\mathcal{E}(g_0m_1)|=O(\phi(b)2^{-g_0m_1}).\tag{144}
\]
From (138), (142) and (144), we obtain
\begin{align*}
\sum{}^{{}^{(g_0)}}=\sum_{r\in\mathcal{R}_2(g_0)}|O(b^2\:2^{-m_1})|^L&=O(b^{2L}\phi(b)2^{-Lm_1-g_0m_1})\\
&=O(b^{2L}\phi(b)2^{-(g_0+L)m_1}).
\end{align*}
Then
\[
\sum_{0\leq g_0\leq \left \lfloor (L+1)\log\log b\right \rfloor }\sum{}^{{}^{(g_0)}} =O(b^{2L}\phi(b)2^{-Lm_1}).\tag{145}
\]
To estimate $\sum{}_{{}_2}$ we write
\begin{equation}
w(j,l)=\left\{
\begin{array}{l l}
    s_j+lr\:, & \quad \text{if  $\:1\leq s_j+lr\leq b/2$}\vspace{2mm}\\ 
    b-(s_j+lr)\:, & \quad \text{if  $\:b/2<s_j+lr\leq b-1$}\:.\\
  \end{array} \right.\tag{146}
  \nonumber
\end{equation}
Because of the fact that $$\cot(\pi-x)=-\cot x\ \ \text{for all}\ x\in\mathbb{R},$$ we have
\[
\vline\: \cot\left(\pi\:\frac{s_j+lr}{b} \right)\:\vline=\cot\left(\pi\:\frac{w(j,l)}{b} \right).\tag{147}
\]
Recall that
\[
Q\left( \frac{r}{b}\right)=\sum_{j=0}^{r-1}j\sum_{l=0}^{d_j}\cot\left(\pi\frac{s_j+lr}{b}\right).\tag{148}
\]
\textit{\textbf{Step 6.1} We shall prove that}
$$Q_1\left( \frac{r}{b}\right)=O(b^2\log b).$$
From (70), (71) and (148), we have
\begin{align*}
\vline\:Q_1\left( \frac{r}{b}\right)\:\vline&\leq\sum_{j=0}^{r-1}j\sum_{l=0}^{d_j}\:\vline\:\cot\left(\pi\frac{s_j+lr}{b}\right)\:\vline\:\tag{149}\\
&=\sum_{j=0}^{r-1}j\sum_{l=0}^{d_j}\cot\left(\pi\:\frac{w(j,l)}{b}\right).
\end{align*}
From $(r,b)=1$ and (15), it follows that $s_j$ runs through a complete residue system $\bmod\: r$. Since by our assumption $r> b/2$, we have that $d_j\in\{0,1 \}$.
Therefore 
$$S_j=\{bj+s_j  \},\ \text{if}\ d_j=0$$
and
$$S_j=\{bj+s_j, bj+s_j+r \},\ \text{if}\ d_j=1.$$
Thus $w(j,l)$ assumes any integer value in the interval $[1,b/2]$ at most four times, since $S_j$ contains in every case described above at most two integers from
the interval $[1,b-1]$. Since $s_j$ and $s_j+r$ assume each integer value from the interval $[1,b-1]$ at most once, it follows that the numbers $s_j$ and $s_j+r$ taken together assume each integer value from the interval $[1,b-1]$ at most two times.\\
If $w(j,l)$ assumes an integer value $\lambda$ from the interval $[1,b/2]$, there are the following two possibilities:\\
$$w(j,l)=s_j+lr,\ \ w(j,l)=b-(s_j+lr).$$
From (149) and the fact that
$$\cot(\pi x)=O\left(\frac{1}{x} \right)\ \ (x\rightarrow 0),$$
it follows that
\begin{align*}
\vline\:Q_1\left( \frac{r}{b}\right)\:\vline&\leq 4b\sum_{j=1}^{\left\lfloor b/2\right\rfloor}\cot\left(\frac{\pi j}{b} \right) \leq 4b\sum_{j=1}^{r-1}\:\vline\:\cot\left(\frac{\pi j}{b} \right)\:\vline\tag{150}\\
&= O\left(b^2\sum_{j=1}^{r-1}\frac{1}{j} \right)=O(b^2\log b).
\end{align*}
This completes the proof of Step 6.1.\\ \\
From $$g_0(r)>\left\lfloor (L+1)\log\log b\right\rfloor$$ it follows that $$r\in\mathcal{E}(\left\lfloor (L+1)\log\log b\right\rfloor m_1).$$ By Lemma \ref{x:35} we get
\[
|\mathcal{E}(\left\lfloor (L+1)\log\log b\right\rfloor m_1)|=O(\phi(b)(\log b)^{-(L+1)})\tag{151}
\]
From (143), (150) and (151), we obtain
\begin{align*}
\sum{}_{{}_2}&=O(b^{2L}(\log b)^L\phi(b)(\log b)^{-(L+1)})\tag{152}\\
&=O(b^{2L}\phi(b)(\log b)^{-1})=O(b^{2L}\phi(b)2^{-Lm_1})
\end{align*}
since $m_1\in\mathbb{N}$ is fixed. From (141), (145) and (148), we obtain
\[
 \sum{}_{{}_1}=\sum_{\substack{r:(r,b)=1\\A_0b\leq r\leq A_1b}}Q_1\left(\frac{r}{b} \right)^L=O(b^{2L}\phi(b)2^{-Lm_1}).
\tag{153}
\]
This completes the proof of Step 6.\\

\textit{\textbf{Step 7.} We shall prove that}
\[
\sum_{\substack{r:(r,b)=1\\ A_0b\leq r\leq A_1b}}Q\left(\frac{r}{b}\right)^L=B\left(A_0,A_1\right)\frac{b^{2L}}{(L+1)\pi^L}\phi(b)\left(\int_0^1f(x)^Ldx\right)+O\left(b^{2L}\phi(b)2^{-m_1}\right),
\]
\textit{where}
$$ B(A_0,A_1)=A_1^{L+1}-A_0^{L+1}.$$
\textit{Proof of Step 7.} In the following we shall first deduce an asymptotic formula for 
$$\sum_{\substack{r:(r,b)=1\\A_0b\leq r\leq A_1b}}Q\left(r,b,m_1 \right)^L$$
for a fixed value of $m_1\in\mathbb{N}$.\\
In the sequel we shall write for simplicity $f(x)$ instead of $f(x;m_1)$.\\
The function $f$ is piecewise linear over the interval $[0,1)$ and therefore is integrable over $[0,1)$.\\
By a standard property of Riemann integration\index{Riemann!integration} for a given $\epsilon>0$ there exists a partition $\mathcal{P}$ with 
$$0=\alpha_0<\alpha_1<\cdots<\alpha_{n-1}<\alpha_n=1,$$ such that 
$$U(f^L,\mathcal{P})-I(f^L,\mathcal{P})<\epsilon$$
and
\[
I(f^L,\mathcal{P})\leq \int_0^1f(x)^Ldx\leq U(f^L,\mathcal{P}),\tag{154}
\]
where 
$$I(f^L,\mathcal{P})=\sum_{i=0}^{n-1}\:\inf_{\alpha^{(i)}\in[\alpha_i,\alpha_{i+1}]}f\left(\alpha^{(i)}\right)^L (\alpha_{i+1}-\alpha_i) $$
and
$$U(f^L,\mathcal{P})=\sum_{i=0}^{n-1}\:\sup_{\alpha^{(i)}\in[\alpha_i,\alpha_{i+1}]}f\left(\alpha^{(i)}\right)^L (\alpha_{i+1}-\alpha_i). $$
We denote by
\begin{align*}
N_i=\left\{r\::\: r\in\mathcal{R}_2,\:\frac{b^*}{r}\in[\alpha_i,\alpha_{i+1}]\right\}.
\end{align*}
By Lemma \ref{x:43}, we have
$$\left|N_i\right|=(\alpha_{i+1}-\alpha_i)(A_1-A_0)\phi(b)(1+o(1)),\ \ (b\rightarrow+\infty).$$
From (66), that is 
$$Q\left(r,b,m_1 \right)= \frac{br}{\pi}f\left(\frac{b^*}{r}\right),$$    
we obtain
\begin{align*}
\sum_{\substack{r:(r,b)=1\\ A_0b\leq r\leq A_1b}}\left(\frac{Q(r,b,m_1)}{r}\right)^L&=\frac{b^L}{\pi^L}\sum_{i=0}^{n-1}\sum_{r\in N_i}f\left(\frac{b^*}{r} \right)^L\tag{155}\\
&\geq \frac{b^L}{\pi^L}\sum_{i=0}^{n-1}\:\inf_{\alpha^{(i)}\in[\alpha_i,\alpha_{i+1}]}f\left(\alpha^{(i)}\right)^L|N_i|\\
&=\frac{b^L}{\pi^L}\sum_{i=0}^{n-1}\:\inf_{\alpha^{(i)}\in[\alpha_i,\alpha_{i+1}]}f\left(\alpha^{(i)}\right)^L (\alpha_{i+1}-\alpha_i) (A_1-A_0)\phi(b)(1+o(1))\\
&\geq \frac{b^L}{\pi^L}\left(\int_0^1f(x)^Ldx-\epsilon\right)(A_1-A_0)\phi(b)(1+o(1)).
\end{align*}
Similarly, we get
\begin{align*}
\sum_{\substack{r:(r,b)=1\\ A_0b\leq r\leq A_1b}}\left(\frac{Q(r,b,m_1)}{r}\right)^L&\leq  \frac{b^L}{\pi^L}\sum_{i=0}^{n-1}\:\sup_{\alpha^{(i)}\in[\alpha_i,\alpha_{i+1}]}f\left(\alpha^{(i)}\right)^L |N_i|\tag{156}\\
&=\frac{b^L}{\pi^L}\sum_{i=0}^{n-1}\:\sup_{\alpha^{(i)}\in[\alpha_i,\alpha_{i+1}]}f\left(\alpha^{(i)}\right)^L (\alpha_{i+1}-\alpha_i) (A_1-A_0)\phi(b)(1+o(1))\\
&\leq \frac{b^L}{\pi^L}\left(\int_0^1f(x)^Ldx+\epsilon\right)(A_1-A_0)\phi(b)(1+o(1)).
\end{align*}
From (154), (155) and (156) we have
\begin{align*}
\left|\sum_{\substack{r:(r,b)=1\\ A_0b\leq r\leq A_1b}}\left(\frac{Q(r,b,m_1)}{r}\right)^L-\frac{b^L}{\pi^L}\left(\int_0^1f(x)^Ldx\right)(A_1-A_0)\phi(b)\right|\leq K_{1,L}b^L\phi(b)\:\epsilon,\tag{157}
\end{align*}
where $K_{1,L}>0$ is a constant that depends only on $L$.\\
By Abel's partial summation\index{formula!Abel partial summation} (\cite{Pra}, Satz 1.4, page 371) we obtain
\begin{align*}
&\sum_{\substack{r:(r,b)=1\\ A_0b\leq r\leq A_1b}}Q(r,b,m_1)^L\tag{158}\\
&=(A_1b)^L\sum_{\substack{r:(r,b)=1\\ A_0b\leq r\leq A_1b}}\left(\frac{Q(r,b,m_1)}{r}\right)^L-L\int_{A_0b}^{A_1b}u^{L-1}\sum_{\substack{r:(r,b)=1\\ A_0b\leq r\leq u}}\left(\frac{Q(r,b,m_1)}{r}\right)^Ldu.
\end{align*}
\textit{We first consider the case when $L$ is even.}\\
We fix $\delta>0$ arbitrarily small and for $$u\geq A_0b(1+\delta),$$ 
we apply (157) replacing $A_1$ by $u/b$ and obtain
\[
\left|\sum_{\substack{r:(r,b)=1\\ A_0b\leq r\leq u}}\left(\frac{Q(r,b,m_1)}{r}\right)^L-\frac{b^L}{\pi^L}\left(\int_0^1f(x)^Ldx\right)\left(\frac{u}{b}-A_0\right)\phi(b)\right|\leq K_{1,L}b^L\phi(b)\:\epsilon.\tag{159}
\]
For 
$$A_0b\leq u\leq A_0b(1+\delta)$$ 
we have by (156) the following estimate
\begin{align*}\sum_{\substack{r:(r,b)=1\\ A_0b\leq r\leq u}}\left(\frac{Q(r,b,m_1)}{r}\right)^L &\leq \sum_{\substack{r:(r,b)=1\\ A_0b\leq r\leq A_1b}}\left(\frac{Q(r,b,m_1)}{r}\right)^L=O\left(b^L\phi(b)\left|\int_0^1f(x)^Ldx\right| \right)\tag{160}\\
&=O\left(b^L\phi(b) \right).
\end{align*}
From (158), (159) and (160), we obtain
\begin{align*}
&\sum_{\substack{r:(r,b)=1\\ A_0b\leq r\leq A_1b}}Q(r,b,m_1)^L\tag{161}\\
&=( A_1b)^L \frac{b^L}{\pi^L}\phi(b)\left(\int_0^1f(x)^Ldx\right)\left(A_1-A_0\right)-L\int_{A_0b}^{A_1b}u^{L-1}\frac{b^L}{\pi^L}\left(\int_0^1f(x)^Ldx\right)\left(\frac{u}{b}-A_0\right)\phi(b)du+\mathcal{R}\\
&=\frac{A_1^Lb^{2L}}{\pi^L}\phi(b)(A_1-A_0)\int_0^1f(x)^Ldx-\frac{Lb^L}{\pi^L}\left(\int_0^1f(x)^Ldx\right)\phi(b)\int_{A_0b}^{A_1b}\left(\frac{u^L}{b}-A_0u^{L-1}\right)du+\mathcal{R}\\
&=\frac{A_1^{L+1}-A_0A_1^L}{\pi^L}b^{2L}\phi(b)\int_0^1f(x)^Ldx-\frac{Lb^L}{\pi^L}\phi(b)\left(\int_0^1f(x)^Ldx\right)\left(\frac{u^{L+1}}{(L+1)b}-\frac{A_0u^L}{L}\right)\vline_{u=A_0b}^{A_1b}+\mathcal{R}\\
&=b^{2L}\phi(b)\left(\int_0^1f(x)^Ldx\right)\frac{A_1^{L+1}-A_0A_1^L}{\pi^L}\\
&\ \ \ \ \ \ -\frac{b^{2L}L}{\pi^L}\phi(b)\left(\int_0^1f(x)^Ldx\right)\left(\frac{A_1^{L+1}-A_0^{L+1}}{L+1}-\frac{A_0(A_1^L-A_0^L)}{L} \right)+\mathcal{R}\\
&=\left(\frac{b^{2L}}{(L+1)\pi^L}\phi(b)\int_0^1f(x)^Ldx\right)(A_1^{L+1}-A_0^{L+1})+\mathcal{R},
\end{align*}
where we have $K_{2,L}$ being a constant depending only on $L$:
$$|\mathcal{R}|\leq K_{2,L}(\epsilon+\delta)b^{2L}\phi(b)=O\left(b^L\phi(b)\right) .$$
\textit{We now consider the case when $L$ is odd.}\\
By the Cauchy-Schwarz inequality we get
$$\sum_{\substack{r:(r,b)=1\\A_0b\leq r\leq u}}\left| \frac{Q(r,b,m_1)}{r}\right|^L \leq \left( \sum_{\substack{r:(r,b)=1\\A_0b\leq r\leq u}}\left(\frac{Q(r,b,m_1)}{r} \right)^{2L}  \right)^{1/2} \left( \sum_{\substack{r:(r,b)=1\\A_0b\leq r\leq u}}1\right)^{1/2}.$$
By the estimate we just proved for the even values of $L$, we obtain
$$\sum_{\substack{r:(r,b)=1\\A_0b\leq r\leq u}}\left| \frac{Q(r,b,m_1)}{r}\right|^L  =  O\left( b^L\phi(b)\left(\int_0^1f(x)^{2L}dx\right)^{1/2}  \right)=O(b^L\phi(b)).$$
Therefore, the estimate (160), proven for even values of $L$ also holds for odd values of $L$.\\

By (69) and (72) we obtain by application of the multinomial theorem\index{theorem!multinomial} for any natural number $L$, the following 
\begin{align*}
&\sum_{\substack{r:(r,b)=1\\A_0b\leq r\leq A_1b}}Q\left(\frac{r}{b} \right)^L=\sum_{\substack{r:(r,b)=1\\A_0b\leq r\leq A_1b}}\left(Q(r,b,m_1)+O(b2^{m_1})+Q_1\left(\frac{r}{b}\right)\right)^L\tag{162}\\
&=\sum_{\substack{r:(r,b)=1\\A_0b\leq r\leq A_1b}}\sum_{\substack{(l_1,l_2,l_3)\in(\mathbb{N}\cup\{0\})^3\\ 0\leq l_1,l_2,l_3\leq L,\\l_1+l_2+l_3=L }}\frac{L!}{l_1!l_2!l_3!}Q(r,b,m_1)^{l_1}Q_1\left(\frac{r}{b}\right)^{l_2}\left(O\left(b2^{m_1}\right)\right)^{l_3}.
\end{align*}
We first apply H\"{o}lder's inequality to the products \index{inequality!H\"{o}lder's}
$$Q(r,b,m_1)^{l_1}Q_1\left(\frac{r}{b}\right)^{l_2}.$$
Let $\tilde{L}=L-l_3$. Then
\begin{align*}
&\sum_{\substack{r:(r,b)=1\\ A_0b\leq r\leq A_1b}}\left|Q(r,b,m_1)\right|^{l_1}\left|Q_1\left(\frac{r}{b}\right)\right|^{l_2}\\
& \leq \left(\sum_{\substack{r:(r,b)=1\\ A_0b\leq r\leq A_1b}}\left|Q(r,b,m_1)\right|^{\tilde{L}}\right)^{l_1/\tilde{L}}\left(\sum_{\substack{r:(r,b)=1\\ A_0b\leq r\leq A_1b}}\left|Q_1\left(\frac{r}{b}\right)\right|^{\tilde{L}}\right)^{l_2/\tilde{L}}.
\end{align*} 
By (153) and (161) we obtain
\begin{align*}
&\sum_{\substack{r:(r,b)=1\\ A_0b\leq r\leq A_1b}}\left|Q(r,b,m_1)\right|^{l_1}\left|Q_1\left(\frac{r}{b}\right)\right|^{l_2}\\
&= O\left( b^{2\tilde{L}}\phi(b)\left(\left|\int_0^1f(x)^{\tilde{L}}dx\right|+1\right)^{l_1/\tilde{L}}\left( b^{2\tilde{L}}\phi(b)2^{-\tilde{L}m_1}\right)^{l_2/\tilde{L}}\right)\\
&=O\left(b^{2\tilde{L}}\phi(b)2^{-m_1}\right).
\end{align*}
We obtain the same estimate in the case $l_1=0$. Thus, all the terms of (162), for which $(l_1,l_2,l_3)\neq (L,0,0)$ may be estimated by
$O(b^{2L}\phi(b)2^{-m_1})$ and we thus obtain from formula (161):
\[
\sum_{\substack{r:(r,b)=1\\ A_0b\leq r\leq A_1b}}Q\left(\frac{r}{b}\right)^L=B\left(A_0,A_1\right)\frac{b^{2L}}{(L+1)\pi^L}\phi(b)\left(\int_0^1f(x)^Ldx\right)+O\left(b^{2L}\phi(b)2^{-m_1}\right),\tag{163}
\]
where
$$ B(A_0,A_1)=A_1^{L+1}-A_0^{L+1}.$$
This completes the proof of Step 7.\\

Set $L=2k$, $k\in\mathbb{N}$.\\
Letting $m_1\rightarrow +\infty$ we know that
$$D_L:=\lim_{m_1\rightarrow +\infty}\int_0^1f(x)^Ldx>0,$$
due to Lemma \ref{x:l222} and Theorem \ref{x:positive}.\\
We have 
$$f(x)=\sum_{s=1}^{2^{m_1}}\frac{1-2\{sx\}}{s}=\sum_{n=-\infty}^{+\infty}a(n)e(nx).$$
From (163) and the fact that $m_1$ can be chosen arbitrarily large we obtain
\begin{align*}
\sum_{\substack{r:(r,b)=1\\A_0b\leq r\leq A_1b}}Q\left(\frac{r}{b} \right)^{2k}=D_{2k}\cdot (A_1^{2k+1}-A_0^{2k+1})\frac{b^{4k}}{(2k+1)\pi^{2k}}\phi(b)\:(1+o(1)),\ \ (b\rightarrow+\infty),
\end{align*}
which proves Theorem \ref{x:44}, that is part (a) of Theorem \ref{x:QQ}, by setting $$E_k=D_{2k}/((2k+1)\pi^{2k}).$$
\end{proof}
\begin{theorem}\label{x:455}
Let $k\in\mathbb{N}$ be fixed. Let also $A_0,A_1$ be fixed constants such that $1/2<A_0<A_1<1$. Then we have
$$\sum_{\substack{r:(r,b)=1\\A_0b\leq r\leq A_1b}}Q\left(\frac{r}{b}\right)^{2k-1}=o\left(b^{4k-2}\phi(b)\right),\ \ (b\rightarrow+\infty).$$
\end{theorem}
\begin{proof}
Set $L=2k-1$, $k\in\mathbb{N}$.\\
We have defined  
$$f(x)=\sum_{s=1}^{2^{m_1}}\frac{1-2\{sx\}}{s}.$$
It follows that
$$f(x)=-f(1-x),\ \ \text{if}\ \ x\in[0,1].$$
Therefore,
$$\int_0^1f(x)^{2k-1}dx=0.$$
From (163) we get
$$\sum_{\substack{r:(r,b)=1\\A_0b\leq r\leq A_1b}}Q\left(\frac{r}{b}\right)^{2k-1}=O\left(b^{2(2k-1)}\phi(b)2^{-m_1}\right)=O\left(b^{4k-2}\phi(b)2^{-m_1}\right).$$
Since $m_1$ can be chosen arbitrarily large we obtain
$$\sum_{\substack{r:(r,b)=1\\A_0b\leq r\leq A_1b}}Q\left(\frac{r}{b}\right)^{2k-1}=o\left(b^{4k-2}\phi(b)\right),\ (b\rightarrow+\infty)$$
Thus, we have proved the theorem (that is part (b) of Theorem \ref{x:QQ}).
%
%
\end{proof}
\begin{theorem}\label{x:466}
Let $k\in\mathbb{N}$ be fixed. Let also $A_0,$ $A_1$ be fixed constants such that $1/2<A_0<A_1<1$. Then there exists a constant $H_k>0$, depending only on $k$, such that 
$$\sum_{\substack{r:(r,b)=1\\A_0b\leq r\leq A_1b}}c_0\left(\frac{r}{b} \right)^{2k}=H_k\cdot(A_1-A_0)b^{2k}\phi(b)(1+o(1)),\ \ (b\rightarrow+\infty).$$
\end{theorem}
\begin{proof}
From Proposition \ref{x:vasiko} for $r,b\in\mathbb{N}$ with $(r,b)=1$, it holds
$$c_0\left(\frac{r}{b}\right)=\frac{1}{r}\:c_0\left(\frac{1}{b}\right)-\frac{1}{r}Q\left(\frac{r}{b}\right)\:.$$
Applying Theorem \ref{x:rassiass} we obtain
$$c_0\left(\frac{r}{b}\right)=-\frac{1}{r}Q\left(\frac{r}{b}\right)+O(\log b)\:.$$
By the binomial theorem we get: 
\begin{align*}
\sum_{\substack{r:(r,b)=1\\A_0b\leq r\leq A_1b}}c_0\left(\frac{r}{b} \right)^{2k}&=\sum_{\substack{r:(r,b)=1\\A_0b\leq r\leq A_1b}}\left(\frac{Q\left(\frac{r}{b}\right)}{r}\right)^{2k}\tag{164}\\
&+O\left( \sum_{l=1}^{2k}\binom{2k}{l}\sum_{\substack{r:(r,b)=1\\A_0b\leq r\leq A_1b}}\left|Q\left(\frac{r}{b}\right)\right|^{2k-l}\left(\log b\right)^l\right).
\end{align*}
By H\"{o}lder's inequality, we get\index{inequality!H\"{o}lder's}
\[
\sum_{\substack{r:(r,b)=1\\A_0b\leq r\leq A_1b}}\left|\frac{Q\left(\frac{r}{b}\right)}{r}\right|^{2k-l}\leq \left(\sum_{\substack{r:(r,b)=1\\A_0b\leq r\leq A_1b}}\left|\frac{Q\left(\frac{r}{b}\right)}{r}\right|^{2k}\right)^{(2k-l)/2k}\left(\sum_{\substack{r:(r,b)=1\\A_0b\leq r\leq A_1b}}1\right)^{l/{2k}}.
\]
Therefore, by (156) we have
\[
\sum_{\substack{r:(r,b)=1\\A_0b\leq r\leq A_1b}}\left|\frac{Q\left(\frac{r}{b}\right)}{r}\right|^{2k-l}=O\left(b^{2k-l}\phi(b)\right).\tag{165}
\]
From (164) and (165), we obtain
\[
\sum_{\substack{r:(r,b)=1\\A_0b\leq r\leq A_1b}}c_0\left(\frac{r}{b} \right)^{2k}=\sum_{\substack{r:(r,b)=1\\A_0b\leq r\leq A_1b}}\left(\frac{Q\left(\frac{r}{b}\right)}{r}\right)^{2k}+O\left(b^{2k-1}\phi(b)\right).\tag{166}
\]
Using Abel's partial summation\index{formula!Abel partial summation} it follows that
\[
\sum_{\substack{r:(r,b)=1\\A_0b\leq r\leq A_1b}}c_0\left(\frac{r}{b} \right)^{2k}=(A_1b)^{-2k}\sum_{\substack{r:(r,b)=1\\A_0b\leq r\leq A_1b}}Q\left(\frac{r}{b} \right)^{2k}+2k\int_{A_0b}^{A_1b}u^{-(2k+1)}\sum_{\substack{r:(r,b)=1\\A_0b\leq r\leq u}}Q\left(\frac{r}{b} \right)^{2k}du.\tag{167}
\]
By Theorem \ref{x:44} we obtain
\[
\sum_{\substack{r:(r,b)=1\\A_0b\leq r\leq u}}Q\left(\frac{r}{b} \right)^{2k}=E_k\cdot\left( \left(\frac{u}{b}\right)^{2k+1}-A_0^{2k+1}\right)b^{4k}\phi(b)(1+o(1)).\tag{168}
\]
From (167) and (168) we get
\begin{align*}
\sum_{\substack{r:(r,b)=1\\A_0b\leq r\leq A_1b}}c_0\left(\frac{r}{b} \right)^{2k}&=E_k\cdot   \left(A_1b\right)^{-2k}\left( A_1^{2k+1}-A_0^{2k+1}\right)b^{4k}\phi(b)(1+o(1))\tag{169}\\
&+2kE_k\cdot\left(\int_{A_0b}^{A_1b}u^{-(2k+1)}\left( \left(\frac{u}{b}\right)^{2k+1}-A_0^{2k+1}\right)du\right)b^{4k}\phi(b)(1+o(1)).
\end{align*}
If we make the substitution $v=u/b$ in (169) we get
\begin{align*}
\sum_{\substack{r:(r,b)=1\\A_0b\leq r\leq A_1b}}c_0\left(\frac{r}{b} \right)^{2k}&=E_k\cdot A_1^{-2k}(A_1^{2k+1}-A_0^{2k+1})b^{2k}\phi(b)(1+o(1))\\
&+2kE_k\cdot\left(\int_{A_0}^{A_1}v^{-(2k+1)}(v^{2k+1}-A_0^{2k+1})dv\right)b^{2k}\phi(b)(1+o(1))\\
&=E_k\cdot \left( A_1-A_1^{-2k}A_0^{2k+1}\right)b^{2k}\phi(b)(1+o(1))\\
&+2kE_k\cdot\left( \int_{A_0}^{A_1}\left(1-A_0^{2k+1}v^{-(2k+1)}\right)dv\right)b^{2k}\phi(b)(1+o(1))\\
&=(2k+1)E_k\cdot \left(A_1-A_0\right)b^{2k}\phi(b)(1+o(1)),\ \ (b\rightarrow+\infty).
\end{align*}
Theorem \ref{x:466}, that is part (c) of Theorem \ref{x:QQ}, now follows by setting $$H_k=(2k+1)E_k.$$
\noindent \textbf{Remark.} From the above theorem it follows that
$$H_k=\frac{D_{2k}}{\pi^{2k}}=\int_0^1\left(\frac{g(x)}{\pi}\right)^{2k}dx,$$
where
$$g(x)=\sum_{l=1}^{+\infty}\frac{1-2\{lx\}}{l}.$$
\end{proof}
\begin{theorem}\label{x:477}
Let $k\in\mathbb{N}$ be fixed. Let also $A_0,A_1$ be fixed constants such that $1/2<A_0<A_1<1$. Then we have
$$\sum_{\substack{r:(r,b)=1\\A_0b\leq r\leq A_1b}}c_0\left(\frac{r}{b}\right)^{2k-1}=o\left(b^{2k-1}\phi(b)\right),\ \ (b\rightarrow+\infty).$$
\end{theorem}
\begin{proof}
In the formulas (164), (165) and (166) from the proof of Theorem \ref{x:466} we replace the exponent $2k$ by $2k-1$ and obtain
\[
\sum_{\substack{r:(r,b)=1\\A_0b\leq r\leq A_1b}}c_0\left(\frac{r}{b} \right)^{2k-1}=\sum_{\substack{r:(r,b)=1\\A_0b\leq r\leq A_1b}}\left(\frac{Q\left(\frac{r}{b}\right)}{r}\right)^{2k-1}+O(b^{2k-2}\phi(b)).\tag{170}
\]
Using Abel's partial summation\index{formula!Abel partial summation} we get
\[
\sum_{\substack{r:(r,b)=1\\A_0b\leq r\leq A_1b}}c_0\left(\frac{r}{b} \right)^{2k-1}=(A_1b)^{-(2k-1)}\sum_{\substack{r:(r,b)=1\\A_0b\leq r\leq A_1b}}Q\left(\frac{r}{b} \right)^{2k-1}+(2k-1)\int_{A_0b}^{A_1b}u^{-2k}\sum_{\substack{r:(r,b)=1\\A_0b\leq r\leq u}}Q\left(\frac{r}{b} \right)^{2k-1}du.\tag{171}
\]
By Theorem \ref{x:455} we obtain
\[
\sum_{\substack{r:(r,b)=1\\A_0b\leq r\leq u}}Q\left(\frac{r}{b} \right)^{2k-1}=o(b^{4k-2}\phi(b)),\ (b\rightarrow+\infty).\tag{172}
\]
Thus, Theorem \ref{x:477} (that is part (d) of Theorem \ref{x:QQ}) follows from the formulas (171) and (172) by substitution.
\end{proof}
\vspace{10mm}
\section{Probabilistic distribution}
\vspace{5mm}
\begin{definition}\label{x:pdd1}
For $z\in\mathbb{R}$, let 
$$F(z)=\text{meas}\{\alpha\in[0,1]\::\: g(\alpha)\leq z  \}$$
with
$$g(\alpha)=\sum_{l=1}^{+\infty}\frac{1-2\{l\alpha\}}{l}$$  
and 
$$C_0(\mathbb{R})=\{f\in C(\mathbb{R})\::\: \forall\: \epsilon>0,\: \exists\:\text{a compact set}\ \mathcal{K}\subset\mathbb{R},\:\text{such that}\ |f(x)|<\epsilon,\forall\: x\not\in \mathcal{K}  \},$$
where \text{``meas"} denotes the Lebesgue measure.\index{measure!Lebesgue}
\end{definition}
\begin{theorem}\label{x:pithanotiko}
i) $F$ is a continuous function of $z$.\\
ii) Let $A_0,$ $A_1$ be fixed constants, such that $1/2< A_0<A_1<1$. Let also
$$H_k=\int_0^1\left(\frac{g(x)}{\pi}\right)^{2k}dx.$$ 
There is a unique positive measure $\mu$ on $\mathbb{R}$ with the following properties:\\
\ (a) For $\alpha<\beta\in\mathbb{R}$ we have
$$\mu([\alpha,\beta])=(A_1-A_0)(F(\beta)-F(\alpha)).$$
(b)
\begin{equation}
\int x^kd\mu=\left\{
\begin{array}{l l}
    (A_1-A_0)H_{k/2}\:, & \quad \text{for even}\: k\\
    0\:, & \quad \text{otherwise}\:.\\
  \end{array} \right.
 \nonumber
\end{equation}
(c) For all $f\in C_0(\mathbb{R})$, we have
$$\lim_{b\rightarrow+\infty}\frac{1}{\phi(b)}\sum_{\substack{r\::\: (r,b)=1\\ A_0b\leq r\leq A_1b}}f\left( \frac{1}{b}c_0\left( \frac{r}{b}\right) \right)=\int f\:d\mu,$$
where $\phi(\cdot)$ denotes the Euler phi-function.
\end{theorem}
\begin{definition}
A distribution function $G$\index{function!distribution} is a monotonically increasing function $$G\!:\!\mathbb{R}\rightarrow[0,1].$$\
The characteristic function $\psi$ of $G$ is defined by the following Stieltjes integral:\index{integral!Stieltjes}
$$\psi(t)=\int_{-\infty}^{+\infty} e^{itu}dG(u).$$
(cf. \cite{elliott}, p.27)
\end{definition}
\begin{lemma}\label{x:pl1}
The distribution function $G$ is continuous if and only if the characteristic function $\psi$ satisfies
$$\liminf_{T\rightarrow+\infty}\frac{1}{2T}\int_{-T}^{T}|\psi(t)|^2dt=0.$$
\end{lemma}
\begin{proof}
See \cite{elliott}, p. 48, Lemma 1.23.
\end{proof}
\begin{definition}
Let $t\geq 1$. We set
$$K=K(t)=\lfloor t^{9/10} \rfloor,\ L=L(t)=\lfloor t^{11/10} \rfloor,\ R=R(t)=\lfloor t^{9/5} \rfloor   $$
and
$$g(\alpha,K)=-2\sum_{l\leq K}\frac{B^*(l\alpha)}{l},\ \ h(\alpha)=-2\sum_{l>K}\frac{B^*(l\alpha)}{l},$$
where $B^*(u)=u-\lfloor u\rfloor-1/2$, $u\in\mathbb{R}$.\\
\noindent Assume that $(\alpha_i)$ with $0=\alpha_0<\alpha_1<\cdots<\alpha_R=1$ is a partition of $[0,1]$ with the following properties:
$$\frac{1}{2}R^{-1}\leq \alpha_{i+1}-\alpha_i\leq 2R^{-1}$$
and $g(\alpha,K)$ is continuous at $\alpha=\alpha_i$ for $0<i<R$.
\end{definition}
\noindent We now make preparations for an application of Lemma \ref{x:pl1} with $G=F$, and 
$$\psi(t)=\Phi(t):=\int_0^1e\left(\frac{tg(\alpha)}{2\pi}\right)d\alpha.$$
\begin{lemma}\label{x:pl2}
The function $h(\alpha)$ has a Fourier expansion
$$h(\alpha)=\sum_{n>K}c(n)\sin(2\pi n\alpha),$$
with
$$|c(n)|\leq \frac{2\tau(n)}{\pi n},$$
where $\tau$ stands for the divisor function.
\end{lemma}
\begin{proof}
From the Fourier expansion 
$$B^*(u)=\frac{i}{2\pi}\sum_{\substack{n=-\infty \\ n\neq0}}^{+\infty}\frac{e(nu)}{n},$$
we obtain
$$h(\alpha)=-\frac{i}{\pi}\sum_{l>K}\frac{1}{l}\sum_{\substack{m=-\infty \\ m\neq 0}}^{+\infty}\frac{e(lm\alpha)}{m}=\sum_{|n|>K}d(n)e(n\alpha)$$
with
$$d(n)=-\frac{i}{\pi n}\left|\{ (l,m)\::\: lm=n,\ l>K\} \right|.$$
We have 
$$h(\alpha)=\sum_{n>K}d(n)\left(e(n\alpha)-e(-n\alpha)\right)=2i\sum_{n>K}d(n)\sin(2\pi n\alpha),$$
which completes the proof of the lemma.
\end{proof}
\begin{definition}
We set
$$h_1(\alpha):=\sum_{K< n \leq L}c(n)\sin(2\pi n\alpha)$$
and
$$h_2(\alpha):=\sum_{n > L}c(n)\sin(2\pi n\alpha).$$
\end{definition}
\begin{lemma}\label{x:pl3}
We have
$$\int_0^1\left(e\left(\frac{t}{2\pi}\left(g(\alpha,K)+h_1(\alpha)\right)\right)-e\left(\frac{tg(\alpha)}{2\pi}\right)\right)d\alpha=O\left(t^{-1/100}\right).$$
\end{lemma}
\begin{proof}
By Parseval's identity\index{Parseval's identity}, it follows that for every $\epsilon>0$ it holds
$$\int_0^1h_2(\alpha)^2d\alpha=\sum_{n>L}c(n)^2\ll L^{-(1-2\epsilon)},$$
because of the estimate $$c(n)\ll n^{-1+\epsilon}.$$
\noindent Thus, for all $\alpha\in[0,1]$ not belonging to an exceptional set $\mathcal{E}$ with $$\text{meas}\mathcal{(E)}=O\left(t^{-1/100}\right),$$ we have
$$h_2(\alpha)=O\left(t^{-1-1/100}\right)$$
and therefore
$$\left|e\left(\frac{th_2(\alpha)}{2\pi}\right)-1\right|=O\left(t^{-1/100}\right)$$
by the Taylor expansion of the exponential function.\\
\noindent Hence,
\begin{align*}
&\left| \int_0^1e\left(\frac{tg(\alpha)}{2\pi}\right)d\alpha-\int_0^1e\left(\frac{t}{2\pi}(g(\alpha,K)+h_1(\alpha))\right)d\alpha\right|\\
&\leq\int_0^1\left|e\left(\frac{t(g(\alpha,K)+h_1(\alpha))}{2\pi}\right)\right|\left|e\left(\frac{th_2(\alpha)}{2\pi}\right)-1\right|d\alpha\leq \int_{\mathcal{E}}2\:d\alpha+\int_{[0,1]\setminus\mathcal{E}}\left|e\left(\frac{th_2(\alpha)}{2\pi}\right)-1\right|\:d\alpha\\
&=O\left(t^{-1/100}\right).
\end{align*}
\end{proof}
\begin{lemma}\label{x:pl4}
There exists a set $I\subseteq\{1,\ldots,R \}$ of non-negative integers, such that 
$$\sum_{i\in I}(\alpha_{i+1}-\alpha_i)=O\left(t^{-1/100}\right)$$
and for $i\not\in I$, $\alpha\in[\alpha_i,\alpha_{i+1}]$ we have
$$|h_1(\alpha)-h_1(\alpha_i)|\leq t^{-(1+1/100)}.$$
\end{lemma}
\begin{proof}
We have
$$\frac{d}{d\alpha}h_1(\alpha)=\sum_{K< n\leq L}2\pi nc(n)\cos(2\pi n\alpha)$$
and
$$\frac{d^2}{d\alpha^2}h_1(\alpha)=-\sum_{K< n\leq L}4\pi^2 n^2c(n)\sin(2\pi n\alpha).$$
By Parseval's identity,\index{Parseval's identity} for every $\epsilon>0$ we get
$$\int_0^1\left|\frac{d}{d\alpha}h_1(\alpha) \right|^2d\alpha=O\left(L^{1+2\epsilon}\right)$$
and by the Cauchy-Schwarz inequality,\index{inequality!Cauchy-Schwarz} it follows that
\[
\int_0^1\left|\frac{d}{d\alpha}h_1(\alpha) \right|d\alpha=O\left(L^{1/2+\epsilon}\right).\tag{173}
\]
We now define the set $I$ as the set of all subscripts $i$ for which the closed interval $[\alpha_i,\alpha_{i+1}]$ contains an $\alpha$ with
$$|h_1(\alpha)-h_1(\alpha_i)|>t^{-(1+1/100)}.$$
Since
\[
h_1(\alpha)=h_1(\alpha_i)+\int_{\alpha_i}^{\alpha}\frac{d}{d\beta}h_1(\beta)d\beta\tag{*}
\]
and 
$$|\alpha-\alpha_i|=O\left( t^{-9/5}\right),$$
it follows that for $i\in I$ there must exist $\beta\in(\alpha_i,\alpha_{i+1})$ with
$$\left|\frac{d}{d\beta}h_1(\beta) \right|\geq t^{3/5}.$$
Because of the estimation of the Fourier coefficients of $\frac{d^2}{d\alpha^2}h_1(\alpha)$, we obtain
$$\left|\frac{d^2}{d\alpha^2}h_1(\alpha)\right|=O\left(L^{2+\epsilon}\right).$$
Analogously to (*) we obtain that
$$\left| \frac{d}{d\alpha}h_1(\alpha)\right|\geq\frac{1}{2}t^{3/5},$$
for every $\alpha\in[\alpha_i,\alpha_{i+1}]$ and therefore
$$\int_{a_i}^{a_{i+1}}\left| \frac{d}{d\alpha}h_1(\alpha)\right| d\alpha\geq\frac{1}{2}t^{3/5}(\alpha_{i+1}-\alpha_i).$$
From (173) we obtain that the measure of the union of the closed intervals $[\alpha_i,\alpha_{i+1}]$ with $i\in I$ is $O(t^{-1/100})$, which concludes the proof of the lemma.
\end{proof}
\begin{lemma}\label{x:pl5}
We have
$$\lim_{t\rightarrow+\infty}\Phi(t)=\lim_{t\rightarrow-\infty}\Phi(t)=0.$$
\end{lemma}
\begin{proof}
We shall prove the result only for $t\rightarrow+\infty$, since the proof of the part when $t\rightarrow-\infty$ is analogous.\\
By Lemma \ref{x:pl3}, we have
$$\Phi(t)=\int_0^1e\left(\frac{tg(\alpha)}{2\pi}\right)d\alpha=\int_0^1e\left(\frac{t}{2\pi}(g(\alpha,K)+h_2(\alpha))\right)+O\left( t^{-1/100}\right)$$
and thus
\begin{align*}
\Phi(t)=\int_0^1e\left(\frac{tg(\alpha)}{2\pi}\right)d\alpha&=\sum_{\substack{i=0\\i\not\in I}}^Re\left(\frac{th_1(\alpha_i)}{2\pi}\right)\int_{\alpha_i}^{\alpha_{i+1}}e\left(\frac{tg(\alpha,K)}{2\pi}\right)d\alpha\\
&+\sum_{\substack{i=0\\i\not\in I}}^R\int_{\alpha_i}^{\alpha_{i+1}}e\left(\frac{tg(\alpha,K)}{2\pi}\right)\left(e\left(\frac{th_1(\alpha)}{2\pi}\right)-e\left(\frac{th_1(\alpha_i)}{2\pi}\right)\right)d\alpha\\
&+O\left( \sum_{i\in I}(\alpha_{i+1}-\alpha_i)\right)+O\left( t^{-1/100}\right).
\end{align*}
From Lemma \ref{x:pl4} we get
\[
\Phi(t)=\int_0^1e\left(\frac{tg(\alpha)}{2\pi}\right)d\alpha=\sum_{\substack{i=0\\i\not\in I}}^Re\left(\frac{th_1(\alpha_i)}{2\pi}\right)\int_{\alpha_i}^{\alpha_{i+1}}e\left(\frac{tg(\alpha,K)}{2\pi}\right)\:d\alpha+O\left( t^{-1/100}\right).\tag{174}
\]
We now estimate
$$\int_{\alpha_i}^{\alpha_{i+1}}e\left(\frac{tg(\alpha,K)}{2\pi}\right)d\alpha,$$
for $i\not\in I$.
\noindent Let $J_i-1$ be the number of discontinuities of the function $g(\alpha,K)$ in the interval $[\alpha_i,\alpha_{i+1}]$. Let $\beta_{i,0}=\alpha_i$, $\beta_{i,J_i}=\alpha_{i+1}$ and let the discontinuities of $g(\alpha,K)$ in $[\alpha_i,\alpha_{i+1}]$ occur at the points $\beta_{i,1}<\beta_{i,2}<\cdots<\beta_{i,J_i-1}$.\\
In the intervals $[\beta_{i,r}\:,\:\beta_{i,r+1}]$ the function $g(\alpha,K)$ is a \text{linear function}, that is 
$$g(\alpha,K)=d_r-2K\alpha,$$
where $d_r\in\mathbb{R}$.
Therefore,
\begin{align*}
\left|\int_{\alpha_i}^{\alpha_{i+1}}e\left(\frac{tg(\alpha,K)}{2\pi}\right)d\alpha\right|&\leq \sum_{r=0}^{J_i}\left|\int_{\beta_{i,r}}^{\beta_{i,r+1}}e\left(\frac{tg(\alpha,K)}{2\pi}\right)d\alpha  \right|\\
&\leq\sum_{r=0}^{J_i}\left|\int_{\beta_{i,r}}^{\beta_{i,r+1}}e\left(-\frac{tK\alpha}{\pi}\right)d\alpha  \right|\\
&=O\left(J_i(tK)^{-1} \right).\tag{175}
\end{align*}
From (174) and (175), we get
\begin{align*}
\int_0^1e\left(\frac{tg(\alpha)}{2\pi}\right)d\alpha&\leq \sum_{\substack{i=0\\i\not\in I}}^R\left|\int_{\alpha_i}^{\alpha_{i+1}}e\left(\frac{tg(\alpha,K)}{2\pi}\right)d\alpha\right|+O\left(t^{-1/100} \right)\\
&=O\left((tK)^{-1} \sum_{\substack{i=0\\i\not\in I}}^RJ_i \right)+O\left(t^{-1/100} \right).
\end{align*}
The number of discontinuities of $g(\alpha,K)$ is $O(K^2)$, since each of the $K$ terms 
$$\frac{B^*(l\alpha)}{l}$$
has $O(K)$ discontinuities in the interval $[0,1]$. We thus have
$$\sum_{i=0}^RJ_i=O(K^2).$$
Then
$$\Phi(t)=O\left(t^{-1/100} \right).$$
Therefore
$$\lim_{t\rightarrow+\infty}\Phi(t)=0.$$
Similarly, we obtain 
$$\lim_{t\rightarrow-\infty}\Phi(t)=0,$$
which completes the proof of the lemma.
\end{proof}
%
%
%
%
\begin{lemma}\label{x:pl6}
$F$ is a continuous function of $z$.
\end{lemma}
\begin{proof}
This follows from Lemma \ref{x:pl1} and Lemma \ref{x:pl5}.\\
\noindent Thus, part (i) of Theorem \ref{x:pithanotiko} is now proved. 
\end{proof}
\noindent In the following we will prove part (ii) of Theorem \ref{x:pithanotiko}.
\begin{definition}\label{x:lkef}
Let $f\!:\!\mathbb{R}\rightarrow\mathbb{R}$. We set
$$\Lambda(f,b):=\frac{1}{\phi(b)}\sum_{\substack{r:(r,b)=1\\A_0b\leq r\leq A_1b}}f\left(\frac{1}{b}c_0\left(\frac{r}{b}\right)  \right).$$
We also set
$$\Lambda(f):=\lim_{b\rightarrow+\infty}\frac{1}{\phi(b)}\sum_{\substack{r:(r,b)=1\\A_0b\leq r\leq A_1b}}f\left(\frac{1}{b}c_0\left(\frac{r}{b}\right)  \right),$$
for all $f$ for which the right hand side exists in $\mathbb{R}$.
\end{definition}
\begin{lemma}\label{x:pl7}
Let $\alpha<\beta\in\mathbb{R}$, $I=[\alpha,\beta)$. The characteristic function $\chi(\:.\:;I)$ is defined by
\begin{eqnarray}
\chi(u;I)=\left\{ 
  \begin{array}{l l}
   1\:, & \quad \text{if $u\in I$}\vspace{2mm}\\ 
   0\:, & \quad \text{otherwise}\:.\\
  \end{array} \right.
\nonumber
\end{eqnarray}
Then
$$\Lambda(\chi)=(A_1-A_0)(F(\beta)-F(\alpha))=(A_1-A_0)\int_\alpha^\beta\chi(u;I)dF(u).$$
\end{lemma}
\begin{proof}
Let $\epsilon>0$ be fixed but arbitrarily small. Let 
\[
\alpha<\frac{1}{b}c_0\left(\frac{r}{b}\right)<\beta.\tag{176}
\]
For simplicity we restrict ourselves to the case $\alpha>0$, since the case $\alpha<0$ can be treated similarly.\\
By Proposition \ref{x:vasiko}, we have
$$c_0\left(\frac{r}{b}\right)=\frac{1}{r}c_0\left(\frac{1}{b}\right)-\frac{1}{r}Q\left(\frac{r}{b}\right)$$
and by Theorem \ref{x:rassiass} we know that
\[
c_0\left(\frac{1}{b}\right)=O(b\log b).\tag{177}
\]
We first assume that $r$ does not belong to the exceptional set $\mathcal{E}(m_1)$, which by Lemma \ref{x:35} satisfies
\[
|\mathcal{E}(m_1)|=O\left(\phi(b)2^{-m_1} \right).\tag{178}
\]
From (177) it follows that 
$$\frac{1}{r}c_0\left(\frac{1}{b}\right)=O(\log b),$$
since $A_0b\leq r\leq A_1b$.\\
Thus from (176), for sufficiently large $b$ it follows that
\[
-br\beta(1+\epsilon)<Q\left(\frac{r}{b}\right)<-br\alpha(1-\epsilon).\tag{179}
\]
We recall the relations (69) and (72), namely
\[
Q\left(\frac{r}{b}\right)=Q_0\left(\frac{r}{b}\right)+Q_1\left(\frac{r}{b}\right)
\]
and
\[
Q_0\left(\frac{r}{b}\right)=Q(r,b,m_1)+O(b2^{m_1}),\tag{180}
\]
respectively, where
$$Q(r,b,m_1)=\frac{br}{\pi}\sum_{s=1}^{2^{m_1}}\frac{1-2\{s\xi\}}{s}=\frac{br}{\pi}f(\xi,m_1).$$
The value of $f(\xi,m_1)$ can be approximated by confining $\xi$ to a union of intervals, which we shall describe below. Since by the 
relation (65), we have
$$\xi=\xi(r,b)=\frac{b^*}{r},\ bb^*\equiv 1(\bmod\:r),$$
this leads to the problem of counting the number of $r$-values for which $b^*/r$ lies in a certain interval. This can be done by the estimate
for the number $N(\xi,j)$, which has been carried out in Lemma \ref{x:43}.\\
By (138), we have
\[
Q_1\left(\frac{r}{b}\right)=O(b^22^{-m_1})\tag{181}
\]
From (179), (180) and (181) it follows that
\[
-\beta(1+2\epsilon)\leq f(\xi,m_1)\leq -\alpha(1-2\epsilon).\tag{182}
\]
Since the function $f(x,m_1)$ is piecewise linear\index{linear!piecewise}, there exist disjoint closed intervals $I_1,\ldots,I_z$, with $I_j\subseteq [0,1]$, where the integer $z$ does not depend on $b$, such that
$$-\beta(1+2\epsilon)\leq f(\xi,m_1)\leq -\alpha(1-2\epsilon)$$
if and only if 
$$\xi\in \bigcup_{j=1}^z I_j.$$
Since 
$$\lim_{m_1\rightarrow +\infty}\int_0^1(f(x,m_1)-g(x))^2 dx=0,$$
for sufficiently large $m_1$, we have for the sum of the lengths of the intervals $I_j$, $1\leq j\leq z$, that
\[
\sum_{j=1}^{z}|I_j|\leq \text{meas}\left\{ x\in[0,1]\::\: -\beta(1+3\epsilon)\leq g(x)\leq -\alpha(1-\epsilon) \right\}+2\epsilon.  \tag{183}
\]
Let 
$$N(\xi,j):=\left|\left\{ r\::\: r\in\mathbb{N}, (r,b)=1,\: A_0b\leq r\leq A_1b,\: \frac{b^*}{r}\in I_j \right\}  \right|.$$
By Lemma \ref{x:43} we have
\[
N(\xi,j)=(A_1-A_0)|I_j|\phi(b)(1+o(1)).\tag{184}
\]
From (178), (183) and (184), we get
\begin{align*}
&\frac{1}{\phi(b)}\left|\left\{ r\::\:  (r,b)=1,\: A_0b\leq r\leq A_1b\ \text{with}\ \alpha<\frac{1}{b}c_0\left(\frac{r}{b}\right)<\beta\right\}\right|\\
&\leq (A_1-A_0)\bigg( \text{meas}\left\{x\in [0,1]\::\: -\beta(1+3\epsilon)\leq g(x)\leq -\alpha(1-\epsilon)  \right\}+3\epsilon\bigg).
\end{align*}
Because of the continuity of $F$, we have for arbitrarily small $\epsilon>0$ and sufficiently large $b$, the following:
\[
\frac{1}{\phi(b)}\left|\left\{ r\::\:  (r,b)=1,\: A_0b\leq r\leq A_1b\ \text{with}\ \alpha<\frac{1}{b}c_0\left(\frac{r}{b}\right)<\beta\right\}\right|\leq (A_1-A_0)(F(\beta)-F(\alpha))+\epsilon.\tag{185}
\]
In a similar manner, we get
\[
\frac{1}{\phi(b)}\left|\left\{ r\::\:  (r,b)=1,\: A_0b\leq r\leq A_1b\ \text{with}\ \alpha<\frac{1}{b}c_0\left(\frac{r}{b}\right)<\beta\right\}\right|\geq (A_1-A_0)(F(\beta)-F(\alpha))-\epsilon.\tag{186}
\]
From (185) and (186), it follows that
\begin{align*}
\lim_{b\rightarrow+\infty}\frac{1}{\phi(b)}&\left|\left\{ r\::\:  (r,b)=1,\: A_0b\leq r\leq A_1b\ \text{with}\ \alpha<\frac{1}{b}c_0\left(\frac{r}{b}\right)<\beta\right\}\right|\\
&=(A_1-A_0)(F(\beta)-F(\alpha))\\
&=(A_1-A_0)\int_\alpha^\beta\chi(u;I)dF(u),
\end{align*}
which completes the proof of the lemma.
\end{proof}
\begin{definition}(cf. \cite{RU}) Let $X$, $Y$ be normed linear spaces. Let also $\Lambda\!:\!X\rightarrow Y$ be a linear map. Its norm is defined by
$$ \|\Lambda\|=\sup\left\{ \frac{\|\Lambda x\|}{\|x\|}\::\:x\in X,\ x\neq 0 \right\} .$$
If $\|\Lambda\| <+\infty$, then $\Lambda$ is called a bounded linear map\index{map!bounded linear}. We denote by $C_c(\mathbb{R})$ the space of all
continuous functions $f\!:\!\mathbb{R}\rightarrow \mathbb{R}$ with compact support equipped with the sup-norm.
\end{definition}
\begin{lemma}\label{x:pl88}
Let $f\in C_c(\mathbb{R})$. Then, we have
$$\Lambda(f)=(A_1-A_0)\int_{-\infty}^{+\infty}f(u)dF(u).$$
The map $\Lambda\!:\!f\rightarrow\Lambda(f)$ is a bounded linear functional on $C_c(\mathbb{R})$.
\end{lemma}
\begin{proof}
Let $f\in C_c(\mathbb{R})$ with support contained in $[a,b]$. Since $f$ is continuous on $[a,b]$, it is also uniformly continuous on $[a,b]$.\\
Given $\epsilon>0$, there exists $\delta=\delta(\epsilon)>0$ such that
\[
|f(u_1)-f(u_2)|<\epsilon,\ \text{if}\ |u_1-u_2|<\delta\ \text{for}\ u_1,u_2\in[a,b]. \tag{187}
\]
Let $a=\alpha_0<\alpha_1<\cdots<\alpha_l=b$ be a partition of $[a,b]$ with $|\alpha_{i+1}-\alpha_i|<\delta.$\\
Let also
\begin{equation}
\chi_i(u)=\left\{
\begin{array}{l l}
   1\:, & \quad \text{for}\ u\in[\alpha_i,\alpha_{i+1}]\\
    0\:, & \quad \text{otherwise}\:.\\
  \end{array} \right.
 \nonumber
\end{equation}
Define $m(f)$, respectively $M(f)$, by
$$m(f)=\sum_{i=0}^{l}\left(\inf_{\alpha\in[\alpha_i,\alpha_{i+1}]}f(\alpha)\right)\chi_i $$
and
$$M(f)=\sum_{i=0}^{l}\left(\sup_{\alpha\in[\alpha_i,\alpha_{i+1}]}f(\alpha)\right)\chi_i .$$
Due to (187) we obtain
\[
0\leq M(f)-m(f)\leq \epsilon,\ \text{for every}\  \epsilon>0.\tag{188}
\]
Since for $f\geq 0$ we have $\Lambda(f)\geq 0$, it follows that
\[
\Lambda(m(f),b)\leq\Lambda(f,b)\leq \Lambda(M(f),b).\tag{189}
\]
Since $m(f)$ and $M(f)$ are linear combinations of the characteristic functions $\chi_i$, we may apply Lemma \ref{x:pl7} and obtain: 
$$\Lambda(m(f))=(A_1-A_0)\int_a^bm(f)(u)dF(u) $$
and
$$\Lambda(M(f))=(A_1-A_0)\int_a^bM(f)(u)dF(u), $$
because the support of $f$ is contained in $[a,b]$. From (188) and (189), we obtain
$$0\leq \Lambda(M(f))-\Lambda(m(f))\leq(A_1-A_0)\epsilon,\ \text{for every}\  \epsilon>0.$$
Therefore, $\Lambda(f)$ exists as well, and 
$$\Lambda(f)=(A_1-A_0)\int_{-\infty}^{+\infty}f(u)dF(u).$$
\end{proof}
\noindent A generalization of Definition \ref{x:pdd1} is the following:
\begin{definition}\label{x:def7}
Let $X$ be a locally compact Hausdorff space. We set
$$C_0(X)=\left\{f\::\:X\rightarrow\mathbb{R},\:f\in C(X),\:\forall\:\epsilon>0,\:\exists\:\text{a compact set}\:\mathcal{K}\subseteq X,\:\text{such that}\:|f(u)|<\epsilon,\:\forall u\not\in\mathcal{K}   \right\}.$$
\end{definition}
\begin{lemma}\label{x:pl9}(Riesz representation theorem)\index{theorem!Riesz representation}\\
\noindent Let $X$ be a locally compact Hausdorff space\index{space!Hausdorff}, $C_0(X)$ be defined as in Definition \ref{x:def7} with the sup-norm. Let $\Lambda$ be a bounded linear functional on $C_0(X)$. Then there is a unique regular Borel measure\index{measure!Borel regular} $\mu$, such that
$$\Lambda(f)=\int_Xfd\mu,$$
for every $f\in C_0(X)$.
\end{lemma}
\begin{proof}
This is part of Theorem 6.19 of \cite{RU}.
\end{proof}
\begin{lemma}\label{x:pl10}
There is a unique positive measure $\mu$ on $\mathbb{R}$, with the following properties:\\
(a) For $\alpha<\beta\in\mathbb{R}$ we have
$$\mu([\alpha,\beta])=(A_1-A_0)(F(\beta)-F(\alpha)).$$
(b) 
\begin{equation}
\int x^kd\mu=\left\{
\begin{array}{l l}
    (A_1-A_0)H_{k/2}\:, & \quad \text{for even}\: k\\
    0\:, & \quad \text{otherwise}\:.\\
  \end{array} \right.
 \nonumber
\end{equation}
(c) For all $f\in C_0(\mathbb{R})$, we have
$$\lim_{b\rightarrow+\infty}\frac{1}{\phi(b)}\sum_{\substack{r\::\: (r,b)=1\\ A_0b\leq r\leq A_1b}}f\left( \frac{1}{b}c_0\left( \frac{r}{b}\right) \right)=\int f\:d\mu,$$
where $\phi(\cdot)$ denotes the Euler phi-function.
\end{lemma}
\begin{proof}
By Lemma \ref{x:pl88} we know that $\Lambda$ is a positive bounded linear functional on $C_c(\mathbb{R})$. Since $C_c(\mathbb{R})$ is dense in $C_0(\mathbb{R})$, with respect to the supremum norm, the functional $\Lambda$ may be extended in a unique way to $C_0(\mathbb{R})$.\\
\noindent By Lemma \ref{x:pl9}, there is a unique measure $\mu$ on $\mathbb{R}$, with
$$\Lambda(f)=\int_\mathbb{R}fd\mu,$$
for every $f\in C_0(\mathbb{R})$. This proves (c).\\
\noindent Due to Lemma \ref{x:pl7} we have
$$\mu([\alpha,\beta])=(A_1-A_0)(F(\beta)-F(\alpha)).$$
It follows that $\mu$ is positive. This proves (a). \\
\noindent \textit{Proof of (b):} \\
\noindent For every $A\in(0,+\infty)$, we set
\begin{equation}
g_A(u)=\left\{
\begin{array}{l l}
    g(u)\:, & \quad \text{if}\ \  |g(u)|<A\\
    0\:, & \quad \text{otherwise}\:.\\
  \end{array} \right.
 \nonumber
\end{equation}
By the definition of the Lebesgue integral,\index{integral!Lebesgue} for $k\in\mathbb{N}$ we have
$$\int_0^1g_A(u)^kdu=\int_{-A}^Ax^kdF(x).$$
We define $\varphi(x):=x^k$ and
\begin{equation}
\varphi_A(x):=\left\{
\begin{array}{l l}
    x^k\:, & \quad \text{if}\ \ |x|\leq A\\
    0\:, & \quad \text{if}\ \ |x|>A\:.\\
  \end{array} \right.
 \nonumber
\end{equation}
Since the function $\varphi_A(x)$ has compact support, we conclude from (c) that
$$\Lambda(\varphi_A)=\int \varphi_A(x)\:d\mu.$$
We choose a sequence $(A_n)$ of real numbers with $\lim_{n\rightarrow+\infty}A_n=+\infty$.\\ 
By Theorem \ref{x:positive} we know that
$$\int_0^1g(u)^{2k}du$$
exists for every $k\in\mathbb{N}$. By the Cauchy-Schwarz inequality we get
$$\int_0^1|g(u)|^kdu \leq \left(\int_0^1 g(u)^{2k}du\right)^{1/2}$$
and
$$g_{A_n}(u)^k\leq |g(u)|^k.$$ 
 By Lebesgue's dominated convergence theorem (cf. \cite{RU}, Theorem 1.34, p. 27)\index{theorem!Lebesgue's dominated convergence} we have
\begin{align*}
\int x^kd\mu&=\int_{-\infty}^{+\infty}x^kdF(x)=\lim_{n\rightarrow+\infty}\int_0^1g_{A_n}(u)^kdu\\
&=\int_0^1\left( \lim_{n\rightarrow+\infty}g_{A_n}(u)^k \right)du=\int_0^1g(u)^kdu
\end{align*}
and thus by Theorems \ref{x:466} and \ref{x:477}, as well as Definition \ref{x:lkef} for $f(x)=x^k$ we get
\begin{align*}
\lim_{b\rightarrow+\infty}\frac{1}{\phi(b)}\sum_{\substack{r\::\: (r,b)=1\\ A_0b\leq r\leq A_1b}}\left( \frac{1}{b}c_0\left( \frac{r}{b}\right) \right)^k&=\Lambda(f)=\int_0^1g(u)^kdu\\
&=\left\{
\begin{array}{l l}
    (A_1-A_0)H_{k/2}\:, & \quad \text{for even}\: k\\
    0\:, & \quad \text{otherwise}.\\
  \end{array} \right.
 \nonumber
\end{align*}
This proves (b). Therefore, the lemma is proved.
\end{proof}
\noindent \textit{Proof of Theorem \ref{x:pithanotiko}:} The theorem now follows from Lemma \ref{x:pl6} and Lemma \ref{x:pl10}.
\begin{flushright}
\qedsymbol
\end{flushright}
\vspace{10mm}
\subsection{Radius of convergence}
\vspace{5mm}
\begin{theorem}\label{x:radiusconv}
The series
$$\sum_{k\geq 0}H_kx^{2k},$$
where 
$$H_k=\int_0^1\left(\frac{g(x)}{\pi}\right)^{2k}dx$$
with
$$g(x)=\sum_{l=1}^{+\infty}\frac{1-2\{lx\}}{l},$$ 
converges only for $x=0$. 
\end{theorem}
\begin{definition}
For $k\in\mathbb{N}\cup\{0\}$ we set
$$I:=I(k)=\left[ e^{-2k-1},e^{-2k}\right]\ \ \text{and}\ \ l_0:=l_0(k)=e^{2k}.$$
We fix $\delta>0$ sufficiently small and set
$$g_1(\alpha):=\sum_{l\leq l_0^{1-2\delta}}\frac{B(l\alpha)}{l},\ \ g_2(\alpha):=\sum_{l_0^{1-2\delta}<l\leq l_0^{1+2\delta}}\frac{B(l\alpha)}{l},\ \ g_3(\alpha):=\sum_{l> l_0^{1+2\delta}}\frac{B(l\alpha)}{l},$$
where $B(u)=1-2\{u\}$, $u\in\mathbb{R}$.
\end{definition}
In the sequel, we assume $k\geq k_0$ sufficiently large.
\begin{lemma}\label{x:6l1}
We have $$g(\alpha)=g_1(\alpha)+g_2(\alpha)+g_3(\alpha),$$ for every $\alpha\in\mathbb{R}$.
\end{lemma}
\begin{proof}
It is obvious by the definition of $g(\alpha),\:g_1(\alpha),\:g_2(\alpha),\:g_3(\alpha)$.
\end{proof}
\begin{lemma}\label{x:6l2}
For $\alpha\in I$, we have 
$$g_1(\alpha)\geq\frac{k}{2},$$
for $k\in\mathbb{N}\cup\{0\}.$
\end{lemma}
\begin{proof}
For $\alpha\in I$, $l\leq l_0^{1-2\delta}$ we have $l\alpha\leq 1/4$ and therefore 
$$B(l\alpha)\geq \frac{1}{2}.$$
Thus
$$g_1(\alpha)\geq \frac{1}{2}\sum_{l\leq l_0^{1-2\delta}}\frac{1}{l}\geq\frac{k}{2}.$$
\end{proof}
\begin{lemma}\label{x:6l3}
It holds
$$|g_2(\alpha)|\leq 8\delta k,$$
for $k\in\mathbb{N}\cup\{0\}$ and sufficiently small $\delta>0$.
\end{lemma}
\begin{proof}
We have
\begin{align*}
|g_2(\alpha)|\leq\sum_{l_0^{1-2\delta}<l\leq l_0^{1+2\delta}}\frac{1}{l}&\leq 2\left( \log(l_0^{1+2\delta})-\log(l_0^{1-2\delta}) \right)\\
&\leq 8\delta k.
\end{align*}
\end{proof}
\begin{lemma}\label{x:6l4}
For all $\alpha\in I$ that do not belong to an exceptional set $\mathcal{E}$ with measure
$$\text{meas}(\mathcal{E})\leq e^{-2k(1+\delta)},$$
we have
$$|g_3(\alpha)|\leq \frac{1}{8}k.$$
\end{lemma}
\begin{proof}
The function $g_3$ has the Fourier expansion:
$$g_3(\alpha)=\sum_{l>l_0^{1+2\delta}}c(l)e(l\alpha),$$
where $c(l)=O(l^{-1+\epsilon})$ for $\epsilon$ arbitrarily small, by Lemma \ref{x:pl2}.\\
\noindent By Parseval's identity\index{Parseval's identity} we have
$$\int_0^1g_3(\alpha)^2d\alpha=\sum_{l>l_0^{1+2\delta}}c(l)^2=O\left( \sum_{l>l_0^{1+2\delta}}l^{-2+2\epsilon}\right)=O\left(l_0^{-1-3\delta/2} \right).$$
This completes the proof of the Lemma.
\end{proof}
\noindent \textit{Proof of Theorem \ref{x:radiusconv}.} \\
\noindent By Lemmas \ref{x:6l1}, \ref{x:6l3} and \ref{x:6l4}, we have
$$|g(\alpha)|\geq |g_1(\alpha)|- |g_2(\alpha)|- |g_3(\alpha)| \geq\frac{k}{4},$$
for all $\alpha\in I$ except for those $\alpha$ that belong to an exceptional set $\mathcal{E}(I):=\mathcal{E}\cap I\subset I$ with
$$\text{meas}(\mathcal{E}(I))\leq \frac{1}{2}|I|,$$
where $|I|$ stands for the length of $I$. Hence, we obtain
$$H_{k}=\int_0^1\left(\frac{g(\alpha)}{\pi}\right)^{2k}d\alpha\geq \frac{1}{2}|I|\left(\frac{k}{4\pi}  \right)^{2k}\geq e^{k\log k}.$$
Therefore
$$\lim_{k\rightarrow+\infty}H_{k}^{1/k}=+\infty$$
and thus the series
$$\sum_{k\geq 0}H_kx^{2k}$$
converges only for $x=0$. This completes the proof of Theorem \ref{x:radiusconv}.
\begin{flushright}
\qedsymbol
\end{flushright}
\vspace{3mm}
\noindent\textbf{Acknowledgments.} The second author (M. Th. Rassias) expresses his gratitude to his Ph.D. advisor Professor E. Kowalski, who proposed to him this inspiring area of research and for providing constructive guidance throughout the preparation
of this work.


\vspace{10mm}
\begin{thebibliography}{99}%







\bibitem{bag} B. Bagchi, \textit{On Nyman, Beurling and Baez-Duarte's Hilbert space reformulation of the Riemann hypothesis}, Proc. Indian Acad. Sci. Math. 116(2)(2006), 137--146. 

\bibitem{bala} R. Balasubramanian, J. B. Conrey and D. R. Heath-Brown, \textit{Asymptotic mean square of the product of the Riemann zeta-function and a Dirichlet polynomial}, J. Reine Angew. Math. 357(1985), 161--181.


\bibitem{BETT2} S. Bettin, \textit{A generalization of Rademacher's reciprocity law}, Acta Arithmetica, 159(4)(2013), 363--374.
\bibitem{BEC} S. Bettin and B. Conrey, \textit{Period functions and cotangent sums}, Algebra \& Number Theory 7(1)(2013), 215--242.
\bibitem{bre} R. de la Bret\`eche and G. Tenenbaum, \textit{S\'eries trigonom\'etriques \`a coefficients arithm\'etiques}, J.  Anal. Math., 92(2004), 1--79.
\bibitem{bruggeman1} R. W. Bruggeman, \textit{Eisenstein series and the distribution of Dedekind sums}, Math. Z., 202(1989), 181--198.
\bibitem{bruggeman2} R. W. Bruggeman, \textit{Dedekind sums and Fourier coefficients of modular forms}, J. Number Theory, 36(1990), 289--321.



\bibitem{elliott} P. D. T. A. Elliott, \textit{Probabilistic Number Theory I: Mean-Value Theorems}, Springer--Verlag, New York, 1979.
\bibitem{eul} F. Erwe, \textit{Differential-und Integralrechnung II}, Bibl. Inst. Mannheim, 1962.


\bibitem{EST} T. Estermann, \textit{On the representation of a number as the sum of two products}, Proc. London Math. Soc. 31(2)(1930), 123--133.







\bibitem{ISH} M. Ishibashi, \textit{The value of the Estermann zeta function at $s=0$}, Acta Arith. 73(4)(1995), 357--361.
\bibitem{ishi2} M. Ishibashi, \textit{$\mathbb{Q}$-linear relations of special values of the Estermann zeta function,} Acta Arith. 86(3)(1998), 239--244.
\bibitem{IWC} H. Iwaniec, \textit{On the mean values for Dirichlet's polynomials and the Riemann zeta function}, J. London Math. Soc. 22(2)(1980), 39--45.
\bibitem{IKW}\; H. Iwaniec and E. Kowalski, \textit{Analytic Number Theory}, Amer. Math. Soc. Colloq., Providence, RI, Vol. 53, 2004.

\bibitem{kiu} I. Kiuchi, \textit{On an exponential sum involving the arithmetic function $\sigma_{\alpha}(n)$}, Math. J. Okayama Univ., 29(1987), 193--205.
\bibitem{klen} A. Klenke, \textit{Probability Theory}, Springer-Verlag, Berlin, Heidelberg, New York, 2006.








\bibitem{Pra} K. Prachar, \textit{Primzahlverteilung}, Springer--Verlang, Berlin, G\"{o}ttingen, Heidelberg, 1957.
\bibitem{prok} A. V. Prokhorov, \textit{Borel-Cantelli lemma}, in: Encyclopedia of Mathematics (ed. M. Hazewinkel), Springer, New York, 2001.
\bibitem{ras} M. Th. Rassias, \textit{On a cotangent sum related to zeros of the Estermann zeta function}, Applied Mathematics and Computation, 240(2014), 161--167.
\bibitem{rock} A. M. Rockett and P. Sz\"usz, \textit{Continued Fractions}, World-Scientific Press, Singapore, 1992.
\bibitem{RU} W. Rudin, \textit{Real and Complex Analysis}, McGraw-Hill Book Co., New York, 1966.





\bibitem{sle} R. \u Sle\u zevi\u ciene and J. Steuding, \textit{On the zeros of the Estermann zeta-function}, Integral Transforms and Special Functions, 13(2002), 363--371.



\bibitem{vardi} I. Vardi, \textit{Dedekind sums have a limiting distribution}, Internat. Math. Res. Notices, 1(1993), 1--12.
\bibitem{VAS} V. I. Vasyunin, \textit{On a biorthogonal system associated with the
Riemann hypothesis}, (in Russian)  Algebra i Analiz  7(3)(1995), 118--135; English translation in  St. Petersburg Math. J.  7(3)(1996), 405--419.

\bibitem{ZAG} D. Zagier, \textit{Quantum modular forms}, in: Quanta of Maths, Clay Math. Proc. 11, Amer. Math. Soc., Providence, RI, 2010, 659--675.
\end{thebibliography}
\end{document}